\documentclass[12pt]{amsart}
\usepackage{amssymb}
\usepackage{hyperref}
\usepackage{thmtools}
\usepackage{thm-restate}
\usepackage{multirow,makecell} 

\usepackage{geometry}
\usepackage[backend=bibtex]{biblatex}
\usepackage{rotating}
\usepackage{tikz,graphicx}
\tikzstyle{vertex}=[circle, draw, inner sep=0pt, minimum size=4pt]
\newcommand{\vertex}{\node[vertex]}
\tikzstyle{ijvxw}=[circle, draw, inner sep=0pt, fill=white, minimum size=2.5pt]
\newcommand{\ijvxw}{\node[ijvxw]}
\tikzstyle{ijvxb}=[circle, draw, inner sep=0pt, fill=black, minimum size=2.5pt]
\newcommand{\ijvxb}{\node[ijvxb]}
\tikzstyle{pvx}=[circle, draw, inner sep=0pt, fill=orange!20, minimum size=4pt]
\newcommand{\pvx}{\node[pvx]}
\tikzstyle{cvx}=[circle, draw, inner sep=0pt, color=black!25, fill=black!25, minimum size=1.5pt]
\newcommand{\cvx}{\node[cvx]}

\definecolor{Cerulean}{cmyk}{0.94,0.11,0,0}
\definecolor{ForestGreen}{cmyk}{0.91,0,0.88,0.12}
\definecolor{RedViolet}{cmyk}{0.07,0.90,0,0.34}
\definecolor{RawSienna}{cmyk}{0,0.72,1,0.45}

\definecolor{applegreen}{rgb}{0.55,0.71,0.0}

\definecolor{darkblue}{rgb}{0.0, 0.0, 0.7}

\newtheorem{theorem}{Theorem}[section]
\newtheorem{lemma}[theorem]{Lemma}
\newtheorem{proposition}[theorem]{Proposition}
\newtheorem{corollary}[theorem]{Corollary}
\theoremstyle{definition}
\newtheorem{definition}[theorem]{Definition}
\newtheorem{example}[theorem]{Example}

\newtheorem{remark}[theorem]{Remark}

\newcommand\ba{\mathbf{a}}

\newcommand\be{\mathbf{e}}

\newcommand\bv{\mathbf{v}}

\newcommand\by{\mathbf{y}}

\newcommand\bbN{\mathbb{N}}

\newcommand\bbR{\mathbb{R}}

\newcommand\bbZ{\mathbb{Z}}

\newcommand\calA{\mathcal{A}}

\newcommand\calC{\mathcal{C}}
\newcommand\calD{\mathcal{D}}

\newcommand\calF{\mathcal{F}}
\newcommand\calG{\mathcal{G}}

\newcommand\calL{\mathcal{L}}

\newcommand\calO{\mathcal{O}}
\newcommand\calP{\mathcal{P}}
\newcommand\calQ{\mathcal{Q}}
\newcommand\calR{\mathcal{R}}

\newcommand\calT{\mathcal{T}}

\newcommand\horiz{\mathrm{horiz}}
\newcommand\inedge{\mathrm{in}}
\newcommand\outedge{\mathrm{out}}
\newcommand\In{\mathrm{In}}
\newcommand\Out{\mathrm{Out}}
\newcommand\inv{^{-1}}

\DeclareMathOperator{\car}{car}
\DeclareMathOperator{\Cat}{Cat}
\DeclareMathOperator{\Nar}{Nar}
\DeclareMathOperator{\vol}{vol}

\title[A unifying framework]{A unifying framework for the $\nu$-Tamari lattice and principal order ideals in Young's lattice}

\author[von Bell]{Matias von Bell}
\address[von Bell]{Department of Mathematics\\
         University of Kentucky\\
}
\email{matias.vonbell@uky.edu}

\author[Gonz\'alez D'Le\'on]{Rafael S. Gonz\'alez D'Le\'on}
\address[Gonz\'alez D'Le\'on]{Escuela de Ciencias Exactas e Ingenier\'ia\\
         Universidad Sergio Arboleda\\
}
\email{rafael.gonzalezl@usa.edu.co}

\author[Mayorga Cetina]{Francisco A. Mayorga Cetina}
\address[Mayorga Cetina]{Escuela de Ciencias Exactas e Ingenier\'ia\\
         Universidad Sergio Arboleda\\
}
\email{francisco.mayorga@correo.usa.edu.co}

\author[Yip]{Martha Yip}
\address[Yip]{Department of Mathematics\\
         University of Kentucky\\
}
\email{martha.yip@uky.edu}

\begin{document}
\parskip=5pt

\begin{abstract}
We present a unifying framework in which both the $\nu$-Tamari lattice, introduced by Pr\'eville-Ratelle and Viennot, and principal order ideals in Young's lattice indexed by lattice paths $\nu$, are realized as the dual graphs of two combinatorially striking triangulations of a family of flow polytopes which we call the $\nu$-caracol flow polytopes. The first triangulation gives a new geometric realization of the $\nu$-Tamari complex introduced by Ceballos, Padrol and Sarmiento. We use the second triangulation to show that the $h^*$-vector of the $\nu$-caracol flow polytope is given by the $\nu$-Narayana numbers, extending a result of M\'esz\'aros when $\nu$ is a staircase lattice path.
Our work generalizes and unifies results on the dual structure of two subdivisions of a polytope studied by Pitman and Stanley.

\vspace{.5cm}
\noindent{\bf \keywordsname}: flow polytope, triangulation, $\nu$-Dyck path, $\nu$-Tamari lattice, Young's lattice 
\end{abstract}

\maketitle 

\section{Introduction}\label{sec.intro}
Flow polytopes are a family of beautiful mathematical objects. They appear in optimization theory as the feasible sets in maximum flow problems and they also appear in other areas of mathematics including representation theory and algebraic combinatorics.
In the following, $G=(V,E)$ denotes a {\em connected directed graph} with vertex set $V=\{1,2,\ldots, n+1\}$ and edge multiset $E$ with $m$ edges, with $n,m \in \bbN$. 
We assume that any edge $(i,j)\in E$ is directed from $i$ to $j$ whenever $i<j$ and hence $G$ is {\em acyclic}. 
At each vertex $i\in V$ we assign a net flow $a_i \in \bbZ$ satisfying the balance condition  $\sum_{i=1}^{n+1}a_i=0$, and hence $a_{n+1}=-\sum_{i=1}^{n}a_i$. 
For $\ba=(a_1,\ldots,a_n,-\hbox{$\sum_{i=1}^na_i$}) \in \bbZ^{n+1}$, an {\em $\ba$-flow on $G$}  is a tuple $(x_e)_{e\in E} \in \bbR^m_{\geq0}$ such that
$$\sum_{e \in \outedge(j)} x_e - \sum_{e \in \inedge(j)}  x_e = a_j$$
where $\inedge(j)$ and $\outedge(j)$ respectively denote the set of incoming and outgoing edges at $j$, for $j=1,\ldots, n$. 
In what follows, by a graph $G$ we mean a connected directed acyclic graph whose sets $\outedge(1)$, $\inedge(n+1)$, and $\inedge(j)$ and $\outedge(j)$ for $j=2,\dots,n$, are nonempty. 
The {\em flow polytope of $G$ with net flow $\ba$} is the set $\calF_G(\ba)$ of $\ba$-flows on $G$.  
In this article we only consider flow polytopes with \emph{unitary flow} $\ba=\be_1 - \be_{n+1}$, where $\be_i$ for $i=1,\dots,n+1$ denotes the standard basis in $\bbR^{n+1}$, and we will abbreviate the flow polytope of $G$ with unitary flow as $\calF_G$. 
In this case, the only integral points of $\calF_G$ are its vertices, which correspond to the unitary flows along maximal directed paths of $G$ from vertex $1$ to $n+1$. 
Such maximal paths are called {\em routes} (see Figure~\ref{fig.length_clique}).

A {\em $d$-simplex} is the convex hull of $d+1$ points in general position in $\bbR^k$  with $k\ge d$.
A (lattice) {\em  triangulation} of a $d$-polytope $\calP$ is a collection $\calT$ of  $d$-simplices each of whose vertices are in $\calP \cap \bbZ^d$, such that the union of the simplices in $\calT$ is $\calP$, and any pair of simplices intersect in a (possibly empty) common face. 
The normalized volume of a $d$-polytope is $d!$ times its Euclidean volume.  Since the Euclidean volume of a unimodular $d$-simplex in $\bbR^d$ is $\frac{1}{d!}$, then by enumerating the simplices in a unimodular triangulation, one can compute the normalized volume of the polytope.

Baldoni and Vergne~\cite{BV08} gave a set of formulas to determine the normalized volume of $\calF_G(\ba)$, and these are known as the Lidskii formulas. 
M\'esz\'aros and Morales~\cite{MM19} described a triangulation approach due to Postnikov and Stanley, providing an alternative proof of the Lidskii formulas. 
Together with Striker~\cite{MMS19}, showed that a family of Postnikov--Stanley triangulations, which they call framed, are equivalent to the triangulations introduced by Danilov, Karzanov and Koshevoy in \cite{DKK12}, which depend on the notion of a {\em framing} on a graph (see Section \ref{sec.framed_triangulations}).
Different framings of a graph give different regular unimodular triangulations of the associated flow polytope. 

The combinatorial structure of a triangulation  $\calT$ of a polytope is encoded in its {\em dual graph}. This is a graph on the set of simplices in $\calT$ with edges between simplices sharing a common facet. 
In this article we introduce the family of $\nu$-caracol graphs $\car(\nu)$ (see Definition \ref{def.nuCar}).
These graphs are indexed by lattice paths $\nu$ in $\bbZ^2$, and are similar to constructions in~\cite{MM19} and~\cite{Y}.
In Sections \ref{sec.lengthframed} and \ref{sec.planarframed} we discuss two particular framings on $\car(\nu)$ which we call the {\em length} and the {\em planar} framings. 
The triangulations arising from these framings have connections to two lattices on $\nu$-Catalan objects (see Section \ref{sec.nucatalan_numbers}) that appear recurrently in the literature:
\begin{enumerate}
    \item The {\em $\nu$-Tamari lattice} $\mathrm{Tam}(\nu)$ introduced by Pr\'eville-Ratelle and Viennot~\cite{PV17}.
    \item The principal order ideals $I(\nu)$ determined by $\nu$ in Young's lattice $Y$.  
    \end{enumerate}
Thus we find that the collection of framed triangulations on $\calF_{\car(\nu)}$ provides a unifying framework for studying these two ubiquitous lattice structures.
The family of $\nu$-Catalan objects is a generalization of the classical Catalan and rational Catalan families of objects that have been extensively studied in the recent literature (see for example~\cite{ALW16,ARW13,CG19,PV17}).

Our main results are the following:

\begin{restatable}[]{theorem}{volumethm}
\label{thm.volume}
The normalized volume of the flow polytope $\calF_{\car(\nu)}$ is given by the number of $\nu$-Dyck paths, that is, the $\nu$-Catalan number $\Cat(\nu)$.
\end{restatable}
\begin{restatable}[]{theorem}{associahedralthm} The length-framed triangulation of $\mathcal{F}_{\car(\nu)}$ is a regular unimodular triangulation whose dual graph is the Hasse diagram of the $\nu$-Tamari lattice $\mathrm{Tam}(\nu)$. 
\label{thm.associahedralTriangulation}
\end{restatable}
\begin{restatable}[]{theorem}{rootthm}\label{thm.roottriangulation}
 The planar-framed triangulation of $\mathcal{F}_{\car(\nu)}$ is a regular unimodular triangulation whose dual graph is the Hasse diagram of the principal order ideal $I(\nu)$ in Young's lattice $Y$.
\end{restatable}

\begin{restatable}[]{theorem}{hstarthm}\label{thm.hstar}
The $h^*$-polynomial of $\calF_{\car(\nu)}$ is the $\nu$-Narayana polynomial.
\end{restatable}

To describe the combinatorial structure of the length-framed and planar-framed triangulations we use three different $\nu$-Catalan families of objects, as each highlights the combinatorics in crucial and distinct ways. 
These are the $(I,\overline{J})$-trees introduced by Ceballos, Padrol and Sarmiento~\cite{CPS19}, the $\nu$-Dyck paths introduced by Pr\'eville-Ratelle and Viennot~\cite{PV17}, and the $\nu$-trees which were also introduced by Ceballos et al.~\cite{CPS20} (see Sections \ref{subsec:nu-Tamari-IJ}, \ref{subsec:nu-Tamari-nu-Dyck}, and \ref{subsec:nu-Tamari-nu-tree} respectively).
The role that they play in the combinatorics of the triangulations is summarized in Table~\ref{tab.objects} below.
The interested reader can visit \cite{CPS19} and \cite{CPS20} for the correspondences between $(I,\overline{J})$-trees, $\nu$-trees and $\nu$-Dyck paths and \cite{BY} for their generalizations to $(I,\overline{J})$-forests, $\nu$-Schr\"oder trees and $\nu$-Schr\"oder paths.

\def\arraystretch{1.2}
\begin{table}[h!]
\centering
\begin{tabular}{|m{0.17\linewidth}|m{0.155\linewidth}|m{0.12\linewidth}|m{0.285\linewidth}|m{0.13\linewidth}|}
\hline
\centering \textbf{Triangulation} 
    & \centering \textbf{Vertices}               
    & \centering \textbf{Simplices}    
    & \centering \textbf{Adjacency}                                                 
    & \textbf{Dual graph} \\ \hline
 \multirow{3}{*}{Length-framed}         
    & Arcs of $(I,\overline{J})$-trees 
    & $(I,\overline{J})$-trees 
    & Two $(I,\overline{J})$-trees that \textcolor{white}{mm} differ by one arc          
    & \multirow{3}{=}{Hasse diag. of $\mathrm{Tam}(\nu)$} \\ \cline{2-4}
    
    & Lattice points above $\nu$ 
    & $\nu$-trees 
    & Two $\nu$-trees that differ by a rotation
    &\\\cline{2-4}
         
    & (not obtained directly)
    & $\nu$-Dyck paths 
    & Two $\nu$-Dyck paths that differ by a rotation
    &\\ \hline
Planar-framed          
    & Lattice points above $\nu$       
    & $\nu$-Dyck paths      
    & Two $\nu$-Dyck paths that differ by a pair EN to NE
    & Hasse diag. of $I(\nu)$\\ \hline
\end{tabular}
\caption{$\nu$-Catalan objects and their role in the combinatorial structure of the two framed triangulations.}
\label{tab.objects}
\end{table}

We point out that the study of length-framed and planar-framed triangulations of flow polytopes on the $\nu$-caracol graphs can be extended systematically to all graphs. 
On the $\nu$-caracol graphs, the length framing can be viewed as ordering both the incoming and outgoing sets of edges at each inner vertex according to decreasing length, while the planar framing can be viewed as ordering the set of incoming edges at each inner vertex with respect to decreasing length while the set of outgoing edges is ordered with respect to increasing length.  
Particularly for graphs which are symmetric with respect to the vertical axis, then our viewpoint suggests that these two framings are in a sense dual to each other, so perhaps it is not surprising that both framings lead to combinatorially interesting triangulations of the flow polytope. 
In a forthcoming article we study these triangulations in the more general setting.

In the classical case when $\nu=(1,1,\ldots,1)=:(1^n)$, the simultaneous appearance of the Tamari lattice and $I(1^n)$ (also known as the lattice of filters of the type $A$ root poset) in the study of polytopes has been observed before in the work of Pitman and Stanley~\cite{PS02} on subdivisions of a family of polytopes $\Pi_n(\nu)$ (see Remark~\ref{rem.PS}) for the case when $\nu=(1^n)$. 
In particular, the canonical triangulation of the order polytope $\calO(Q_n)$, where $Q_n$ is the product of a $2$-chain and an $n$-chain, was used to show that $\Pi_n(1^n)$ has a subdivision whose dual structure is given by the lattice $I(1^n)$. 
A second subdivision of $\Pi_n(1^n)$ was constructed in \cite{PS02} whose dual structure is the Tamari lattice. 
Via the Cayley trick, M\'esz\'aros and Morales~\cite[Section 7]{MM19} observed that there is an embedding of $\Pi_n(\nu)$ in a polytope that is integrally equivalent to $\calF_{\car(\nu)}$. 
Following their conclusions, our results imply that the two triangulations of $\calF_{\car(\nu)}$ in Theorems \ref{thm.associahedralTriangulation} and \ref{thm.roottriangulation} induce two subdivisions of $\Pi_n(\nu)$ whose dual graphs are the Hasse diagrams of the $\nu$-Tamari lattice and $I(\nu)$ respectively, thereby generalizing the result obtained in \cite{PS02} for the classical case $\nu=(1^n)$. 
Our results also unify the techniques used to exhibit these triangulations since both are obtained directly as framed-triangulations of $\calF_{\car(\nu)}$. 

The results of M\'esz\'aros, Morales and Striker in \cite{MMS19} imply that the polytope $\calF_{\car(\nu)}$ is integrally equivalent to an order polytope $\calO(Q_{\nu})$ and that the planar-framed triangulation of $\calF_{\car(\nu)}$ corresponds to the canonical triangulation of $\calO(Q_{\nu})$ under this equivalence.  This also generalizes the relation between $\Pi_n(\nu)$ and $\calO(Q_{\nu})$ (see Figure \ref{fig.orderpolytope}) in the classical case when $\nu=(1^n)$.

This article is organized as follows. In Section \ref{sec.nucaracolgraph} we introduce the $\nu$-caracol graph $\car(\nu)$ and its associated flow polytope $\calF_{\car(\nu)}$, proving in Theorem \ref{thm.volume} that its normalized volume is given by the $\nu$-Catalan number $\Cat(\nu)$. 
In Section \ref{sec.framed_triangulations} we describe the theory of framed triangulations as presented in~\cite{MMS19}. 
In Section~\ref{sec.lengthframed} we define the length framing of $\car(\nu)$, prove Theorem \ref{thm.associahedralTriangulation}, and as consequence we conclude that the associated triangulation is a geometric realization of the $\nu$-Tamari complex.
In Section~\ref{sec.planarframed} we define the planar framing of $\car(\nu)$ and prove Theorem \ref{thm.roottriangulation}. 
We explain the relationship to order polytopes, showing that our unifying framework generalizes results of Pitman and Stanley in \cite{PS02}.
As an application, in Section \ref{sec.hstar} we use the dual graph of the planar-framed triangulation of $\calF_{\car(\nu)}$ to obtain the $h^*$-polynomial, which proves Theorem \ref{thm.hstar}. 
This result also gives a new proof that the $h$-vector of the $\nu$-Tamari complex is given by the $\nu$-Narayana numbers.

\section{The family of \texorpdfstring{$\nu$-}-caracol flow polytopes}
\label{sec.nucaracolgraph}

In~\cite{BGHHKMY19}, the second and fourth authors studied the flow polytope of the caracol graph, whose normalized volume is the number of Dyck paths from $(0,0)$ to $(n,n)$, a Catalan number.  We now extend this construction. 
\subsection{\texorpdfstring{$\nu$-}-Dyck paths and \texorpdfstring{$\nu$-}-Catalan numbers} \label{sec.nucatalan_numbers}

Let $a,b$ be nonnegative integers, and let $\nu$ be a lattice path from $(0,0)$ to $(b,a)$, consisting of a sequence of $a$ \emph{north} steps $N=(0,1)$ and $b$ \emph{east} steps $E=(1,0)$.
A {\em $\nu$-Dyck path} is a lattice path from $(0,0)$ to $(b,a)$ that lies weakly above $\nu$.

When $a$ and $b$ are coprime positive integers and $\nu$ is the lattice path that borders the
squares which intersect the line $y=\frac{a}{b}x$, this is the special case of the {\em rational $(a,b)$-Dyck path} studied by 
Armstrong, Loehr and Warrington in \cite{ALW16} who showed that the number of rational $(a,b)$-Dyck paths is the {\em $(a,b)$-Catalan number} $\Cat(a,b) = \frac{1}{a+b}\binom{a+b}{a}$. 
See Figure~\ref{fig:threeTamariObjects} for an example of a $(3,5)$-rational Dyck path.
When $(a,b)=(n,n+1)$, this is the case of the classical Catalan number $\Cat(n) = \frac{1}{2n+1}\binom{2n+1}{n} = \frac{1}{n+1}\binom{2n}{n}$.
For general $\nu$, the number $\Cat(\nu)$ of $\nu$-Dyck paths is calculated by a determinantal formula which can be derived by an application of the Gessel--Viennot Lemma \cite{GesselViennot1985}: 
$$\Cat(\nu)=\det \left(\binom{1+\sum_{k=1}^{a-j} \nu_k }{1+j-i} \right)_{1\leq i,j\leq a-1}, 
$$
but no closed-form positive formula is known. 
For more on $\nu$-Dyck paths, see  for example Ceballos and Gonz\'{a}lez D'Le\'{o}n~\cite{CG19}, or Pr\'eville-Ratelle and Viennot~\cite{PV17}.

\subsection{The \texorpdfstring{$\nu$-}-caracol graph}

\begin{definition}\label{def.nuCar}
Let $a,b$ be nonnegative integers, and let $\nu$ be a lattice path from $(0,0)$ to $(b,a)$ where $\nu = NE^{\nu_1}NE^{\nu_2} \cdots NE^{\nu_a}$.
The {\em $\nu$-caracol graph} $\car(\nu)$ is the graph on the vertex set $[a+3]$, together with $\nu_i$ copies of the edge $(1, i+2)$ for $i=1,\ldots, a$, the edges $(i,a+3)$ for $i=2,\ldots, a+1$, and the edges $(i,i+1)$ for $i=1,\ldots,a+2$.
\end{definition}

Note that in this construction, the graph $\car(\nu)$ has $n+1:=a+3$ vertices, and the in-degree $\inedge_i$ of the vertex $i$ in $\car(\nu)$ is $\inedge_2=1$, $\inedge_{i} = \nu_{i-2}+1$ for $i=3,\ldots, n$ and $\inedge_{n+1}=n-1$.
The number of edges $m$ of $\car(\nu)$ is computed by summing the in-degrees of its vertices, so that
$$m= \sum_{i=2}^{n+1} \inedge_i
	= 1 + \sum_{i=1}^{a} (\nu_i+1) + (a+1)
	= 2a+b+2.$$
The (intrinsic) dimension of a flow polytope is given by $\dim\calF_G = |E(G)|-|V(G)|+1$, so we can conclude from this that $\dim\calF_{\car(\nu)} = m-n = a+b$. 

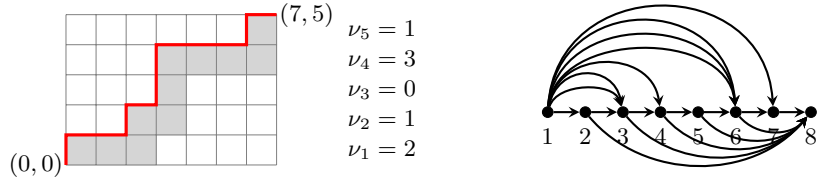
\begin{figure}[ht!]
\centering
\begin{tikzpicture}
\begin{scope}[scale=0.4, xshift=-420, yshift=-50]
	\draw[fill, color=gray!33] (0,0) rectangle (3,1);
	\draw[fill, color=gray!33] (2,1) rectangle (4,2);
	\draw[fill, color=gray!33] (3,2) rectangle (4,3);
	\draw[fill, color=gray!33] (3,3) rectangle (7,4);
	\draw[fill, color=gray!33] (6,4) rectangle (7,5);
	\draw[very thin, color=gray!100] (0,0) grid (7,5);
    \node at (-1,0) {\scriptsize$(0,0)$};
	\node at (8,5) {\scriptsize$(7,5)$};
	
	\node at (10.5,4.5) {\scriptsize$\textcolor{black}{\nu_5=1}$};
	\node at (10.5,3.5) {\scriptsize$\textcolor{black}{\nu_4=3}$};
	\node at (10.5,2.5) {\scriptsize$\textcolor{black}{\nu_3=0}$};
	\node at (10.5,1.5) {\scriptsize$\textcolor{black}{\nu_2=1}$};
	\node at (10.5,0.5) {\scriptsize$\textcolor{black}{\nu_1=2}$};
					
	\draw[very thick, color=red] (0,0)--(0,1)--(2,1)--(2,2)--(3,2)--(3,4)--(6,4)--(6,5)--(6,5)--(7,5);
\end{scope}
\begin{scope}[scale=.5]
	\vertex[fill,label=below:\scriptsize{$1$}](a1) at (1,0) {};
	\vertex[fill,label=below:\scriptsize{$2$}](a2) at (2,0) {};
	\vertex[fill,label=below:\scriptsize{$3$}](a3) at (3,0) {};
	\vertex[fill,label=below:\scriptsize{$4$}](a4) at (4,0) {};
	\vertex[fill,label=below:\scriptsize{$5$}](a5) at (5,0) {};
	\vertex[fill,label=below:\scriptsize{$6$}](a6) at (6,0) {};
	\vertex[fill,label=below:\scriptsize{$7$}](a7) at (7,0) {};
	\vertex[fill,label=below:\scriptsize{$8$}](a8) at (8,0) {};
	
	\draw[-stealth, thick] (a1)--(a2);
	\draw[-stealth, thick] (a2)--(a3);
	\draw[-stealth, thick] (a3)--(a4);
	\draw[-stealth, thick] (a4)--(a5);
	\draw[-stealth, thick] (a5)--(a6);
	\draw[-stealth, thick] (a6)--(a7);
	\draw[-stealth, thick] (a7)--(a8);
	\draw[-stealth, thick] (a1) .. controls (1.25, 1.3) and (2.75, 1.3) .. (a3);	
	\draw[-stealth, thick] (a1) .. controls (1.25, .8) and (2.75, .8) .. (a3);
	\draw[-stealth, thick] (a1) .. controls (1.25, 1.7) and (3.75, 1.7) .. (a4);
	\draw[-stealth, thick] (a1) .. controls (1.25, 2.2) and (5.75, 2.2) .. (a6);	
	\draw[-stealth, thick] (a1) .. controls (1.25, 2.7) and (5.75, 2.7) .. (a6);
	\draw[-stealth, thick] (a1) .. controls (1.25, 3.2) and (5.75, 3.2) .. (a6);
	\draw[-stealth, thick] (a1) .. controls (1.25, 3.7) and (6.75, 3.7) .. (a7);

	\draw[-stealth, thick] (a6) to[out=-50,in=230] (a8);
	\draw[-stealth, thick] (a5) to[out=-50,in=230] (a8);
	\draw[-stealth, thick] (a4) to[out=-50,in=230] (a8);
	\draw[-stealth, thick] (a3) to[out=-50,in=230] (a8);
	\draw[-stealth, thick] (a2) to[out=-50,in=230] (a8);
\end{scope}
\end{tikzpicture}
\caption{A lattice path $\nu = NE^2NE^1NE^0NE^3NE^1$  and its associated $\nu$-caracol graph $\car(\nu)$.
}
\label{fig.numcargraph}
\end{figure}

The flow polytope on the graph $\car(\nu)$ in the special case when $\nu=(1^n)$ has previously been studied by M\'esz\'aros~\cite{M16} and by Benedetti et al.~\cite{BGHHKMY19}.

\begin{remark}\label{rem:extraSteps}
The careful reader will notice that in Definition \ref{def.nuCar} of the graph $\car(\nu)$ we chose to use a lattice path $\nu$ that begins with an $N$ step.  
This choice was made for convenience of the presentation, and is not a true restriction. 
From the results in Section \ref{sec.framed_triangulations} one can verify that the combinatorial structure of a framed triangulation of the flow polytope $\calF_{\car(\nu)}$ is not affected by changing the number of $N$ steps at the beginning of $\nu$ (or by changing the number of $E$ steps at the end of $\nu$). 
Hence without loss of generality and unless otherwise specified, all lattice paths $\nu$ begin with at least one $N$ step.
\end{remark}

Two integral polytopes $\calP\subseteq \bbR^m$ and $\calQ\subseteq\bbR^n$ are {\em integrally equivalent} if there exists an affine transformation $\varphi:\bbR^m\rightarrow\bbR^n$ whose restriction to $\calP$ preserves the lattice.  That is, $\varphi$ is a bijection between $\bbZ^m \cap \mathrm{aff}(P)$ and $\bbZ^n \cap \mathrm{aff}(Q)$.
Integrally equivalent polytopes have the same Ehrhart polynomial, and hence the same volume.

\begin{remark}\label{rem.PS}
M\'esz\'aros and Morales~\cite[Corollary 6.17]{MM19} have previously considered a closely-related variant of the flow polytope $\calF_{\car(\nu)}$, denoted as $\Pi^\star_{a}(\nu)$ in their work. 
The underlying graph of the flow polytope $\Pi^\star_{a}(\nu)$ can be obtained from $\car(\nu)$ by deleting the edge $(2,n+1)$ and contracting the edge $(1,2)$, and a simple transformation reveals that the flow polytopes $\calF_{\car(\nu)}$ and $\Pi^\star_{a}(\nu)$ are integrally equivalent.

They observed that the normalized volume of $\Pi^*_{a}(\nu)$ is the number of lattice points in the Pitman--Stanley polytope $\Pi_{a}(\nu) = \{ \by\in \bbR^{a} \mid \sum_{i=1}^k y_i \leq \sum_{i=1}^k \nu_i\}$, which is equal to the number of $\nu$-Dyck paths.
In the next section, we obtain a direct proof of this result by giving a combinatorial interpretation to the vector partitions enumerated by the Kostant partition function in the generalized Lidskii formula.
This method was first considered in~\cite{BGHHKMY19} and further developed in~\cite{Y}.
\end{remark}

\subsection{The volume of the \texorpdfstring{$\nu$-}-caracol flow polytope}

We begin by defining the Kostant partition function of a graph, and the special case of the Lidskii volume formula which we will use.

For $i=1,\ldots,n$, we call $\alpha_i = \be_i-\be_{i+1}$ the $i$-th {\em simple root}.
For each edge $e=(i,j)$ of a graph $G$, let  $\alpha_e=\be_i-\be_j= \alpha_i+\cdots +\alpha_{j-1} $, and $\Phi_G^+ =\{\alpha_e\mid e\in E(G)\}$ will be called the multiset of {\em positive roots associated to $G$}.

A {\em vector partition} of the vector $\bv$ with respect to $\Phi_G^+$ is a decomposition of $\bv$ into a non-negative linear combination of the positive roots associated to $G$.
The {\em Kostant partition function} of $G$ evaluated at $\bv$, denoted by $K_G(\bv)$, is the number of vector partitions of $\bv$ with respect to $\Phi_G^+$.
Integral $\bv$-flows on $G$ are equivalent to vector partitions of $\bv$, so the number of integral $\bv$-flows on $G$, and hence the number of lattice points in $\calF_G(\bv)$, is $K_G(\bv)$.

Let $G$ be a graph on the vertex set $\{1,\ldots, n+1\}$.  For $i=2,\ldots, n+1$, let $u_i=\inedge_i-1$ be one less than the in-degree of the vertex $i$.
\begin{proposition}[{Baldoni and Vergne~\cite[Theorem 38]{BV08}}] \label{thm.Lidskii}
Let $G$ be a graph with $n+1$ vertices and $m$ edges.  The normalized volume of the flow polytope $\calF_G$ with unitary net flow $\ba=\be_1-\be_{n+1}$ is given by
$$\vol\calF_G = K_G(\bv_{\mathrm{in}}),$$
where $\bv_{\mathrm{in}} = (0, u_2, \ldots, u_{n}, -(m-n-u_{n+1}))$.
\end{proposition}

For flow polytopes of $\nu$-caracol graphs with unitary netflow, the Kostant partition function $K_{\car(\nu)}(\bv_{\mathrm{in}})$ has a simple combinatorial interpretation which we now describe.
This generalizes the construction for the case $\nu= NE^{k-1}NE^k\cdots NE^k$ considered in~\cite[Section 2.4]{Y}. 

\begin{definition}
Let $\nu=NE^{\nu_1}\cdots NE^{\nu_a}$ be a lattice path from $(0,0)$ to $(b,a)$. 
An {\em in-degree gravity diagram} for the flow polytope $\calF_{\car(\nu)}$ consists of a collection of dots and line segments with the following properties:
\begin{enumerate}
\item[(i)] The dots are arranged in columns indexed by the simple roots $\alpha_3,\ldots, \alpha_{a+2}$, with $\nu_1+\cdots+\nu_{j-2}$ dots in the column indexed by $\alpha_j$, and all dots are drawn justified upwards.
\item[(ii)] Horizontal line segments may be drawn between dots in consecutive columns so that each dot is incident to at most one line segment.  A trivial line segment is a singleton dot.  All non-trivial line segments must contain a dot in the column indexed by $\alpha_{a+2}$ (that is, all line segments are justified to the right).  Longer line segments appear above shorter line segments.
\end{enumerate}
We denote the set of all in-degree gravity diagrams by $\calG_{\car(\nu)}(\bv_{\mathrm{in}})$.
See Figure~\ref{fig.embed} for an example of an in-degree gravity diagram.
\end{definition}

The proof of the following Lemma is analogous to the one in~\cite[Theorem 3.1]{BGHHKMY19} for out-degree gravity diagrams. See also~\cite{Y}.
\begin{lemma}\label{lem.ingrav_bijection1}
There is a bijection between the set of vector partitions of $\bv_{\mathrm{in}}$ with respect to $\Phi_{\car(\nu)}^+$ and the set of in-degree gravity diagrams for the flow polytope $\calF_{\car(\nu)}$.
Consequently, $K_{\car(\nu)}(\bv_{\mathrm{in}}) = | \calG_{\car(\nu)}(\bv_{\mathrm{in}})|$.
\end{lemma}

\begin{example}
Let $\nu=NE^2NENNE^3NE$. 
A vector partition of $\bv_{\mathrm{in}} 
= 2\alpha_3+3\alpha_4+3\alpha_5+6\alpha_6+7\alpha_7$ with respect to the positive roots in $\Phi_{\car(\nu)}^+$ is
$$\bv_{\mathrm{in}} = \alpha_{(3,8)}
	+ \alpha_{(5,8)} 
	+ 2\alpha_{(6,8)}
	+\alpha_3+2\alpha_4+\alpha_5+2\alpha_6+3\alpha_7.$$
This vector partition is represented by the gravity diagram on the left of Figure~\ref{fig.embed}.
\end{example} 

We are ready to make a connection from in-degree gravity diagrams to $\nu$-Dyck paths. 

\begin{lemma}\label{lem.ingrav_bijection2}
There is a bijection between the set $\calG_{\car(\nu)}(\bv_{\mathrm{in}})$ of in-degree gravity diagrams for the flow polytope $\calF_{\car(\nu)}$ and the set $\calD_\nu$ of $\nu$-Dyck paths.
\end{lemma}
\begin{proof}
First recall that a $\nu$-Dyck path is a lattice path in the rectangular grid from $(0,0)$ to $(b,a)$ that lies weakly above the path $\nu$.
Also recall that in an in-degree gravity diagram for $\car(\nu)$, the column indexed by $\alpha_k$ has $\nu_1+\cdots+\nu_k$  dots, for $k=3,\ldots, a+2$.  
This is precisely the number of squares in the row between the lines $x=0$, $y=k-1$, $y=k$, and above $\nu$.

Therefore, given an in-degree gravity diagram $\Gamma\in \calG_{\car(\nu)}(\bv_{\mathrm{in}})$, we may rotate it $90$ degrees counterclockwise and embed the array of dots into the squares of $\bbZ^2$ so that the dots in the column indexed by $\alpha_{a+2}$ lie in the row of squares just above the line $y=a$, and the dots in the first row of $\Gamma$ lie in the column of squares just right of the line $x=0$.  
By the previous observation, we see that the dots of $\Gamma$ occupy every square in $\bbZ^2$ between the lines $x=0$, $x=b$ and $y=a+1$, and which lie above the path $\nu$.
See Figure~\ref{fig.embed} for an illustration.

Line segments of the rotated embedded gravity diagram $\Gamma$ are now vertical, and they extend down from just above the top row of the rectangular grid.  
The lengths of these vertical line segments are weakly decreasing from left to right, so the line segments of $\Gamma$ define a unique $\nu$-Dyck path that separates the dots in $\Gamma$ which are incident to a line segment in $\Gamma$, from the dots which are not incident to any (proper) line segment in $\Gamma$.
This construction defines a map $\Xi: \calG_{\car(\nu)}(\bv_{\mathrm{in}}) \rightarrow \calD_\nu$.

Conversely, any $\nu$-Dyck path defines an in-degree gravity diagram $\Gamma$ for $\calF_{\car(\nu)}$, where every dot of $\Gamma$ that occupies a square that is above the $\nu$-Dyck path is incident to a line segment of $\Gamma$, and every dot of $\Gamma$ that occupies a square that is below the $\nu$-Dyck path is not incident to any (proper) line segment of $\Gamma$.
Therefore, $\Xi$ is a bijection.
\end{proof}

\begin{figure}[ht!]
\begin{center}
\begin{tikzpicture}
\begin{scope}[scale=0.4, xshift=0, yshift=0]	
	\vertex[fill, minimum size=3pt] at (4,6) {};
	\vertex[fill, minimum size=3pt] at (4,7) {}; 
	\vertex[fill, minimum size=3pt] at (5,5) {};
	\vertex[fill, minimum size=3pt] at (5,6) {};
	\vertex[fill, minimum size=3pt] at (5,7) {};	
	\vertex[fill, minimum size=3pt] at (6,5) {}; 
	\vertex[fill, minimum size=3pt] at (6,6) {};
	\vertex[fill, minimum size=3pt] at (6,7) {};	
	\vertex[fill, minimum size=3pt] at (7,2) {};
	\vertex[fill, minimum size=3pt] at (7,3) {};
	\vertex[fill, minimum size=3pt] at (7,4) {};
	\vertex[fill, minimum size=3pt] at (7,5) {}; 
	\vertex[fill, minimum size=3pt] at (7,6) {};
	\vertex[fill, minimum size=3pt] at (7,7) {};
	\vertex[fill, minimum size=3pt] at (8,1) {};	
	\vertex[fill, minimum size=3pt] at (8,2) {};
	\vertex[fill, minimum size=3pt] at (8,3) {};
	\vertex[fill, minimum size=3pt] at (8,4) {};
	\vertex[fill, minimum size=3pt] at (8,5) {}; 
	\vertex[fill, minimum size=3pt] at (8,6) {};
	\vertex[fill, minimum size=3pt] at (8,7) {};		
	\draw (4,7)--(8,7);
	\draw (6,6)--(8,6);
	\draw (7,5)--(8,5);
	\draw (7,4)--(8,4);	
	\node at (4, 8) {\tiny$\alpha_3$};	
	\node at (5, 8) {\tiny$\alpha_4$};		
	\node at (6, 8) {\tiny$\alpha_5$};
	\node at (7, 8) {\tiny$\alpha_6$};
	\node at (8, 8) {\tiny$\alpha_7$};			
\end{scope}
\begin{scope}[scale=0.45, xshift=420, yshift=30]
	\draw[fill, color=gray!33] (0,0) rectangle (3,1);
	\draw[fill, color=gray!33] (2,1) rectangle (4,2);
	\draw[fill, color=gray!33] (3,2) rectangle (4,3);
	\draw[fill, color=gray!33] (3,3) rectangle (7,4);
	\draw[fill, color=gray!33] (6,4) rectangle (7,5);
	\draw[very thin, color=gray!100] (0,0) grid (7,5);
		
	\node at (-1,0) {\scriptsize$(0,0)$};
	\node at (8,5) {\scriptsize$(7,5)$};
	
	\node at (10.5,4.5) {\scriptsize$\textcolor{black}{\nu_5=1}$};
	\node at (10.5,3.5) {\scriptsize$\textcolor{black}{\nu_4=3}$};
	\node at (10.5,2.5) {\scriptsize$\textcolor{black}{\nu_3=0}$};
	\node at (10.5,1.5) {\scriptsize$\textcolor{black}{\nu_2=1}$};
	\node at (10.5,0.5) {\scriptsize$\textcolor{black}{\nu_1=2}$};
	
	\vertex[fill, minimum size=3pt] at (0.5, 5.5) {}; 
	\vertex[fill, minimum size=3pt] at (1.5, 5.5) {};
	\vertex[fill, minimum size=3pt] at (2.5, 5.5) {}; 
	\vertex[fill, minimum size=3pt] at (3.5, 5.5) {}; 
	\vertex[fill, minimum size=3pt] at (4.5, 5.5) {};
	\vertex[fill, minimum size=3pt] at (5.5, 5.5) {};
	\vertex[fill, minimum size=3pt] at (6.5, 5.5) {};	
	\vertex[fill, minimum size=3pt] at (0.5, 4.5) {}; 
	\vertex[fill, minimum size=3pt] at (1.5, 4.5) {};
	\vertex[fill, minimum size=3pt] at (2.5, 4.5) {}; 
	\vertex[fill, minimum size=3pt] at (3.5, 4.5) {};
	\vertex[fill, minimum size=3pt] at (4.5, 4.5) {};	
	\vertex[fill, minimum size=3pt] at (5.5, 4.5) {};
	\vertex[fill, minimum size=3pt] at (0.5, 3.5) {}; 
	\vertex[fill, minimum size=3pt] at (1.5, 3.5) {};
	\vertex[fill, minimum size=3pt] at (2.5, 3.5) {};
	\vertex[fill, minimum size=3pt] at (0.5, 2.5) {}; 
	\vertex[fill, minimum size=3pt] at (1.5, 2.5) {};
	\vertex[fill, minimum size=3pt] at (2.5, 2.5) {};	
	\vertex[fill, minimum size=3pt] at (0.5, 1.5) {};
	\vertex[fill, minimum size=3pt] at (1.5, 1.5) {};

	\draw (0.5,1.5)--(0.5,5.5);
	\draw (1.5,3.5)--(1.5,5.5);
	\draw (2.5,4.5)--(2.5,5.5);
	\draw (3.5,4.5)--(3.5,5.5);			
	\node at (-0.5, 1.5) {\begin{turn}{90}\tiny$\alpha_3$\end{turn}};	
	\node at (-0.5, 2.5) {\begin{turn}{90}\tiny$\alpha_4$\end{turn}};	
	\node at (-0.5, 3.5) {\begin{turn}{90}\tiny$\alpha_5$\end{turn}};	
	\node at (-0.5, 4.5) {\begin{turn}{90}\tiny$\alpha_6$\end{turn}};	
	\node at (-0.5, 5.5) {\begin{turn}{90}\tiny$\alpha_7$\end{turn}};	
					
	\draw[very thick, color=red] (0,0)--(0,1)--(1,1)--(1,3)--(2,3)--(2,4)--(4,4)--(4,5)--(7,5);
\end{scope}
\end{tikzpicture}
\end{center}
\caption{A gravity diagram (left) representing a vector partition of $\bv_{\mathrm{in}}$ associated to $\car(\nu)$ for $\nu=NE^2NENNE^3NE$.
The bijection $\Xi$ of Theorem~\ref{thm.volume} sends the gravity diagram to the $\nu$-Dyck path via a $90$ degree rotation (right).
}
\label{fig.embed}
\end{figure}
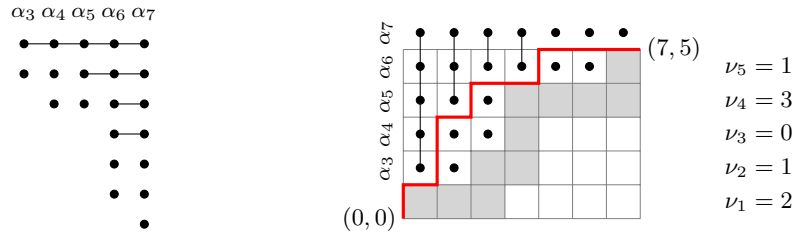

\volumethm*
\begin{proof}
Combining Proposition~\ref{thm.Lidskii} and Lemmas~\ref{lem.ingrav_bijection1} and~\ref{lem.ingrav_bijection2}, the normalized volume of $\calF_{\car(\nu)}$ is
$$\vol \calF_{\car(\nu)}  = K_{\car(\nu)}(\bv_{\mathrm{in}}) = | \calG_{\car(\nu)}(\bv_{\mathrm{in}})| = |\calD_\nu| = \Cat(\nu).$$
\end{proof}

In Sections~\ref{sec.lengthframed} and~\ref{sec.planarframed}, we construct two regular unimodular triangulations for the flow polytope $\calF_{\car(\nu)}$ with combinatorially interesting dual graph structures, giving two more proofs that the normalized volume of $\calF_{\car(\nu)}$ is the number of $\nu$-Dyck paths.

\section{Framed triangulations of a flow polytope}\label{sec.framed_triangulations}


We now describe the family of triangulations defined by Danilov, Karzanov, and Koshevoy~\cite{DKK12}, interpreted as special cases of the Postnikov--Stanley triangulations described by  M\'esz\'aros, Morales and Striker in~\cite{MMS19}.

We call {\em inner vertices} the vertices $\{2,\ldots, n\}$ of a graph $G$ on $n+1$ vertices. A {\em framing} at the inner vertex $i$ is a pair of linear orders $(\prec_{\inedge(i)}, \prec_{\outedge(i)})$ on the incoming and outgoing edges at $i$.
A {\em framed graph}, denoted $(G,\prec)$,  is a graph with a framing at every inner vertex.  
In Sections~\ref{sec.lengthframed} and~\ref{sec.planarframed}, we will consider two specific framings of the caracol graphs $\car(\nu)$, which lead to combinatorially interesting triangulations of $\calF_{\car(\nu)}$. An example of these two framings is given in Figure~\ref{fig.two_framings_ex}. 

\begin{figure}[ht!]
\centering
\begin{tikzpicture}
\begin{scope}[scale=.6]
	\vertex[fill,label=below:\scriptsize{$1$}](a1) at (1,0) {};
	\vertex[fill,label=below:\scriptsize{$2$}](a2) at (2,0) {};
	\vertex[fill,label=below:\scriptsize{$3$}](a3) at (3,0) {};
	\vertex[fill,label=below:\scriptsize{$4$}](a4) at (4,0) {};
	\vertex[fill,label=below:\scriptsize{$5$}](a5) at (5,0) {};
	\vertex[fill,label=below:\scriptsize{$6$}](a6) at (6,0) {};
    \vertex[fill,label=below:\scriptsize{$7$}](a7) at (7,0) {};
    \vertex[fill,label=below:\scriptsize{$8$}](a8) at (8,0) {};
	
	\draw[-stealth, thick] (a1)--(a2);
	\draw[-stealth, thick] (a2)--(a3);
	\draw[-stealth, thick] (a3)--(a4);
	\draw[-stealth, thick] (a4)--(a5);
	\draw[-stealth, thick] (a5)--(a6);
	\draw[-stealth, thick] (a6)--(a7);
	\draw[-stealth, thick] (a7)--(a8);
	\draw[-stealth, thick] (a1) .. controls (1.25, 1.3) and (2.75, 1.3) .. (a3);	
	\draw[-stealth, thick] (a1) .. controls (1.25, .8) and (2.75, .8) .. (a3);
	\draw[-stealth, thick] (a1) .. controls (1.25, 1.7) and (3.75, 1.7) .. (a4);
	\draw[-stealth, thick] (a1) .. controls (1.25, 2.2) and (5.75, 2.2) .. (a6);	
	\draw[-stealth, thick] (a1) .. controls (1.25, 2.7) and (5.75, 2.7) .. (a6);
	\draw[-stealth, thick] (a1) .. controls (1.25, 3.2) and (5.75, 3.2) .. (a6);
	\draw[-stealth, thick] (a1) .. controls (1.25, 3.7) and (6.75, 3.7) .. (a7);

	\draw[-stealth, thick] (a6) to[out=50,in=130] (a8);
	\draw[-stealth, thick] (a5) to[out=50,in=130] (a8);
	\draw[-stealth, thick] (a4) to[out=50,in=130] (a8);
	\draw[-stealth, thick] (a3) to[out=50,in=130] (a8);
	\draw[-stealth, thick] (a2) to[out=50,in=130] (a8);

	\node at (5.5,0.1) {\textcolor{red}{\scriptsize$4$}};
	\node at (4.35,2.45) {\textcolor{red}{\scriptsize$1$}};
	\node at (4.35,2.05) {\textcolor{red}{\scriptsize$2$}};
	\node at (4.35,1.65) {\textcolor{red}{\scriptsize$3$}};	

	\node at (6.25,.4) {\textcolor{red}{\scriptsize$1$}};
	\node at (6.4,0) {\textcolor{red}{\scriptsize$2$}};	
\end{scope}
\begin{scope}[xshift=200, scale=.6]
	\vertex[fill,label=below:\scriptsize{$1$}](a1) at (1,0) {};
	\vertex[fill,label=below:\scriptsize{$2$}](a2) at (2,0) {};
	\vertex[fill,label=below:\scriptsize{$3$}](a3) at (3,0) {};
	\vertex[fill,label=below:\scriptsize{$4$}](a4) at (4,0) {};
	\vertex[fill,label=below:\scriptsize{$5$}](a5) at (5,0) {};
	\vertex[fill,label=below:\scriptsize{$6$}](a6) at (6,0) {};
	\vertex[fill,label=below:\scriptsize{$7$}](a7) at (7,0) {};
	\vertex[fill,label=below:\scriptsize{$8$}](a8) at (8,0) {};

	\draw[-stealth, thick] (a1)--(a2);
	\draw[-stealth, thick] (a2)--(a3);
	\draw[-stealth, thick] (a3)--(a4);
	\draw[-stealth, thick] (a4)--(a5);
	\draw[-stealth, thick] (a5)--(a6);
	\draw[-stealth, thick] (a6)--(a7);
	\draw[-stealth, thick] (a7)--(a8);
	\draw[-stealth, thick] (a1) .. controls (1.25, 1.3) and (2.75, 1.3) .. (a3);	
	\draw[-stealth, thick] (a1) .. controls (1.25, .8) and (2.75, .8) .. (a3);
	\draw[-stealth, thick] (a1) .. controls (1.25, 1.7) and (3.75, 1.7) .. (a4);
	\draw[-stealth, thick] (a1) .. controls (1.25, 2.2) and (5.75, 2.2) .. (a6);	
	\draw[-stealth, thick] (a1) .. controls (1.25, 2.7) and (5.75, 2.7) .. (a6);
	\draw[-stealth, thick] (a1) .. controls (1.25, 3.2) and (5.75, 3.2) .. (a6);
	\draw[-stealth, thick] (a1) .. controls (1.25, 3.7) and (6.75, 3.7) .. (a7);
	\draw[-stealth, thick] (a6) to[out=-50,in=230] (a8);
	\draw[-stealth, thick] (a5) to[out=-50,in=230] (a8);
	\draw[-stealth, thick] (a4) to[out=-50,in=230] (a8);
	\draw[-stealth, thick] (a3) to[out=-50,in=230] (a8);
	\draw[-stealth, thick] (a2) to[out=-50,in=230] (a8);

	\node at (5.5,0.1) {\textcolor{red}{\scriptsize$4$}};
	\node at (4.35,2.45) {\textcolor{red}{\scriptsize$1$}};
	\node at (4.35,2.05) {\textcolor{red}{\scriptsize$2$}};
	\node at (4.35,1.65) {\textcolor{red}{\scriptsize$3$}};	

	\node at (6.3,.1) {\textcolor{red}{\scriptsize$1$}};
	\node at (6.5,-.25) {\textcolor{red}{\scriptsize$2$}};		
\end{scope}
\end{tikzpicture}
\caption{Length (left) and planar (right) framings at the vertex $6$ of $G=\car(\nu)$.} 
\label{fig.two_framings_ex}
\end{figure}
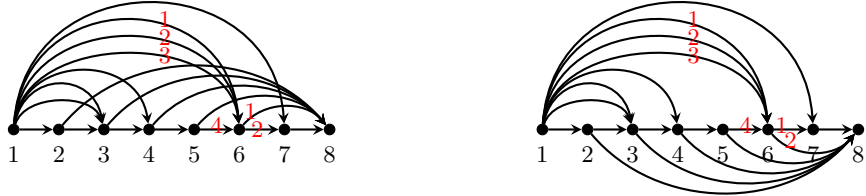

For an inner vertex $i$ of a graph $G$, let $\In(i)$ and $\Out(i)$ respectively denote the set of maximal paths ending at $i$ and the set of maximal paths beginning at $i$.
For a route $R$ containing an inner vertex $i$, let $Ri$ (respectively $iR$) denote the maximal subpath of $R$ ending (respectively beginning) at $i$.
Define linear orders $\prec_{\In(i)}$ and $\prec_{\Out(i)}$ on $\In(i)$ and $\Out(i)$ as follows.
Given paths $R, Q\in\In(i)$, let $j\leq i$ be the smallest vertex after which $Ri$ and $Qi$ coincide. 
Let $e_R$ be the edge of $R$ entering $j$ and let $e_Q$ be the edge of $Q$ entering $j$.
Then $R \prec_{\In(i)} Q$ if and only if $e_R \prec_{\inedge(j)} e_Q$.
Similarly for $R, Q\in\Out(i)$, let $j\geq i$ be the largest vertex before which $iR$ and $iQ$ coincide.
Then $R \prec_{\Out(i)} Q$ if and only if $e_R \prec_{\outedge(j)} e_Q$.

Two routes $R$ and $Q$ containing an inner vertex $i$ are {\em coherent at $i$} if $Ri$ and $Qi$ are ordered the same as $iR$ and $iQ$.
Routes $R$ and $Q$ are {\em coherent} if they are coherent at each common inner vertex.
A set of mutually coherent routes is a {\em clique}.
For a maximal clique $C$, let $\Delta_C$ denote the convex hull of the vertices of $\calF_G$ corresponding to the unitary flows along the routes in $C$.
\begin{proposition}[{Danilov et al.~\cite{DKK12}}]
Let $(G,\prec)$ be a framed graph.  Then 
$$\{\Delta_C \mid C \hbox{ is a maximal clique of $(G,\prec)$}\}$$
is the set of the top dimensional simplices in a regular unimodular triangulation of $\calF_G$.
\end{proposition}

\section{The length-framed triangulation} 
\label{sec.lengthframed}

The goal of this section is to show that the flow polytope $\calF_{\car(\nu)}$ has a regular unimodular triangulation whose dual graph structure is given by the Hasse diagram of the $\nu$-Tamari lattice. 
The triangulation in question arises as a DKK triangulation with the {\em length-framing}. We show that this length-framed triangulation is combinatorially equivalent to the $\nu$-Tamari complex introduced by Ceballos, Padrol, and Sarmiento in \cite{CPS19}.

\subsection{The \texorpdfstring{$\nu$-}-Tamari lattice}
\label{subsec:nu-Tamari}
The $\nu$-Tamari lattice $\mathrm{Tam}(\nu)$ was introduced by Pr\'eville-Ratelle and Viennot~\cite[Theorem 1]{PV17} as a partial order on the set of $\nu$-Dyck paths. 
Using an alternative description with $(I,\overline{J})$-trees,
Ceballos, Padrol and Sarmiento~\cite{CPS19} 
realized the $\nu$-Tamari lattice as the one-skeleton of a polyhedral complex known as the $\nu$-associahedron $K_\nu$, which generalizes the classical associahedron. 
In \cite{CPS20}, they gave further descriptions of the $\nu$-Tamari lattice using $\nu$-trees and $\nu$-bracket vectors, proving a special case of Rubey's lattice conjecture. 
In \cite{BY} the first and fourth authors generalize $\nu$-Dyck paths and $\nu$-trees to $\nu$-Schr\"oder paths and $\nu$-Schr\"oder trees in their study of the face poset of $K_\nu$. 

In this article we use three descriptions of $\mathrm{Tam}(\nu)$, as each provides a useful viewpoint. 
The $(I,\overline{J})$-tree description shows that the length-framed triangulation is combinatorially the $(I,\overline{J})$-Tamari complex, which is constructed using $(I,\overline{J})$-trees. 
The $\nu$-Dyck path perspective makes it clear when the length-framed and planar-framed triangulations coincide (see Proposition~\ref{prop.whenTheTriangulationsAreEquivalent}). 
Finally, the $\nu$-tree description captures the combinatorial structure of the triangulation succinctly, with $\nu$-trees playing an analogous role to $\nu$-Dyck paths in the planar-framed triangulation (compare Figures~\ref{fig.length_clique_and_nu_tree} and~\ref{fig.planar_clique}).

\begin{figure}[ht!]
\centering
\begin{tikzpicture}
\begin{scope}[xshift=-100, yshift=-10, scale=0.5]
    \draw[fill, color=gray!33] (0,0) rectangle (1,1); 
    \draw[fill, color=gray!33] (1,1) rectangle (3,2);
    \draw[fill, color=gray!33] (3,2) rectangle (5,3);
    \draw[fill, color=gray!33] (1,0) rectangle (2,1);
    \draw[fill, color=gray!33] (3,1) rectangle (4,2);  
    \draw[very thin, color=gray!100] (0,0) grid (5,3);  
   	\draw[dashed, thick, color=gray!10] (0,0)--(5,3); 
	
    \draw[very thick, red] (0,0) -- (0,2) -- (2,2) -- (2,3) -- (5,3);     
\end{scope}

\begin{scope}[xshift = 0, yshift = 0, scale=0.6]
	\node[style={circle,draw, inner sep=1pt, fill=black}] (1)  at (1,0)  {{\tiny \textcolor{white}{1}}};
	\node[style={circle,draw, inner sep=1pt, fill=black}] (3)  at (3,0)  {{\tiny \textcolor{white}{3}}};
	\node[style={circle,draw, inner sep=1pt, fill=black}] (5)  at (5,0)  {{\tiny \textcolor{white}{5}}};
	\node[style={circle,draw, inner sep=1pt, fill=black}] (6)  at (6,0)  {{\tiny \textcolor{white}{6}}};
	\node[style={circle,draw, inner sep=1pt, fill=black}] (8)  at (8,0)  {{\tiny \textcolor{white}{8}}};
	\node[style={circle,draw, inner sep=1pt, fill=black}] (9)  at (9,0)  {{\tiny \textcolor{white}{9}}};
	
	\node[style={circle,draw, inner sep=1pt, fill=none}] (2)  at (2,0)  {\tiny{$\overline{2}$}};
	\node[style={circle,draw, inner sep=1pt, fill=none}] (4)  at (4,0)  {{\tiny $\overline{4}$}};
	\node[style={circle,draw, inner sep=1pt, fill=none}] (7)  at (7,0)  {{\tiny $\overline{7}$}};
	\node[style={circle,draw, inner sep=0pt, fill=none}] (10)  at (10,0)  {{\tiny $\overline{10}$}};	
	
	\node[] (E0)  at (1,-1)  {\textcolor{black}{{\footnotesize $E_0$}}};
	\node[] (N0)  at (2,-1)  {\textcolor{blue}{{\footnotesize $N_0$}}};
	\node[] (E1)  at (3,-1)  {\textcolor{blue}{{\footnotesize $E_1$}}};
    \node[] (N1)  at (4,-1)  {\textcolor{blue}{{\footnotesize $N_1$}}};	
	\node[] (E2)  at (5,-1)  {\textcolor{blue}{{\footnotesize $E_2$}}};
	\node[] (E3)  at (6,-1)  {\textcolor{blue}{{\footnotesize $E_3$}}};	
	\node[] (N2)  at (7,-1)  {\textcolor{blue}{{\footnotesize $N_2$}}};	
	\node[] (E4)  at (8,-1)  {\textcolor{blue}{{\footnotesize $E_4$}}};
	\node[] (E5)  at (9,-1)  {\textcolor{blue}{{\footnotesize $E_5$}}};	
	\node[] (N3)  at (10,-1)  {\textcolor{black}{{\footnotesize $N_3$}}};;	
	
	\draw[] (1) to [bend left=60] (2);
    \draw[] (1) to [bend left=70] (10);

    \draw[] (3) to [bend left=60] (4);
    \draw[] (3) to [bend left=65] (10);

    \draw[] (3) to [bend left=60] (7);    
    \draw[] (5) to [bend left=60] (7);
    \draw[] (6) to [bend left=60] (7);
    
    \draw[] (8) to [bend left=60] (10);
    \draw[] (9) to [bend left=60] (10);
\end{scope}

\begin{scope}[xshift=220, yshift=-10, scale=0.5]
    \draw[fill, color=gray!33] (0,0) rectangle (1,1); 
    \draw[fill, color=gray!33] (1,1) rectangle (3,2);
    \draw[fill, color=gray!33] (3,2) rectangle (5,3);
    \draw[fill, color=gray!33] (1,0) rectangle (2,1);
    \draw[fill, color=gray!33] (3,1) rectangle (4,2);  
    \draw[very thin, color=gray!100] (0,0) grid (5,3);  
		
	\draw[very thick, color=RawSienna] (0,0)--(0,3)--(5,3);
	\draw[very thick, color=RawSienna] (1,3)--(1,1);
	\draw[very thick, color=RawSienna] (2,2)--(3,2);
    \draw[very thick, color=RawSienna] (2,2)--(1,2);

	\node[circle,color=black,fill=black, inner sep=1.5pt] (r) at (0,3) {};
	\node[circle,color=black,fill=black, inner sep=1.5pt] (r) at (2,2) {};
	\node[circle,color=black,fill=black, inner sep=1.5pt] (r) at (1,3) {};
	\node[circle,color=black,fill=black, inner sep=1.5pt] (r) at (0,0) {};
	\node[circle,color=black,fill=black, inner sep=1.5pt] (r) at (1,1) {};	
	\node[circle,color=black,fill=black, inner sep=1.5pt] (r) at (1,2) {};	
	\node[circle,color=black,fill=black, inner sep=1.5pt] (r) at (3,2) {};
	\node[circle,color=black,fill=black, inner sep=1.5pt] (r) at (4,3) {};	
	\node[circle,color=black,fill=black, inner sep=1.5pt] (r) at (5,3) {};	  
\end{scope}
\end{tikzpicture}
\caption{Three corresponding $\nu$-Catalan objects, with $\nu=NENE^2NE^2$. A $\nu$-Dyck path (left). An $(I,\overline{J})$-tree, with $I=\{1,3,5,6,8,9\}$ and $\overline{J} = \{\overline{2},\overline{4},\overline{7},\overline{10}\}$ (center). The path $\nu$ can be read from the blue labels under the $(I,\overline{J})$-tree. A $\nu$-tree (right), which is a grid representation of the  $(I,\overline{J})$-tree in the center.} 
\label{fig:threeTamariObjects}
\end{figure}
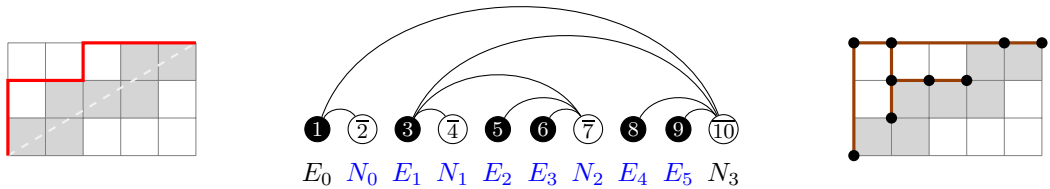

\subsubsection{The \texorpdfstring{$\nu$-}-Tamari lattice as the rotation lattice on \texorpdfstring{$\nu$-}-Dyck paths}
\label{subsec:nu-Tamari-nu-Dyck}

We first give the description of $\mathrm{Tam}(\nu)$ in terms of $\nu$-Dyck paths. 
A {\em valley} of a lattice path is a point $p$ at the end of an east step that is immediately followed by a north step. 
Let $\mu$ be a $\nu$-Dyck path. For any lattice point $p$ on $\mu$, let {\em $\horiz_\nu(p)$} denote the maximum number of east steps that can be added to the right of $p$ without crossing $\nu$. 
For example, $\horiz_\nu(p)$ of the lattice points on the $\nu$-Dyck path in Figure~\ref{fig:threeTamariObjects} are $0,1,3,2,1,3,2,1,0$ as it is traversed from $(0,0)$ to $(5,3)$. 
The set of $\nu$-Dyck paths can then be endowed with the structure of a poset with the covering relation $\lessdot_\nu$ defined as follows. If $p$ is a valley of $\mu$, let $q$ be the first lattice point in $\mu$ after $p$ with $\horiz_\nu(p)= \horiz_\nu(q)$, and let $\mu_{[p,q]}$ denote the subpath of $\mu$ between $p$ and $q$. 
Define a {\em rotation on $\mu$ at $p$} by switching the east step preceding $p$ with the subpath $\mu_{[p,q]}$.
If $\mu'$ is the lattice path obtained by rotating $\mu$ at $p$, then $\mu \lessdot_\nu \mu'$ is a covering relation in $\mathrm{Tam}(\nu)$. 
Let $<_\nu$ denote the transitive closure of the relation $\lessdot_\nu$.

\begin{definition}[Pr\'eville-Ratelle -- Viennot~\cite{PV17}]
The {\em $\nu$-Tamari lattice} $\mathrm{Tam}(\nu)$ is a lattice on the set of $\nu$-Dyck paths induced by the relation $<_\nu$. 
\end{definition}

The leftmost lattice in Figure~\ref{fig.35TamariLattice} gives an example of the $\nu$-Tamari lattice indexed by $\nu$-Dyck paths.

\begin{figure}[ht!]
\centering
\begin{tikzpicture}[scale=.75]
\tikzstyle{vertex}=[circle, fill=ForestGreen, inner sep=0pt, minimum size=5pt]


\begin{scope}[scale=.33, xshift=-650, yshift=0]
	\vertex[color=ForestGreen] (a) at (0,0)  {};
	\vertex[color=ForestGreen] (b) at (3,7) {};
	\vertex[color=ForestGreen] (c) at (4.7,11) {};
	\vertex[color=ForestGreen] (d) at (2,15) {};
	\vertex[color=ForestGreen] (e) at (-1,11) {};
	\vertex[color=ForestGreen] (f) at (-4,7) {};
	\vertex[color=ForestGreen] (g) at (-4,3) {};
	\draw[thick, color=ForestGreen] (a) -- (b) -- (c) -- (d) -- (e) -- (f) -- (g) -- (a);
	\draw[thick, color=ForestGreen] (e) --
    (b);
\end{scope}	

\begin{scope}[scale=0.3, xshift=-670, yshift=-40]
	\draw[fill, color=gray!33] (0,0) rectangle (2,1);
	\draw[fill, color=gray!33] (1,1) rectangle (4,2);
	\draw[fill, color=gray!33] (3,2) rectangle (5,3);
	\draw[very thin, color=gray!100] (0,0) grid (5,3);
	\draw[thick, color=red] (0,0)--(0,1)--(1,1)--(1,2)--(3,2)--(3,2)--(3,3)--(5,3);
\end{scope}
\begin{scope}[scale=0.3, xshift=-590, yshift=160]
	\draw[fill, color=gray!33] (0,0) rectangle (2,1);
	\draw[fill, color=gray!33] (1,1) rectangle (4,2);
	\draw[fill, color=gray!33] (3,2) rectangle (5,3);
	\draw[very thin, color=gray!100] (0,0) grid (5,3);
	\draw[thick, color=red] (0,0)--(0,1)--(1,1)--(1,2)--(2,2)--(2,3)--(3,3)--(5,3);
\end{scope}
\begin{scope}[scale=0.3, xshift=-530, yshift=300]
	\draw[fill, color=gray!33] (0,0) rectangle (2,1);
	\draw[fill, color=gray!33] (1,1) rectangle (4,2);
	\draw[fill, color=gray!33] (3,2) rectangle (5,3);
	\draw[very thin, color=gray!100] (0,0) grid (5,3);
	\draw[thick, color=red] (0,0)--(0,1)--(1,1)--(1,3)--(2,3)--(3,3)--(5,3);
\end{scope}
\begin{scope}[scale=0.3, xshift=-830, yshift=450]
	\draw[fill, color=gray!33] (0,0) rectangle (2,1);
	\draw[fill, color=gray!33] (1,1) rectangle (4,2);
	\draw[fill, color=gray!33] (3,2) rectangle (5,3);
	\draw[very thin, color=gray!100] (0,0) grid (5,3);
	\draw[thick, color=red] (0,0)--(0,3)--(3,3)--(5,3);
\end{scope}
\begin{scope}[scale=0.3, xshift=-930, yshift=330]
	\draw[fill, color=gray!33] (0,0) rectangle (2,1);
	\draw[fill, color=gray!33] (1,1) rectangle (4,2);
	\draw[fill, color=gray!33] (3,2) rectangle (5,3);
	\draw[very thin, color=gray!100] (0,0) grid (5,3);
	\draw[thick, color=red] (0,0)--(0,2)--(1,2)--(1,3)--(5,3);
\end{scope}
\begin{scope}[scale=0.3, xshift=-1010, yshift=200]
	\draw[fill, color=gray!33] (0,0) rectangle (2,1);
	\draw[fill, color=gray!33] (1,1) rectangle (4,2);
	\draw[fill, color=gray!33] (3,2) rectangle (5,3);
	\draw[very thin, color=gray!100] (0,0) grid (5,3);
	\draw[thick, color=red] (0,0)--(0,2)--(2,2)--(2,3)--(5,3);
\end{scope}
\begin{scope}[scale=0.3, xshift=-1010, yshift=40]
	\draw[fill, color=gray!33] (0,0) rectangle (2,1);
	\draw[fill, color=gray!33] (1,1) rectangle (4,2);
	\draw[fill, color=gray!33] (3,2) rectangle (5,3);
	\draw[very thin, color=gray!100] (0,0) grid (5,3);
	\draw[thick, color=red] (0,0)--(0,2)--(3,2)--(3,3)--(5,3);
\end{scope}


\begin{scope}[scale=.33, xshift=0, yshift=0]
	\vertex[color=ForestGreen] (a) at (0,0)  {};
	\vertex[color=ForestGreen] (b) at (3,7) {};
	\vertex[color=ForestGreen] (c) at (4.7,11) {};
	\vertex[color=ForestGreen] (d) at (2,15) {};
	\vertex[color=ForestGreen] (e) at (-1,11) {};
	\vertex[color=ForestGreen] (f) at (-4,7) {};
	\vertex[color=ForestGreen] (g) at (-4,3) {};
	\draw[thick, color=ForestGreen] (a) -- (b) -- (c) -- (d) -- (e) -- (f) -- (g) -- (a);
	\draw[thick, color=ForestGreen] (e) --
    (b);
\end{scope}	
\begin{scope}[xshift=0, yshift=-5, scale=0.25]
	\ijvxw(1) at (1,0){};\ijvxw(3) at (3,0){};\ijvxw(5) at (5,0){};
	\ijvxw(6) at (6,0){};\ijvxw(8) at (8,0){};\ijvxw(9) at (9,0){};
	\ijvxb(2) at (2,0){};\ijvxb(4) at (4,0){};\ijvxb(7) at (7,0){};\ijvxb(10) at (10,0){};
	\draw[] (1) to [bend left=60] (2);	\draw[] (3) to [bend left=60] (4);
    \draw[] (5) to [bend left=60] (7);  \draw[] (6) to [bend left=60] (7);
    \draw[] (8) to [bend left=60] (10); \draw[] (9) to [bend left=60] (10);    
    \draw[] (1) to [bend left=60] (10);
    \draw[] (1) to [bend left=60] (7);  
    \draw[] (1) to [bend left=60] (4);    
\end{scope}
\begin{scope}[xshift=-120, yshift=22, scale=0.25]
	\ijvxw(1) at (1,0){};\ijvxw(3) at (3,0){};\ijvxw(5) at (5,0){};
	\ijvxw(6) at (6,0){};\ijvxw(8) at (8,0){};\ijvxw(9) at (9,0){};
	\ijvxb(2) at (2,0){};\ijvxb(4) at (4,0){};\ijvxb(7) at (7,0){};\ijvxb(10) at (10,0){};
	\draw[] (1) to [bend left=60] (2);	\draw[] (3) to [bend left=60] (4);
    \draw[] (5) to [bend left=60] (7);  \draw[] (6) to [bend left=60] (7);
    \draw[] (8) to [bend left=60] (10); \draw[] (9) to [bend left=60] (10);    
    \draw[] (1) to [bend left=60] (10);
    \draw[] (3) to [bend left=60] (7);    
    \draw[] (1) to [bend left=60] (7);
\end{scope}
\begin{scope}[xshift=-120, yshift=62, scale=0.25]
	\ijvxw(1) at (1,0){};\ijvxw(3) at (3,0){};\ijvxw(5) at (5,0){};
	\ijvxw(6) at (6,0){};\ijvxw(8) at (8,0){};\ijvxw(9) at (9,0){};
	\ijvxb(2) at (2,0){};\ijvxb(4) at (4,0){};\ijvxb(7) at (7,0){};\ijvxb(10) at (10,0){};
	\draw[] (1) to [bend left=60] (2);	\draw[] (3) to [bend left=60] (4);
    \draw[] (5) to [bend left=60] (7);  \draw[] (6) to [bend left=60] (7);
    \draw[] (8) to [bend left=60] (10); \draw[] (9) to [bend left=60] (10);    
    \draw[] (1) to [bend left=60] (10);
    \draw[] (3) to [bend left=60] (7);    
    \draw[] (3) to [bend left=60] (10);
\end{scope}
\begin{scope}[xshift=-95, yshift=95, scale=0.25]
	\ijvxw(1) at (1,0){};\ijvxw(3) at (3,0){};\ijvxw(5) at (5,0){};
	\ijvxw(6) at (6,0){};\ijvxw(8) at (8,0){};\ijvxw(9) at (9,0){};
	\ijvxb(2) at (2,0){};\ijvxb(4) at (4,0){};\ijvxb(7) at (7,0){};\ijvxb(10) at (10,0){};
	\draw[] (1) to [bend left=60] (2);	\draw[] (3) to [bend left=60] (4);
    \draw[] (5) to [bend left=60] (7);  \draw[] (6) to [bend left=60] (7);
    \draw[] (8) to [bend left=60] (10); \draw[] (9) to [bend left=60] (10);    
    \draw[] (1) to [bend left=60] (10);
    \draw[] (5) to [bend left=60] (10);    
    \draw[] (3) to [bend left=60] (10);
\end{scope}
\begin{scope}[xshift=25, yshift=50, scale=0.25]
	\ijvxw(1) at (1,0){};\ijvxw(3) at (3,0){};\ijvxw(5) at (5,0){};
	\ijvxw(6) at (6,0){};\ijvxw(8) at (8,0){};\ijvxw(9) at (9,0){};
	\ijvxb(2) at (2,0){};\ijvxb(4) at (4,0){};\ijvxb(7) at (7,0){};\ijvxb(10) at (10,0){};
	\draw[] (1) to [bend left=60] (2);	\draw[] (3) to [bend left=60] (4);
    \draw[] (5) to [bend left=60] (7);  \draw[] (6) to [bend left=60] (7);
    \draw[] (8) to [bend left=60] (10); \draw[] (9) to [bend left=60] (10);    
    \draw[] (1) to [bend left=60] (10);
    \draw[] (5) to [bend left=60] (10);  
    \draw[] (1) to [bend left=60] (4);    
\end{scope}
\begin{scope}[xshift=40, yshift=90, scale=0.25]
	\ijvxw(1) at (1,0){};\ijvxw(3) at (3,0){};\ijvxw(5) at (5,0){};
	\ijvxw(6) at (6,0){};\ijvxw(8) at (8,0){};\ijvxw(9) at (9,0){};
	\ijvxb(2) at (2,0){};\ijvxb(4) at (4,0){};\ijvxb(7) at (7,0){};\ijvxb(10) at (10,0){};
	\draw[] (1) to [bend left=60] (2);	\draw[] (3) to [bend left=60] (4);
    \draw[] (6) to [bend left=60] (10);  \draw[] (6) to [bend left=60] (7);
    \draw[] (8) to [bend left=60] (10); \draw[] (9) to [bend left=60] (10);    
    \draw[] (1) to [bend left=60] (10);
    \draw[] (5) to [bend left=60] (10);  
    \draw[] (1) to [bend left=60] (4);    
\end{scope}
\begin{scope}[xshift=-65, yshift=135, scale=0.25]
	\ijvxw(1) at (1,0){};\ijvxw(3) at (3,0){};\ijvxw(5) at (5,0){};
	\ijvxw(6) at (6,0){};\ijvxw(8) at (8,0){};\ijvxw(9) at (9,0){};
	\ijvxb(2) at (2,0){};\ijvxb(4) at (4,0){};\ijvxb(7) at (7,0){};\ijvxb(10) at (10,0){};
	\draw[] (1) to [bend left=60] (2);	\draw[] (3) to [bend left=60] (4);
    \draw[] (6) to [bend left=60] (10);  \draw[] (6) to [bend left=60] (7);
    \draw[] (8) to [bend left=60] (10); \draw[] (9) to [bend left=60] (10);    
    \draw[] (1) to [bend left=60] (10);
    \draw[] (5) to [bend left=60] (10);  
    \draw[] (3) to [bend left=60] (10);    
\end{scope}


\begin{scope}[scale=.33, xshift=650, yshift=0]
	\vertex[color=ForestGreen] (a) at (0,0)  {};
	\vertex[color=ForestGreen] (b) at (3,7) {};
	\vertex[color=ForestGreen] (c) at (4.7,11) {};
	\vertex[color=ForestGreen] (d) at (2,15) {};
	\vertex[color=ForestGreen] (e) at (-1,11) {};
	\vertex[color=ForestGreen] (f) at (-4,7) {};
	\vertex[color=ForestGreen] (g) at (-4,3) {};
	\draw[thick, color=ForestGreen] (a) -- (b) -- (c) -- (d) -- (e) -- (f) -- (g) -- (a);
	\draw[thick, color=ForestGreen] (e) --
    (b);
\end{scope}	

\begin{scope}[scale=0.3, xshift=770, yshift=-40]
	\draw[fill, color=gray!33] (0,0) rectangle (2,1);
	\draw[fill, color=gray!33] (1,1) rectangle (4,2);
	\draw[fill, color=gray!33] (3,2) rectangle (5,3);
	\draw[very thin, color=gray!100] (0,0) grid (5,3);
	\draw[very thick, color=RawSienna] (0,0)--(0,3)--(5,3);
	\draw[very thick, color=RawSienna] (0,1)--(1,1);
	\draw[very thick, color=RawSienna] (0,2)--(3,2);
	
	\node[circle,color=black,fill=black, inner sep=1pt] (r) at (0,3) {};
	\node[circle,color=black,fill=black, inner sep=1pt] (r) at (0,2) {};
	\node[circle,color=black,fill=black, inner sep=1pt] (r) at (0,1) {};
	\node[circle,color=black,fill=black, inner sep=1pt] (r) at (0,0) {};
	\node[circle,color=black,fill=black, inner sep=1pt] (r) at (1,1) {};	
	\node[circle,color=black,fill=black, inner sep=1pt] (r) at (2,2) {};	
	\node[circle,color=black,fill=black, inner sep=1pt] (r) at (3,2) {};
	\node[circle,color=black,fill=black, inner sep=1pt] (r) at (4,3) {};	
	\node[circle,color=black,fill=black, inner sep=1pt] (r) at (5,3) {};	
\end{scope}
\begin{scope}[scale=0.3, xshift=850, yshift=155]
	\draw[fill, color=gray!33] (0,0) rectangle (2,1);
	\draw[fill, color=gray!33] (1,1) rectangle (4,2);
	\draw[fill, color=gray!33] (3,2) rectangle (5,3);
	\draw[very thin, color=gray!100] (0,0) grid (5,3);
	\draw[very thick, color=RawSienna] (0,0)--(0,3)--(5,3);
	\draw[very thick, color=RawSienna] (0,1)--(1,1);
	\draw[very thick, color=RawSienna] (2,2)--(3,2);
	\draw[very thick, color=RawSienna] (2,2)--(2,3);
	
	\node[circle,color=black,fill=black, inner sep=1pt] (r) at (0,3) {};
	\node[circle,color=black,fill=black, inner sep=1pt] (r) at (2,2) {};
	\node[circle,color=black,fill=black, inner sep=1pt] (r) at (0,1) {};
	\node[circle,color=black,fill=black, inner sep=1pt] (r) at (0,0) {};
	\node[circle,color=black,fill=black, inner sep=1pt] (r) at (1,1) {};	
	\node[circle,color=black,fill=black, inner sep=1pt] (r) at (2,3) {};	
	\node[circle,color=black,fill=black, inner sep=1pt] (r) at (3,2) {};
	\node[circle,color=black,fill=black, inner sep=1pt] (r) at (4,3) {};	
	\node[circle,color=black,fill=black, inner sep=1pt] (r) at (5,3) {};	
\end{scope}

\begin{scope}[scale=0.3, xshift=900, yshift=300]
	\draw[fill, color=gray!33] (0,0) rectangle (2,1);
	\draw[fill, color=gray!33] (1,1) rectangle (4,2);
	\draw[fill, color=gray!33] (3,2) rectangle (5,3);
	\draw[very thin, color=gray!100] (0,0) grid (5,3);
	\draw[very thick, color=RawSienna] (0,0)--(0,3)--(5,3);
	\draw[very thick, color=RawSienna] (0,1)--(1,1);
	\draw[very thick, color=RawSienna] (3,3)--(3,2);
	
	\node[circle,color=black,fill=black, inner sep=1pt] (r) at (0,3) {};
	\node[circle,color=black,fill=black, inner sep=1pt] (r) at (3,3) {};
	\node[circle,color=black,fill=black, inner sep=1pt] (r) at (0,1) {};
	\node[circle,color=black,fill=black, inner sep=1pt] (r) at (0,0) {};
	\node[circle,color=black,fill=black, inner sep=1pt] (r) at (1,1) {};	
	\node[circle,color=black,fill=black, inner sep=1pt] (r) at (2,3) {};	
	\node[circle,color=black,fill=black, inner sep=1pt] (r) at (3,2) {};
	\node[circle,color=black,fill=black, inner sep=1pt] (r) at (4,3) {};	
	\node[circle,color=black,fill=black, inner sep=1pt] (r) at (5,3) {};	
\end{scope}
\begin{scope}[scale=0.3, xshift=590, yshift=450]
	\draw[fill, color=gray!33] (0,0) rectangle (2,1);
	\draw[fill, color=gray!33] (1,1) rectangle (4,2);
	\draw[fill, color=gray!33] (3,2) rectangle (5,3);
	\draw[very thin, color=gray!100] (0,0) grid (5,3);
	\draw[very thick, color=RawSienna] (0,0)--(0,3)--(5,3);
	\draw[very thick, color=RawSienna] (1,3)--(1,1);
	\draw[very thick, color=RawSienna] (3,3)--(3,2);
	
	\node[circle,color=black,fill=black, inner sep=1pt] (r) at (0,3) {};
	\node[circle,color=black,fill=black, inner sep=1pt] (r) at (3,3) {};
	\node[circle,color=black,fill=black, inner sep=1pt] (r) at (1,3) {};
	\node[circle,color=black,fill=black, inner sep=1pt] (r) at (0,0) {};
	\node[circle,color=black,fill=black, inner sep=1pt] (r) at (1,1) {};	
	\node[circle,color=black,fill=black, inner sep=1pt] (r) at (2,3) {};	
	\node[circle,color=black,fill=black, inner sep=1pt] (r) at (3,2) {};
	\node[circle,color=black,fill=black, inner sep=1pt] (r) at (4,3) {};	
	\node[circle,color=black,fill=black, inner sep=1pt] (r) at (5,3) {};	
\end{scope}

\begin{scope}[scale=0.3, xshift=500, yshift=330]
	\draw[fill, color=gray!33] (0,0) rectangle (2,1);
	\draw[fill, color=gray!33] (1,1) rectangle (4,2);
	\draw[fill, color=gray!33] (3,2) rectangle (5,3);
	\draw[very thin, color=gray!100] (0,0) grid (5,3);
	\draw[very thick, color=RawSienna] (0,0)--(0,3)--(5,3);
	\draw[very thick, color=RawSienna] (1,3)--(1,1);
	\draw[very thick, color=RawSienna] (2,2)--(3,2);
    \draw[very thick, color=RawSienna] (2,2)--(2,3);

	\node[circle,color=black,fill=black, inner sep=1pt] (r) at (0,3) {};
	\node[circle,color=black,fill=black, inner sep=1pt] (r) at (2,2) {};
	\node[circle,color=black,fill=black, inner sep=1pt] (r) at (1,3) {};
	\node[circle,color=black,fill=black, inner sep=1pt] (r) at (0,0) {};
	\node[circle,color=black,fill=black, inner sep=1pt] (r) at (1,1) {};	
	\node[circle,color=black,fill=black, inner sep=1pt] (r) at (2,3) {};	
	\node[circle,color=black,fill=black, inner sep=1pt] (r) at (3,2) {};
	\node[circle,color=black,fill=black, inner sep=1pt] (r) at (4,3) {};	
	\node[circle,color=black,fill=black, inner sep=1pt] (r) at (5,3) {};	
\end{scope}

\begin{scope}[scale=0.3, xshift=420, yshift=200]
	\draw[fill, color=gray!33] (0,0) rectangle (2,1);
	\draw[fill, color=gray!33] (1,1) rectangle (4,2);
	\draw[fill, color=gray!33] (3,2) rectangle (5,3);
	\draw[very thin, color=gray!100] (0,0) grid (5,3);
	\draw[very thick, color=RawSienna] (0,0)--(0,3)--(5,3);
	\draw[very thick, color=RawSienna] (1,3)--(1,1);
	\draw[very thick, color=RawSienna] (2,2)--(3,2);
    \draw[very thick, color=RawSienna] (2,2)--(1,2);

	\node[circle,color=black,fill=black, inner sep=1pt] (r) at (0,3) {};
	\node[circle,color=black,fill=black, inner sep=1pt] (r) at (2,2) {};
	\node[circle,color=black,fill=black, inner sep=1pt] (r) at (1,3) {};
	\node[circle,color=black,fill=black, inner sep=1pt] (r) at (0,0) {};
	\node[circle,color=black,fill=black, inner sep=1pt] (r) at (1,1) {};	
	\node[circle,color=black,fill=black, inner sep=1pt] (r) at (1,2) {};	
	\node[circle,color=black,fill=black, inner sep=1pt] (r) at (3,2) {};
	\node[circle,color=black,fill=black, inner sep=1pt] (r) at (4,3) {};	
	\node[circle,color=black,fill=black, inner sep=1pt] (r) at (5,3) {};	
\end{scope}

\begin{scope}[scale=0.3, xshift=420, yshift=60]
	\draw[fill, color=gray!33] (0,0) rectangle (2,1);
	\draw[fill, color=gray!33] (1,1) rectangle (4,2);
	\draw[fill, color=gray!33] (3,2) rectangle (5,3);
	\draw[very thin, color=gray!100] (0,0) grid (5,3);
	\draw[very thick, color=RawSienna] (0,0)--(0,3)--(5,3);
	\draw[very thick, color=RawSienna] (1,2)--(1,1);
	\draw[very thick, color=RawSienna] (0,2)--(3,2);
    \draw[very thick, color=RawSienna] (2,2)--(1,2);

	\node[circle,color=black,fill=black, inner sep=1pt] (r) at (0,3) {};
	\node[circle,color=black,fill=black, inner sep=1pt] (r) at (2,2) {};
	\node[circle,color=black,fill=black, inner sep=1pt] (r) at (0,2) {};
	\node[circle,color=black,fill=black, inner sep=1pt] (r) at (0,0) {};
	\node[circle,color=black,fill=black, inner sep=1pt] (r) at (1,1) {};	
	\node[circle,color=black,fill=black, inner sep=1pt] (r) at (1,2) {};	
	\node[circle,color=black,fill=black, inner sep=1pt] (r) at (3,2) {};
	\node[circle,color=black,fill=black, inner sep=1pt] (r) at (4,3) {};	
	\node[circle,color=black,fill=black, inner sep=1pt] (r) at (5,3) {};	
\end{scope}
\end{tikzpicture}
\caption{The $\nu$-Tamari lattice indexed by $\nu$-Dyck paths (left),  $(I,\overline{J})$-trees (center), and $\nu$-trees (right) for $\nu=NENE^2NE^2$.}
\label{fig.35TamariLattice}
\end{figure}
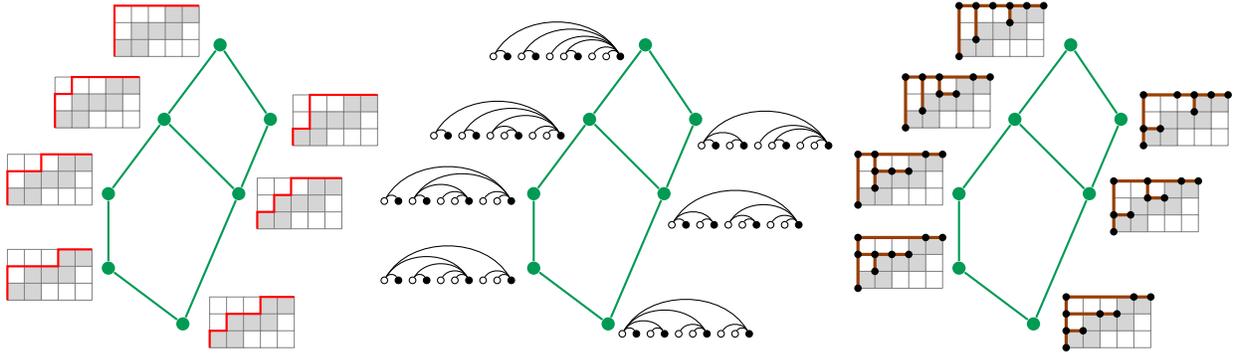
\subsubsection{The \texorpdfstring{$\nu$-}-Tamari lattice as the flip lattice on \texorpdfstring{$(I,\overline{J})$-}-trees}
\label{subsec:nu-Tamari-IJ}
Next, we consider the description of $\mathrm{Tam}(\nu)$ in terms of $(I,\overline{J})$-trees as introduced by Ceballos, Padrol and Sarmiento~\cite{CPS19}.

\begin{definition} Let $k\in \mathbb{Z}$ and let $I\sqcup \overline{J}$ be a bipartition of $[k]$ such that $1 \in I$ and $k \in \overline{J}$. An $(I,\overline{J})$-{\em tree} is a maximal subgraph of the complete bipartite graph $K_{|I|,|\overline{J}|}$ that is 
\begin{itemize}
    \item[(i)] {\em increasing}: each arc $(i,\overline{j})$ satisfies $i < \overline{j}$; and 
    \item[(ii)] {\em non-crossing}: the graph does not contain arcs $(i,\overline{j})$ and $(i',\overline{j}')$ with $i < i' < j < \overline{j}'$. 
\end{itemize}
\end{definition}

To a pair $(I,\overline{J})$ we can associate a unique lattice path $\nu$ as follows. 
Assign to the $i$-th element of $I$ the label $E_i$ and to the $j$-th element of $\overline{J}$ the label $N_j$. 
Reading the labels of the nodes $k=2,\ldots, k-1$ in increasing order yields a lattice path $\nu$ from $(0,0)$ to $(|I|-1,|\overline{J}|-1)$. 
See Figure~\ref{fig:threeTamariObjects} for an example. 
Conversely, a lattice path $\nu$ determines a unique pair $(I,\overline{J})$, and hence a unique set of $(I,\overline{J})$-trees. 
Let $\mathcal{T}_{\nu}$ denote the set of $(I,\overline{J})$-trees determined by $\nu$.

Given two $(I,\overline{J})$-trees $T$ and $T'$ in $\calT_\nu$, we say that $T'$ is an {\em increasing flip} of $T$ if $T'$ is obtained from $T$ by replacing an arc $(i,j)$ with an arc $(i',j')$, where $i < i'$ and $j<j'$. 
Define a relation on $\calT_\nu$ by $T \lessdot_{I,\overline{J}} T'$ whenever $T'$ is obtained from $T$ by an increasing flip. 
The transitive closure $<_{I,\overline{J}}$ of the relation $T\lessdot_{I,\overline{J}} T'$ gives a lattice structure on $\calT_\nu$ (see \cite[Lemma 3.1]{CPS19}). 

\begin{proposition}[{\cite[Theorem 3.4]{CPS19}}] 
The increasing flip lattice on the set of $(I,\overline{J})$-trees in $\calT_\nu$ is isomorphic to $\mathrm{Tam}(\nu)$. 
\end{proposition}

The lattice in the center of Figure~\ref{fig.35TamariLattice} gives a $\nu$-Tamari lattice with vertices indexed by $(I,\overline{J})$-trees. The following corollary is then immediate.

\begin{corollary} \label{cor:nuTamariHasseDiag}
The Hasse diagram of the $\nu$-Tamari lattice is the graph whose vertices are the $(I,\overline{J})$-trees determined by $\nu$, with edges between $(I,\overline{J})$-trees that differ by exactly one arc.
\end{corollary}

\subsubsection{The \texorpdfstring{$\nu$-}-Tamari lattice as the rotation lattice on \texorpdfstring{$\nu$-}-trees}
\label{subsec:nu-Tamari-nu-tree}
Given a lattice path $\nu$ from $(0,0)$ to $(b,a)$, let $\calL_\nu$ denote the set of lattice points in the plane which lie weakly above $\nu$ inside the rectangle defined by $(0,0)$ and $(b,a)$. 
An $(I,\overline{J})$-tree $T \in \calT_\nu$ can be represented as a point configuration in $\calL_\nu$. 
For each arc $(E_x,N_y)$ in $T$, we associate the point $(x,y)$ in $\calL_\nu$. 
The collection of points corresponding to the arcs of $T$ is called the {\em grid representation} of $T$. 
These grid representations were studied in detail in \cite{CPS20} under the name $\nu$-trees (see \cite[Remark 3.7]{CPS20}). 
The word `tree' is justified by the fact that each point except $(0,a)$ in a grid representation has either one point above it in the same column or one point to its left in the same row, but not both \cite[Lemma 2.2]{CPS20}. 
Thus we can connect each point to the point above it or to its left, forming a rooted binary tree with a root at $(0,a)$.  
Figure~\ref{fig:threeTamariObjects} (right) gives an example of a $\nu$-tree, which is the grid representation of the $(I,\overline{J})$-tree in the center. 
The non-crossing condition for arcs in an $(I,\overline{J})$-tree can be translated to $\nu$-trees and a $\nu$-tree can then be defined without reference to an $(I,\overline{J})$-tree as follows.
\begin{definition}
Two lattice points $p$ and $q$ in $\calL_\nu$ are said to be {\em $\nu$-incompatible} if $p$ is southwest or northeast of $q$, and the smallest rectangle containing $p$ and $q$ contains only lattice points of $\calL_\nu$.
The points $p$ and $q$ are {\em $\nu$-compatible} if they are not $\nu$-incompatible. 
A \emph{$\nu$-tree} is a maximal set of $\nu$-compatible points in $\calL_\nu$. 
\end{definition}

If a $\nu$-tree has points $p$, $q$ and $r$ such that $r$ is the southwest corner of the rectangle determined by $p$ and $q$ (with $p$ northwest of $q$ or vice versa), then replacing $r$ with the lattice point at the point at the northeast corner of the rectangle is called a {\em (right) rotation}. 
For example, in Figure~\ref{fig:threeTamariObjects}, the only possible rotation in the $\nu$-tree replaces the lattice point $(1,2)$ with $(2,3)$. 
Rotations in a $\nu$-tree are a direct translation of increasing flips for $(I,\overline{J})$-trees. 
Define a partial order  $<_{\nu}$ on the set of $\nu$-trees given by a covering relation $T \lessdot_\nu T'$ if and only if $T'$ is formed from $T$ by a rotation. 
This partial order is the {\em rotation lattice} of $\nu$-trees \cite[Theorem 2.8]{CPS20}. 
The rightmost lattice in Figure~\ref{fig.35TamariLattice} shows the $\nu$-Tamari lattice indexed with $\nu$-trees.
\begin{proposition}[{\cite[Theorem 3.3]{CPS20}}] The rotation lattice on the set of $\nu$-trees is isomorphic to $\mathrm{Tam}(\nu)$. 
\end{proposition}

\subsection{The triangulation}

In this subsection we study the {\em length-framed} triangulation of $\calF_{\car(\nu)}$ and show its connection with $\mathrm{Tam}(\nu)$. 
To define the length framing of $\car(\nu)$, or any framing for that matter, we need to be able to distinguish between the multiedges. To that end, we label multiedges between two vertices (from top to bottom in a planar drawing of $\car(\nu)$) with increasing natural numbers (see Figure \ref{fig.two_framings_ex}).

\begin{definition}\label{def.lengthframing}
Let $G$ be a graph on the vertex set $\{1,\ldots, n+1\}$.
Define the {\em length} of a directed edge $(i,j)$ to be $j-i$. Given an inner vertex $i\in[2,n]$ of $G$, the {\em length framing} for $G$ at $i$ is the pair of linear orders $(\prec_{\inedge(i)}, \prec_{\outedge(i)})$ where longer edges precede shorter edges and multiedges with smaller labels precede ones with larger labels. 
Figure~\ref{fig.two_framings_ex} gives an example of the length framing of $\car(\nu)$ with $\nu=NE^2NENNE^3NE$.
\end{definition}

Recall from Section \ref{sec.intro} that the vertices of $\calF_{\car(\nu)}$ are determined by routes (unitary flows) in $\car(\nu)$. These are completely characterized by two edges in $\car(\nu)$: the initial edge that is of the form $(1,j+1)$ with label $i$, and the terminal edge that is of the form $(\ell+1,n+1)$ (which always has label $i=1$) with $1 \le j\leq \ell < n$. 
We denote such a route by $R_{j,i,\ell}$. 

\begin{lemma}\label{lem:RouteArcBijection}
The set of routes $\calR_\nu$ in the $\nu$-caracol graph $\car(\nu)$ is in bijection with the set $\mathcal{A}_{\nu}$ of possible arcs in the $(I,\overline{J})$-trees in $\mathcal{T}_{\nu}$. 
\end{lemma}
\begin{proof}
We define a map $\varphi: \mathcal{R}_\nu\rightarrow \mathcal{A}_\nu$.
The elements in the sets $I$ and $\overline{J}$  
respectively correspond to the $N$ and $E$ steps in the path $\overline{\nu}=E\nu N$, as in Figure \ref{fig:threeTamariObjects}.
Describing the bijection in terms of the $N$ and $E$ steps is easier than using the elements of $I$ and $\overline{J}$, so we add indices to the $N$ and $E$ steps in order to distinguish between them. 
First index the $j$-th  left-to-right $N$ step in $\overline{\nu}$ by $j$, then index each $E$ with a pair $(j,i)$ where $j$ is index of the next $N_j$ in $\overline{\nu}$, and $i$ is the number of steps taken to reach $N_j$. 
Now, arcs in the $(I,\overline{J})$-trees in $\mathcal{T}_{\nu}$ can be expressed as pairs of the form $(E_{j,i},N_\ell)$.
Recall that the routes in $\car(\nu)$ are of the form $R_{j,i,\ell}$. We define the map $\varphi$ by $\varphi(R_{j,i,\ell})=(E_{j,i},N_{\ell})$. Figure~\ref{fig.length_clique} shows an example of this correspondence between routes and arcs. 

For a route $R_{j,i,\ell}$, we have that $1\leq j <n$, and $i\leq \inedge_{j+1}$, and hence $E_{j,i}$ is the label for a node in an $(I,\overline{J})$-tree in $\mathcal{T}_\nu$. 
Since $1\leq \ell <n$, we also have that $N_\ell$ is the label of a node. 
Now $(E_{j,i},N_\ell)$ is a valid arc in $\mathcal{A}_\nu$ since $j \leq \ell$, and the map $\varphi$ is thus well-defined. 
Define the inverse map $\varphi\inv :\mathcal{A}_\nu \to \mathcal{R}_\nu$ by $\varphi\inv((E_{j,i},N_\ell)) = R_{j,i,\ell}$. 
Since $(E_{j,i},N_\ell)$ is an arc, the node $N_{j+1}$ is preceded by at least $i$ many nodes labeled by $E$ steps. 
Thus $\car(\nu)$ has an edge $(1,j+1)$ labeled $i$. Now $R_{j,i,\ell}$ is the unique route in $\mathcal{R}_\nu$ determined by $j,i$ and $\ell$, and so $\varphi\inv$ is well-defined. 
It is clear that $\varphi\circ \varphi\inv$ and $\varphi\inv\circ \varphi$ are identity maps.
\end{proof}

\begin{figure}[ht!]
\centering
\begin{tikzpicture}
\begin{scope}[xshift=0, yshift=0, scale=.33]
    \node[]  at (5.5,1.5)  {\tiny \textcolor{cyan}{$1$}};
    \cvx[color=black](a1) at (1,0) {}; \cvx[color=black](a2) at (2,0) {}; 
	\cvx[color=black](a3) at (3,0) {}; \cvx[color=black](a4) at (4,0) {}; 
	\cvx[color=black](a5) at (5,0) {}; \cvx[color=black](a6) at (6,0) {};
	\draw[color=black, thick] (a1)--(a6);
	\draw[color=black!25] (a1) .. controls (1.25,.4) and (2.75,.4) .. (a3);	
    \draw[color=black!25] (a1) .. controls (1.25,0.8) and (3.75,0.8) .. (a4);
    \draw[color=black!25] (a1) .. controls (1.25,1.25) and (3.75,1.25) .. (a4);    
	\draw[color=black!25] (a1) .. controls (1.25, 1.75) and (4.75,1.75) .. (a5);
	\draw[color=black!25] (a1) .. controls (1.25, 2.25) and (4.75,2.25) .. (a5);
	\draw[color=black!25] (a4) to[out=50,in=130] (a6);
	\draw[color=black!25] (a3) to[out=50,in=130] (a6);
	\draw[color=black!25] (a2) to[out=50,in=130] (a6);
\end{scope}
\begin{scope}[xshift=60, yshift=0, scale=.33]
	\node[]  at (5.5,1.5)  {\tiny \textcolor{cyan}{$2$}};
	\cvx[color=black](a1) at (1,0) {}; 
	\cvx[color=black](a2) at (2,0) {}; 
	\cvx[color=black](a3) at (3,0) {}; 
	\cvx[color=black](a4) at (4,0) {}; 
	\cvx[color=black!25](a5) at (5,0) {};
	\cvx[color=black](a6) at (6,0) {};
	\draw[color=black!25] (a4)--(a6);
	\draw[color=black!25] (a1) .. controls (1.25,.4) and (2.75,.4) .. (a3);	
    \draw[color=black!25] (a1) .. controls (1.25,0.8) and (3.75,0.8) .. (a4);
    \draw[color=black!25] (a1) .. controls (1.25,1.25) and (3.75,1.25) .. (a4);    
	\draw[color=black!25] (a1) .. controls (1.25, 1.75) and (4.75,1.75) .. (a5);
	\draw[color=black!25] (a1) .. controls (1.25, 2.25) and (4.75,2.25) .. (a5);
	\draw[color=black!25] (a3) to[out=50,in=130] (a6);
	\draw[color=black!25] (a2) to[out=50,in=130] (a6);
	\draw[color=black, thick] (a4) to[out=50,in=130] (a6);
	\draw[color=black, thick] (a1)--(a4);
\end{scope}
\begin{scope}[xshift=120, yshift=0, scale=.33]
	\node[]  at (5.5,1.5)  {\tiny \textcolor{cyan}{$3$}};
	\cvx[color=black](a1) at (1,0) {}; 
	\cvx[color=black](a2) at (2,0) {}; 
	\cvx[color=black!25](a3) at (3,0) {}; 
	\cvx[color=black!25](a4) at (4,0) {}; 
	\cvx[color=black!25](a5) at (5,0) {};
	\cvx[color=black](a6) at (6,0) {};
	\draw[color=black!25] (a2)--(a6);
	\draw[color=black!25] (a1) .. controls (1.25,.4) and (2.75,.4) .. (a3);	
    \draw[color=black!25] (a1) .. controls (1.25,0.8) and (3.75,0.8) .. (a4);
    \draw[color=black!25] (a1) .. controls (1.25,1.25) and (3.75,1.25) .. (a4);    
	\draw[color=black!25] (a1) .. controls (1.25, 1.75) and (4.75,1.75) .. (a5);
	\draw[color=black!25] (a1) .. controls (1.25, 2.25) and (4.75,2.25) .. (a5);
	\draw[color=black!25] (a4) to[out=50,in=130] (a6);
	\draw[color=black!25] (a3) to[out=50,in=130] (a6);
	\draw[color=black!25] (a2) to[out=50,in=130] (a6);
    \draw[color=black, thick] (a1)--(a2);
	\draw[color=black, thick] (a2) to[out=50,in=130] (a6);
\end{scope}
\begin{scope}[xshift=180, yshift=0, scale=.33]
	\node[]  at (5.5,1.5)  {\tiny \textcolor{cyan}{$4$}};
	\cvx[color=black](a1) at (1,0) {}; 
	\cvx[color=black!25](a2) at (2,0) {}; 
	\cvx[color=black](a3) at (3,0) {}; 
	\cvx[color=black](a4) at (4,0) {}; 
	\cvx[color=black!25](a5) at (5,0) {};
	\cvx[color=black](a6) at (6,0) {};
	\draw[color=black!25] (a1)--(a3);
	\draw[color=black!25] (a4)--(a6);
    \draw[color=black!25] (a1) .. controls (1.25,0.8) and (3.75,0.8) .. (a4);
    \draw[color=black!25] (a1) .. controls (1.25,1.25) and (3.75,1.25) .. (a4);    
	\draw[color=black!25] (a1) .. controls (1.25, 1.75) and (4.75,1.75) .. (a5);
	\draw[color=black!25] (a1) .. controls (1.25, 2.25) and (4.75,2.25) .. (a5);
	\draw[color=black!25] (a4) to[out=50,in=130] (a6);
	\draw[color=black!25] (a3) to[out=50,in=130] (a6);
	\draw[color=black!25] (a2) to[out=50,in=130] (a6);
	\draw[color=black, thick] (a1) .. controls (1.25,.4) and (2.75,.4) .. (a3);
	\draw[color=black, thick] (a3)--(a4);
	\draw[color=black, thick] (a4) to[out=50,in=130] (a6);
\end{scope}
\begin{scope}[xshift=240, yshift=0, scale=.33]
	\node[]  at (5.5,1.5)  {\tiny \textcolor{cyan}{$5$}};
	\cvx[color=black](a1) at (1,0) {}; 
	\cvx[color=black!25](a2) at (2,0) {}; 
	\cvx[color=black](a3) at (3,0) {}; 
	\cvx[color=black!25](a4) at (4,0) {}; 
	\cvx[color=black!25](a5) at (5,0) {};
	\cvx[color=black](a6) at (6,0) {};
	\draw[color=black!25] (a1)--(a3);
	\draw[color=black!25] (a3)--(a6);
    \draw[color=black!25] (a1) .. controls (1.25,0.8) and (3.75,0.8) .. (a4);
    \draw[color=black!25] (a1) .. controls (1.25,1.25) and (3.75,1.25) .. (a4);    
	\draw[color=black!25] (a1) .. controls (1.25, 1.75) and (4.75,1.75) .. (a5);
	\draw[color=black!25] (a1) .. controls (1.25, 2.25) and (4.75,2.25) .. (a5);
	\draw[color=black!25] (a4) to[out=50,in=130] (a6);
	\draw[color=black!25] (a2) to[out=50,in=130] (a6);
    \draw[color=black, thick] (a1) .. controls (1.25,.4) and (2.75,.4) .. (a3);	
	\draw[color=black, thick] (a3) to[out=50,in=130] (a6);
\end{scope}
\begin{scope}[xshift=30, yshift=-35, scale=.33]
	\node[]  at (5.5,1.5)  {\tiny \textcolor{cyan}{$6$}};
	\cvx[color=black](a1) at (1,0) {}; 
	\cvx[color=black!25](a2) at (2,0) {}; 
	\cvx[color=black!25](a3) at (3,0) {}; 
	\cvx[color=black](a4) at (4,0) {}; 
	\cvx[color=black!25](a5) at (5,0) {};
	\cvx[color=black](a6) at (6,0) {};
	\draw[color=black!25] (a1)--(a4); \draw[color=black!25] (a4)--(a6);
	\draw[color=black!25] (a1) .. controls (1.25,.4) and (2.75,.4) .. (a3);	
    \draw[color=black!25] (a1) .. controls (1.25,1.25) and (3.75,1.25) .. (a4);
	\draw[color=black!25] (a1) .. controls (1.25, 1.75) and (4.75,1.75) .. (a5);
	\draw[color=black!25] (a1) .. controls (1.25, 2.25) and (4.75,2.25) .. (a5);
	\draw[color=black!25] (a4) to[out=50,in=130] (a6);
	\draw[color=black!25] (a3) to[out=50,in=130] (a6);
	\draw[color=black!25] (a2) to[out=50,in=130] (a6);
    \draw[color=black, thick] (a1) .. controls (1.25,0.8) and (3.75,0.8) .. (a4);
	\draw[color=black, thick] (a4) to[out=50,in=130] (a6);
\end{scope}
\begin{scope}[xshift=90, yshift=-35, scale=.33]
	\node[]  at (5.5,1.5)  {\tiny \textcolor{cyan}{$7$}};
	\cvx[color=black](a1) at (1,0) {}; 
	\cvx[color=black!25](a2) at (2,0) {}; 
	\cvx[color=black!25](a3) at (3,0) {}; 
	\cvx[color=black](a4) at (4,0) {}; 
	\cvx[color=black!25](a5) at (5,0) {};
	\cvx[color=black](a6) at (6,0) {};
	\draw[color=black!25] (a1)--(a4);
	\draw[color=black!25] (a4)--(a6);
	\draw[color=black!25] (a1) .. controls (1.25,.4) and (2.75,.4) .. (a3);	
    \draw[color=black!25] (a1) .. controls (1.25,0.8) and (3.75,0.8) .. (a4);
    
	\draw[color=black!25] (a1) .. controls (1.25, 1.75) and (4.75,1.75) .. (a5);
	\draw[color=black!25] (a1) .. controls (1.25, 2.25) and (4.75,2.25) .. (a5);
	\draw[color=black!25] (a3) to[out=50,in=130] (a6);
	\draw[color=black!25] (a2) to[out=50,in=130] (a6);
	\draw[color=black, thick] (a4) to[out=50,in=130] (a6);    
    \draw[color=black, thick] (a1) .. controls (1.25,1.25) and (3.75,1.25) .. (a4);
\end{scope}
\begin{scope}[xshift=150, yshift=-35, scale=.33]
	\node[]  at (5.5,1.5)  {\tiny \textcolor{cyan}{$8$}};
	\cvx[color=black](a1) at (1,0) {}; 
	\cvx[color=black!25](a2) at (2,0) {}; 
	\cvx[color=black!25](a3) at (3,0) {}; 
	\cvx[color=black!25](a4) at (4,0) {}; 
	\cvx[color=black](a5) at (5,0) {};
	\cvx[color=black](a6) at (6,0) {};
	\draw[color=black!25] (a1)--(a5);
	\draw[color=black!25] (a1) .. controls (1.25,.4) and (2.75,.4) .. (a3);	
    \draw[color=black!25] (a1) .. controls (1.25,0.8) and (3.75,0.8) .. (a4);
    \draw[color=black!25] (a1) .. controls (1.25,1.25) and (3.75,1.25) .. (a4);    
    \draw[color=black!25] (a1) .. controls (1.25, 2.25) and (4.75,2.25) .. (a5);
	\draw[color=black!25] (a4) to[out=50,in=130] (a6);
	\draw[color=black!25] (a3) to[out=50,in=130] (a6);
	\draw[color=black!25] (a2) to[out=50,in=130] (a6);
	\draw[color=black, thick] (a1) .. controls (1.25, 1.75) and (4.75,1.75) .. (a5);
	\draw[color=black, thick] (a5)--(a6);
\end{scope}
\begin{scope}[xshift=210, yshift=-35, scale=.33]
	\node[]  at (5.5,1.5)  {\tiny \textcolor{cyan}{$9$}};
	\cvx[color=black](a1) at (1,0) {}; 
	\cvx[color=black!25](a2) at (2,0) {}; 
	\cvx[color=black!25](a3) at (3,0) {}; 
	\cvx[color=black!25](a4) at (4,0) {}; 
	\cvx[color=black](a5) at (5,0) {};
	\cvx[color=black](a6) at (6,0) {};
	\draw[color=black!25] (a1)--(a5);
	\draw[color=black!25] (a1) .. controls (1.25,.4) and (2.75,.4) .. (a3);	
    \draw[color=black!25] (a1) .. controls (1.25,0.8) and (3.75,0.8) .. (a4);
    \draw[color=black!25] (a1) .. controls (1.25,1.25) and (3.75,1.25) .. (a4);    
	\draw[color=black!25] (a1) .. controls (1.25, 1.75) and (4.75,1.75) .. (a5);
	\draw[color=black!25] (a4) to[out=50,in=130] (a6);
	\draw[color=black!25] (a3) to[out=50,in=130] (a6);
	\draw[color=black!25] (a2) to[out=50,in=130] (a6);
	\draw[color=black, thick] (a1) .. controls (1.25, 2.25) and (4.75,2.25) .. (a5);
	\draw[color=black, thick] (a5)--(a6);
\end{scope}

\begin{scope}[xshift = 300, yshift = -25, scale=0.55]
	\node[style={circle,draw, inner sep=2pt, fill=black}, label=below:{\scriptsize$E_{1,1}$}] (1)  at (1,0)  {};
	\node[style={circle,draw, inner sep=2pt, fill=black}, label=below:{\scriptsize$E_{2,1}$}] (3)  at (3,0)  {};
	\node[style={circle,draw, inner sep=2pt, fill=black}, label=below:{\scriptsize$E_{3,2}$}] (5)  at (5,0)  {};
	\node[style={circle,draw, inner sep=2pt, fill=black}, label=below:{\scriptsize$E_{3,1}$}] (6)  at (6,0)  {};
	\node[style={circle,draw, inner sep=2pt, fill=black}, label=below:{\scriptsize$E_{4,2}$}] (8)  at (8,0)  {};
	\node[style={circle,draw, inner sep=2pt, fill=black}, label=below:{\scriptsize$E_{4,1}$}] (9)  at (9,0)  {};
	\node[style={circle,draw, thick, inner sep=2pt, fill=none}, label=below:{\scriptsize$N_1$}] (2)  at (2,0)  {};
	\node[style={circle,draw, thick,inner sep=2pt, fill=none}, label=below:{\scriptsize$N_2$}] (4)  at (4,0)  {};
	\node[style={circle,draw, thick, inner sep=2pt, fill=none}, label=below:{\scriptsize$N_3$}] (7)  at (7,0)  {};
	\node[style={circle,draw, thick, inner sep=2pt, fill=none}, label=below:{\scriptsize$N_4$}] (10)  at (10,0)  {};	

	\draw[thick] (1) to [bend left=60] (2);
    \draw[thick] (1) to [bend left=60] (7);  
    \draw[thick] (1) to [bend left=70] (10);
    \draw[thick] (3) to [bend left=60] (4);
    \draw[thick] (3) to [bend left=60] (7);    
    \draw[thick] (5) to [bend left=60] (7);
    \draw[thick] (6) to [bend left=60] (7);
    \draw[thick] (8) to [bend left=60] (10);
    \draw[thick] (9) to [bend left=60] (10);

    \node[] (1)  at (5,2.9)  {\scriptsize \textcolor{cyan}{$1$}};
    \node[] (2)  at (4,1.9)  {\scriptsize \textcolor{cyan}{$2$}};
    \node[] (3)  at (1.9,0.5)  {\scriptsize \textcolor{cyan}{$3$}};
    \node[] (4)  at (4.2,1.2)  {\scriptsize \textcolor{cyan}{$4$}};    
    \node[] (5)  at (3.9,0.5)  {\scriptsize \textcolor{cyan}{$5$}};  
    \node[] (6)  at (5.2,0.6)  {\scriptsize \textcolor{cyan}{$6$}};  
    \node[] (7)  at (5.8,0.3)  {\scriptsize \textcolor{cyan}{$7$}};
    \node[] (8)  at (8.3,0.8)  {\scriptsize \textcolor{cyan}{$8$}};    
    \node[] (9)  at (8.9,0.34)  {\scriptsize \textcolor{cyan}{$9$}};    
\end{scope}
\end{tikzpicture}
\caption{A maximal clique of routes (left) representing a simplex in the length-framed triangulation of $\calF_{\car(\nu)}$ for $\nu=NENE^2NE^2$.  The bijection $\varphi$ of Lemma~\ref{lem:RouteArcBijection} sends each route to the corresponding arc of the $(I,\overline{J})$-tree on the right. The bijection $\Phi$ in the proof of Theorem~\ref{thm.associahedralTriangulation} sends the maximal clique to the $(I,\overline{J})$-tree.
}
\label{fig.length_clique}
\end{figure}
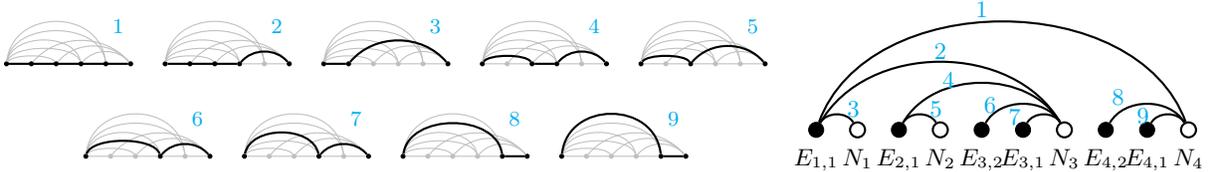

\begin{lemma}\label{lem:coherent}
Let $\prec$ denote the length framing, and let $\varphi$ be the bijection in Lemma~\ref{lem:RouteArcBijection}. 
Two routes $R_{j,i,\ell}$ and $R_{j',i',\ell'}$ in the framed graph $(\car(\nu),\prec)$ are coherent if and only if $\varphi(R_{j,i,\ell})=(E_{j,i},N_{\ell})$ and $\varphi(R_{j',i',\ell'})=(E_{j',i'},N_{\ell'})$ are non-crossing arcs in $\mathcal{A}_\nu$.  
\end{lemma}
\begin{proof}
We can assume without loss of generality that $\ell < \ell'$ (if $\ell=\ell'$, the arcs are non-crossing and the routes are coherent).
There are two ways in which the arcs $\varphi(R_{j,i,\ell})$ and $\varphi(R_{j',i',\ell'})$ can cross: (1) when $j<j'$, or (2) when $j=j'$ with $i'< i$. In both cases we have that $R_{j,i,\ell}$ and $R_{j',i',\ell'}$ are incoherent at the vertex $j'$ (see Figure \ref{fig:phiCoherence}). Conversely, if $R_{j,i,\ell}$ and $R_{j',i',\ell'}$ are incoherent, they must be incoherent at a minimal vertex $j'$. Thus either $j<j'$ or $j=j'$ with $i'<i$, which are precisely the ways in which $\varphi(R_{j,i,\ell})$ and $\varphi(R_{j',i',\ell'})$ can cross.
\end{proof}

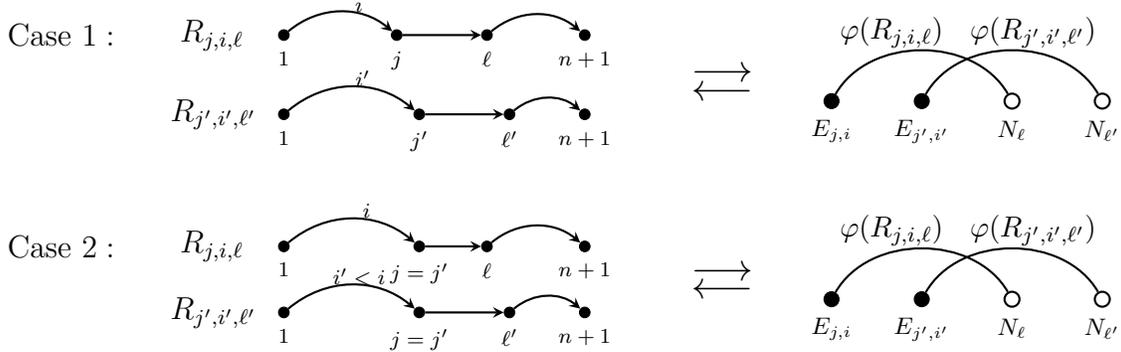
\begin{figure}[ht!]
\begin{tikzpicture}

\begin{scope}[xshift=-20, yshift=-100, scale=0.5]
	\node[] (1)  at (0,0)  {Case $1:$};
\end{scope}
\begin{scope}[xshift = 50, yshift = -100, scale=0.5]
	\node[label=left:{$R_{j,i,\ell}$}] (1)  at (0.5,0)  {};
	
	\vertex[fill,label=below:\tiny{$1$}](a1) at (1,0) {};
	\vertex[fill,label=below:\tiny{$j$}](a4) at (4,0) {};
	\vertex[fill,label=below:\tiny{$\ell$}](a7) at (6.4,0) {};
	\vertex[fill,label=below:\tiny{$n+1$}](a9) at (9,0) {};	

    \node[] (2) at (3,0.8) {{\tiny $i$}};
	
	\draw[-stealth, thick] (a1) to[out=40,in=140] (a4);
	\draw[-stealth, thick] (a4)--(a7);	\draw[-stealth, thick] (a7) to[out=40,in=140] (a9);

\end{scope}
\begin{scope}[xshift = 50, yshift = -130, scale=0.5]
	\node[label=left:{$R_{j',i',\ell'}$}] (1)  at (0.85,0)  {};
	
	\vertex[fill,label=below:\tiny{$1$}](a1) at (1,0) {};
	\vertex[fill,label=below:\tiny{$j'$}](a4) at (4.6,0) {};
	\vertex[fill,label=below:\tiny{$\ell'$}](a7) at (7,0) {};
	\vertex[fill,label=below:\tiny{$n+1$}](a9) at (9,0) {};	

    \node[] (2) at (3.1,0.95) {{\tiny $i'$}};
	
	\draw[-stealth, thick] (a1) to[out=40,in=140] (a4);
	\draw[-stealth, thick] (a4)--(a7);	\draw[-stealth, thick] (a7) to[out=40,in=140] (a9);

\end{scope}
\begin{scope}[xshift=260, yshift=-125, scale=0.4]
	\node[style={circle,draw, inner sep=2pt, fill=black}, label=below:{\scriptsize $E_{j,i}$}] (1)  at (1,0)  {};
	\node[style={circle,draw, inner sep=2pt, thick, fill=none}, label=below:{\scriptsize $N_{\ell}$}] (5)  at (7,0)  {};
	\node[style={circle,draw, inner sep=2pt, fill=black}, label=below:{\scriptsize $E_{j',i'}$}] (3)  at (4,0)  {};
	\node[style={circle,draw, inner sep=2pt, thick, fill=none}, label=below:{\scriptsize $N_{\ell'}$}] (7)  at (10,0)  {};
	
	\draw[thick] (1) to [bend left=60] (5);
    \draw[thick] (3) to [bend left=60] (7);   

	\node[] (1)  at (3,2.3)  {{\small $\varphi(R_{j,i,\ell})$}};
	\node[] (1)  at (7.7,2.3)  {{\small $\varphi(R_{j',i',\ell'})$}};

\end{scope}

\begin{scope}[xshift=230, yshift=-115, scale=0.5]
	\node[] (1)  at (0,0)  {\large $\longrightarrow$};
	\node[] (1)  at (0,-0.5)  {\large $\longleftarrow$};	
\end{scope}


\begin{scope}[xshift=-20, yshift=-180, scale=0.5]
	\node[] (1)  at (0,0)  {Case $2:$};
\end{scope}

\begin{scope}[xshift = 50, yshift = -180, scale=0.5]
	
	\node[label=left:{$R_{j,i,\ell}$}] (1)  at (0.5,0)  {};
	
	\vertex[fill,label=below:\tiny{$1$}](a1) at (1,0) {};
	\vertex[fill,label=below:\tiny{$j=j'$}](a4) at (4.6,0) {};
	\vertex[fill,label=below:\tiny{$\ell$}](a7) at (6.4,0) {};
	\vertex[fill,label=below:\tiny{$n+1$}](a9) at (9,0) {};	

    \node[] (2) at (3.2,0.95) {{\tiny $i$}};
	
	\draw[-stealth, thick] (a1) to[out=40,in=140] (a4);
	\draw[-stealth, thick] (a4)--(a7);	\draw[-stealth, thick] (a7) to[out=40,in=140] (a9);
\end{scope}
\begin{scope}[xshift = 50, yshift = -205, scale=0.5]
	\node[label=left:{$R_{j',i',\ell'}$}] (1)  at (0.85,0)  {};
	
	\vertex[fill,label=below:\tiny{$1$}](a1) at (1,0) {};
	\vertex[fill,label=below:\tiny{$j=j'$}](a4) at (4.6,0) {};
	\vertex[fill,label=below:\tiny{$\ell'$}](a7) at (7,0) {};
	\vertex[fill,label=below:\tiny{$n+1$}](a9) at (9,0) {};	

    \node[] (2) at (3,0.95) {{\tiny $i'<i$}};
	
	\draw[-stealth, thick] (a1) to[out=40,in=140] (a4);
	\draw[-stealth, thick] (a4)--(a7);	\draw[-stealth, thick] (a7) to[out=40,in=140] (a9);
\end{scope}
\begin{scope}[xshift=260, yshift=-200, scale=0.4]
	\node[style={circle,draw, inner sep=2pt, fill=black}, 
	label=below:{\scriptsize $E_{j,i}$}] (1)  at (1,0)  {};
	\node[style={circle,draw, inner sep=2pt, thick, fill=none}, label=below:{\scriptsize  $N_{\ell}$}] (5)  at (7,0)  {};
	\node[style={circle,draw, inner sep=2pt, fill=black}, label=below:{\scriptsize $E_{j',i'}$}] (3)  at (4,0)  {};
	\node[style={circle,draw, inner sep=2pt, thick, fill=none}, label=below:{\scriptsize $N_{\ell'}$}] (7)  at (10,0)  {};
	
	\draw[thick] (1) to [bend left=60] (5);
    \draw[thick] (3) to [bend left=60] (7);   

	\node[] (1)  at (3,2.3)  {{\small $\varphi(R_{j,i,\ell})$}};
	\node[] (1)  at (7.7,2.3)  {{\small $\varphi(R_{j',i',\ell'})$}};
\end{scope}

\begin{scope}[xshift=230, yshift=-190, scale=0.5]
	\node[] (1)  at (0,0)  {\large $\longrightarrow$};
	\node[] (1)  at (0,-0.5)  {\large $\longleftarrow$};	
\end{scope}

\end{tikzpicture}
\caption{The two cases in the proof of Lemma~\ref{lem:coherent}.}
\label{fig:phiCoherence}
\end{figure}

\associahedralthm*
\begin{proof} 
By Lemma~\ref{lem:coherent}, the bijection $\varphi$ in Lemma~\ref{lem:RouteArcBijection} extends to a bijection $\Phi$ from the set of maximal cliques of routes in the length-framed $\car(\nu)$ to the set $\calT_\nu$ of $(I,\overline{J})$-trees determined by $\nu$. 
Two simplices in a DKK triangulation of a flow polytope are adjacent if and only if they differ by a single vertex, that is, if the corresponding maximal cliques differ by a single route. 
Under the bijection $\Phi$, two simplices are adjacent if and only if their corresponding $(I,\overline{J})$-trees differ by a single arc, which by Corollary~\ref{cor:nuTamariHasseDiag} is precisely the description of the Hasse diagram of the $\nu$-Tamari lattice. 
\end{proof}

\begin{example}
Let $\nu= NENE^2NE^2$.  
One example of the bijection $\Phi$ between cliques of routes of $\car(\nu)$ and $(I,\overline{J})$-trees is illustrated in Figure~\ref{fig.length_clique}. 
The dual graph of the length-framed triangulation of $\calF_{\car(\nu)}$ is shown in Figure~\ref{fig.35TamariLattice}.
\end{example}

In \cite{CPS19} Ceballos, Padrol, and Sarmiento introduced the \emph{$(I, \overline{J})$-Tamari complex} $\mathcal{A}_{I,\overline{J}}$ as the flag simplicial complex on $\{(i, \overline{j}) \in I \times \overline{J} \mid i < \overline{j}\}$ whose minimal non-faces are the pairs $\{(i, \overline{j}), (i', \overline{j}')\}$ with $i < i' < j < j'$, that is, the complex on collections of non-crossing arcs of $(I,\overline{J})$-trees. 
The following result is then a corollary of Theorem \ref{thm.associahedralTriangulation}.

\begin{corollary}\label{cor.geometric_realization}
Let $\nu$ be the lattice path from $(0,0)$ to $(b,a)$ associated to the pair $(I,\overline{J})$. 
The length-framed triangulation of $\calF_{\car(\nu)}$ is a geometric realization of the $(I,\overline{J})$-Tamari complex of dimension $a+b$ in $\bbR^{a+b+3}$.
\end{corollary}

\begin{remark}
A graph closely related to $\car(\nu)$ can produce a realization of the  $(I, \overline{J})$-Tamari complex in $\bbR^{a+b+1}$ that is a simple projection of the geometric realization of Corollary \ref{cor.geometric_realization}. The geometric realization of Corollary \ref{cor.geometric_realization} is integrally equivalent to the first of three realizations given in \cite[Theorem 1.1]{CPS19}.
\end{remark}

\subsection{A second description in terms of \texorpdfstring{$\nu$-}-trees}

We conclude this section by indexing simplices in the length-framed triangulation of $\calF_{\car(\nu)}$ by $\nu$-trees, which in terms of the lattice path $\nu$ is an analogous counterpart to the $\nu$-Dyck path description of the planar-framed triangulation in Section \ref{sec.planarframed}. The vertices (routes) of $\calF_{\car(\nu)}$ are encoded by the lattice points in $\calL_\nu$, with the lattice points in each $\nu$-tree corresponding to a maximal simplex in the length-framed triangulation. Furthermore, it then follows that two simplices are adjacent in the length-framed triangulation if their corresponding $\nu$-trees differ by a rotation.

\begin{lemma}\label{lem:latticeptBijection1}
Let $\nu$ be a lattice path from $(0,0)$ to $(b,a)$.
There exists a (first) bijection $\theta:\calR_\nu \to \calL_\nu$ between the set $\calR_\nu$ of routes  in $\car(\nu)$ and the set $\calL_\nu$ of lattice points  in the rectangle defined by $(0,0)$ and $(b,a)$ that lie weakly above $\nu$. 
\end{lemma}
\begin{proof}

The map $\varphi: \calR_\nu \to \calA_\nu$ in Lemma~\ref{lem:RouteArcBijection} gives a bijection. Let $\gamma:\calA_\nu \to \calL_\nu$ be given by $(E_x,N_y) \mapsto (x,y)$. For any arc $(E_x,N_y)$, since $E_x$ appears before $N_y$ in $\overline{\nu}$, we have $\gamma(E_x,N_y) =(x,y)\in \calL_\nu$. The inverse $\gamma\inv$ is well-defined since if $(x,y) \in \calL_\nu$, then $N_y$ is preceded by at least $x$ $E$ steps in $\overline{\nu}$, and hence $(E_x,N_y) \in \calA_\nu$. Now $\theta :=\gamma\circ \varphi$ is the desired bijection. 
\end{proof}

A second bijection between the routes $\calR_\nu$ of $\car(\nu)$ and the lattice points $\calL_\nu$ is given in Lemma~\ref{lem.bijectionlatticepoints}.
The present bijection $\theta$ leads to a characterization of the routes which appear in every simplex of the length-framed triangulation of $\calF_{\car(\nu)}$.

Recall from Lemma~\ref{lem:coherent} that two routes in $\calR_\nu$ are coherent if and only if their corresponding arcs in $\calA_\nu$ are non-crossing. Since the non-crossing condition for arcs in $(I,\overline{J})$-trees translates to the $\nu$-compatibility condition of $\nu$-trees, we have that routes in $\calR_\nu$ are coherent if and only if their corresponding lattice points via $\theta$ are $\nu$-compatible. 
If we associate each lattice point in $\calL_\nu$ with the corresponding route in $\car(\nu)$ via the bijection $\theta$ in Lemma~\ref{lem:latticeptBijection1}, then the lattice points in a $\nu$-tree correspond to a maximal clique of routes. 
Two adjacent simplices in the length-framed triangulation of $\calF_{\car(\nu)}$ differ by a single vertex, and the corresponding $\nu$-trees differ by a single lattice point via a rotation. 
Note that a $\nu$-tree will always contain the root $(0,a)$, the valleys of the lattice path $\nu$, along with each initial point of any initial $N$ steps of $\nu$ and each terminal point of any terminal $E$ steps of $\nu$. 
These points correspond to the routes which are coherent with all other routes in the length-framing, and thus appear in every simplex of the length-framed triangulation. 

In the example in Figure~\ref{fig.length_clique_and_nu_tree}, the routes which appear in every simplex of the length-framed triangulation of $\calF_{\car(\nu)}$ are labeled $1,3,5,7,8$ and $9$. 

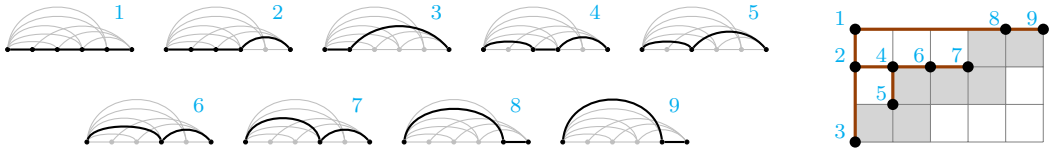
\begin{figure}[ht!]
\centering
\begin{tikzpicture}
\begin{scope}[xshift=0, yshift=0, scale=.33]
    \node[]  at (5.5,1.5)  {\tiny \textcolor{cyan}{$1$}};
    \cvx[color=black](a1) at (1,0) {}; \cvx[color=black](a2) at (2,0) {}; 
	\cvx[color=black](a3) at (3,0) {}; \cvx[color=black](a4) at (4,0) {}; 
	\cvx[color=black](a5) at (5,0) {}; \cvx[color=black](a6) at (6,0) {};
	\draw[color=black, thick] (a1)--(a6);
	\draw[color=black!25] (a1) .. controls (1.25,.4) and (2.75,.4) .. (a3);	
    \draw[color=black!25] (a1) .. controls (1.25,0.8) and (3.75,0.8) .. (a4);
    \draw[color=black!25] (a1) .. controls (1.25,1.25) and (3.75,1.25) .. (a4);    
	\draw[color=black!25] (a1) .. controls (1.25, 1.75) and (4.75,1.75) .. (a5);
	\draw[color=black!25] (a1) .. controls (1.25, 2.25) and (4.75,2.25) .. (a5);
	\draw[color=black!25] (a4) to[out=50,in=130] (a6);
	\draw[color=black!25] (a3) to[out=50,in=130] (a6);
	\draw[color=black!25] (a2) to[out=50,in=130] (a6);
\end{scope}
\begin{scope}[xshift=60, yshift=0, scale=.33]
	\node[]  at (5.5,1.5)  {\tiny \textcolor{cyan}{$2$}};
	\cvx[color=black](a1) at (1,0) {}; 
	\cvx[color=black](a2) at (2,0) {}; 
	\cvx[color=black](a3) at (3,0) {}; 
	\cvx[color=black](a4) at (4,0) {}; 
	\cvx[color=black!25](a5) at (5,0) {};
	\cvx[color=black](a6) at (6,0) {};
	\draw[color=black!25] (a4)--(a6);
	\draw[color=black!25] (a1) .. controls (1.25,.4) and (2.75,.4) .. (a3);	
    \draw[color=black!25] (a1) .. controls (1.25,0.8) and (3.75,0.8) .. (a4);
    \draw[color=black!25] (a1) .. controls (1.25,1.25) and (3.75,1.25) .. (a4);    
	\draw[color=black!25] (a1) .. controls (1.25, 1.75) and (4.75,1.75) .. (a5);
	\draw[color=black!25] (a1) .. controls (1.25, 2.25) and (4.75,2.25) .. (a5);
	\draw[color=black!25] (a3) to[out=50,in=130] (a6);
	\draw[color=black!25] (a2) to[out=50,in=130] (a6);
	\draw[color=black, thick] (a4) to[out=50,in=130] (a6);
	\draw[color=black, thick] (a1)--(a4);
\end{scope}
\begin{scope}[xshift=120, yshift=0, scale=.33]
	\node[]  at (5.5,1.5)  {\tiny \textcolor{cyan}{$3$}};
	\cvx[color=black](a1) at (1,0) {}; 
	\cvx[color=black](a2) at (2,0) {}; 
	\cvx[color=black!25](a3) at (3,0) {}; 
	\cvx[color=black!25](a4) at (4,0) {}; 
	\cvx[color=black!25](a5) at (5,0) {};
	\cvx[color=black](a6) at (6,0) {};
	\draw[color=black!25] (a2)--(a6);
	\draw[color=black!25] (a1) .. controls (1.25,.4) and (2.75,.4) .. (a3);	
    \draw[color=black!25] (a1) .. controls (1.25,0.8) and (3.75,0.8) .. (a4);
    \draw[color=black!25] (a1) .. controls (1.25,1.25) and (3.75,1.25) .. (a4);    
	\draw[color=black!25] (a1) .. controls (1.25, 1.75) and (4.75,1.75) .. (a5);
	\draw[color=black!25] (a1) .. controls (1.25, 2.25) and (4.75,2.25) .. (a5);
	\draw[color=black!25] (a4) to[out=50,in=130] (a6);
	\draw[color=black!25] (a3) to[out=50,in=130] (a6);
	\draw[color=black!25] (a2) to[out=50,in=130] (a6);
    \draw[color=black, thick] (a1)--(a2);
	\draw[color=black, thick] (a2) to[out=50,in=130] (a6);
\end{scope}
\begin{scope}[xshift=180, yshift=0, scale=.33]
	\node[]  at (5.5,1.5)  {\tiny \textcolor{cyan}{$4$}};
	\cvx[color=black](a1) at (1,0) {}; 
	\cvx[color=black!25](a2) at (2,0) {}; 
	\cvx[color=black](a3) at (3,0) {}; 
	\cvx[color=black](a4) at (4,0) {}; 
	\cvx[color=black!25](a5) at (5,0) {};
	\cvx[color=black](a6) at (6,0) {};
	\draw[color=black!25] (a1)--(a3);
	\draw[color=black!25] (a4)--(a6);
    \draw[color=black!25] (a1) .. controls (1.25,0.8) and (3.75,0.8) .. (a4);
    \draw[color=black!25] (a1) .. controls (1.25,1.25) and (3.75,1.25) .. (a4);    
	\draw[color=black!25] (a1) .. controls (1.25, 1.75) and (4.75,1.75) .. (a5);
	\draw[color=black!25] (a1) .. controls (1.25, 2.25) and (4.75,2.25) .. (a5);
	\draw[color=black!25] (a4) to[out=50,in=130] (a6);
	\draw[color=black!25] (a3) to[out=50,in=130] (a6);
	\draw[color=black!25] (a2) to[out=50,in=130] (a6);
	\draw[color=black, thick] (a1) .. controls (1.25,.4) and (2.75,.4) .. (a3);
	\draw[color=black, thick] (a3)--(a4);
	\draw[color=black, thick] (a4) to[out=50,in=130] (a6);
\end{scope}
\begin{scope}[xshift=240, yshift=0, scale=.33]
	\node[]  at (5.5,1.5)  {\tiny \textcolor{cyan}{$5$}};
	\cvx[color=black](a1) at (1,0) {}; 
	\cvx[color=black!25](a2) at (2,0) {}; 
	\cvx[color=black](a3) at (3,0) {}; 
	\cvx[color=black!25](a4) at (4,0) {}; 
	\cvx[color=black!25](a5) at (5,0) {};
	\cvx[color=black](a6) at (6,0) {};
	\draw[color=black!25] (a1)--(a3);
	\draw[color=black!25] (a3)--(a6);
    \draw[color=black!25] (a1) .. controls (1.25,0.8) and (3.75,0.8) .. (a4);
    \draw[color=black!25] (a1) .. controls (1.25,1.25) and (3.75,1.25) .. (a4);    
	\draw[color=black!25] (a1) .. controls (1.25, 1.75) and (4.75,1.75) .. (a5);
	\draw[color=black!25] (a1) .. controls (1.25, 2.25) and (4.75,2.25) .. (a5);
	\draw[color=black!25] (a4) to[out=50,in=130] (a6);
	\draw[color=black!25] (a2) to[out=50,in=130] (a6);
    \draw[color=black, thick] (a1) .. controls (1.25,.4) and (2.75,.4) .. (a3);	
	\draw[color=black, thick] (a3) to[out=50,in=130] (a6);
\end{scope}
\begin{scope}[xshift=30, yshift=-35, scale=.33]
	\node[]  at (5.5,1.5)  {\tiny \textcolor{cyan}{$6$}};
	\cvx[color=black](a1) at (1,0) {}; 
	\cvx[color=black!25](a2) at (2,0) {}; 
	\cvx[color=black!25](a3) at (3,0) {}; 
	\cvx[color=black](a4) at (4,0) {}; 
	\cvx[color=black!25](a5) at (5,0) {};
	\cvx[color=black](a6) at (6,0) {};
	\draw[color=black!25] (a1)--(a4); \draw[color=black!25] (a4)--(a6);
	\draw[color=black!25] (a1) .. controls (1.25,.4) and (2.75,.4) .. (a3);	
    \draw[color=black!25] (a1) .. controls (1.25,1.25) and (3.75,1.25) .. (a4);
	\draw[color=black!25] (a1) .. controls (1.25, 1.75) and (4.75,1.75) .. (a5);
	\draw[color=black!25] (a1) .. controls (1.25, 2.25) and (4.75,2.25) .. (a5);
	\draw[color=black!25] (a4) to[out=50,in=130] (a6);
	\draw[color=black!25] (a3) to[out=50,in=130] (a6);
	\draw[color=black!25] (a2) to[out=50,in=130] (a6);
    \draw[color=black, thick] (a1) .. controls (1.25,0.8) and (3.75,0.8) .. (a4);
	\draw[color=black, thick] (a4) to[out=50,in=130] (a6);
\end{scope}
\begin{scope}[xshift=90, yshift=-35, scale=.33]
	\node[]  at (5.5,1.5)  {\tiny \textcolor{cyan}{$7$}};
	\cvx[color=black](a1) at (1,0) {}; 
	\cvx[color=black!25](a2) at (2,0) {}; 
	\cvx[color=black!25](a3) at (3,0) {}; 
	\cvx[color=black](a4) at (4,0) {}; 
	\cvx[color=black!25](a5) at (5,0) {};
	\cvx[color=black](a6) at (6,0) {};
	\draw[color=black!25] (a1)--(a4);
	\draw[color=black!25] (a4)--(a6);
	\draw[color=black!25] (a1) .. controls (1.25,.4) and (2.75,.4) .. (a3);	
    \draw[color=black!25] (a1) .. controls (1.25,0.8) and (3.75,0.8) .. (a4);
    
	\draw[color=black!25] (a1) .. controls (1.25, 1.75) and (4.75,1.75) .. (a5);
	\draw[color=black!25] (a1) .. controls (1.25, 2.25) and (4.75,2.25) .. (a5);
	\draw[color=black!25] (a3) to[out=50,in=130] (a6);
	\draw[color=black!25] (a2) to[out=50,in=130] (a6);
	\draw[color=black, thick] (a4) to[out=50,in=130] (a6);    
    \draw[color=black, thick] (a1) .. controls (1.25,1.25) and (3.75,1.25) .. (a4);
\end{scope}
\begin{scope}[xshift=150, yshift=-35, scale=.33]
	\node[]  at (5.5,1.5)  {\tiny \textcolor{cyan}{$8$}};
	\cvx[color=black](a1) at (1,0) {}; 
	\cvx[color=black!25](a2) at (2,0) {}; 
	\cvx[color=black!25](a3) at (3,0) {}; 
	\cvx[color=black!25](a4) at (4,0) {}; 
	\cvx[color=black](a5) at (5,0) {};
	\cvx[color=black](a6) at (6,0) {};
	\draw[color=black!25] (a1)--(a5);
	\draw[color=black!25] (a1) .. controls (1.25,.4) and (2.75,.4) .. (a3);	
    \draw[color=black!25] (a1) .. controls (1.25,0.8) and (3.75,0.8) .. (a4);
    \draw[color=black!25] (a1) .. controls (1.25,1.25) and (3.75,1.25) .. (a4);    
    \draw[color=black!25] (a1) .. controls (1.25, 2.25) and (4.75,2.25) .. (a5);
	\draw[color=black!25] (a4) to[out=50,in=130] (a6);
	\draw[color=black!25] (a3) to[out=50,in=130] (a6);
	\draw[color=black!25] (a2) to[out=50,in=130] (a6);
	\draw[color=black, thick] (a1) .. controls (1.25, 1.75) and (4.75,1.75) .. (a5);
	\draw[color=black, thick] (a5)--(a6);
\end{scope}
\begin{scope}[xshift=210, yshift=-35, scale=.33]
	\node[]  at (5.5,1.5)  {\tiny \textcolor{cyan}{$9$}};
	\cvx[color=black](a1) at (1,0) {}; 
	\cvx[color=black!25](a2) at (2,0) {}; 
	\cvx[color=black!25](a3) at (3,0) {}; 
	\cvx[color=black!25](a4) at (4,0) {}; 
	\cvx[color=black](a5) at (5,0) {};
	\cvx[color=black](a6) at (6,0) {};
	\draw[color=black!25] (a1)--(a5);
	\draw[color=black!25] (a1) .. controls (1.25,.4) and (2.75,.4) .. (a3);	
    \draw[color=black!25] (a1) .. controls (1.25,0.8) and (3.75,0.8) .. (a4);
    \draw[color=black!25] (a1) .. controls (1.25,1.25) and (3.75,1.25) .. (a4);    
	\draw[color=black!25] (a1) .. controls (1.25, 1.75) and (4.75,1.75) .. (a5);
	\draw[color=black!25] (a4) to[out=50,in=130] (a6);
	\draw[color=black!25] (a3) to[out=50,in=130] (a6);
	\draw[color=black!25] (a2) to[out=50,in=130] (a6);
	\draw[color=black, thick] (a1) .. controls (1.25, 2.25) and (4.75,2.25) .. (a5);
	\draw[color=black, thick] (a5)--(a6);
\end{scope}

\begin{scope}[xshift=330, yshift=-35, scale=0.5]
    \draw[fill, color=gray!33] (0,0) rectangle (1,1); 
    \draw[fill, color=gray!33] (1,1) rectangle (3,2);
    \draw[fill, color=gray!33] (3,2) rectangle (5,3);
    \draw[fill, color=gray!33] (1,0) rectangle (2,1);
    \draw[fill, color=gray!33] (3,1) rectangle (4,2);  
    \draw[very thin, color=gray!100] (0,0) grid (5,3);  
	
	\draw[very thick, color=RawSienna] (0,0)-- (0,3)--(5,3);
	\draw[very thick, color=RawSienna] (1,1)--(1,2)--(0,2);
	\draw[very thick, color=RawSienna] (2,2)--(3,2);
    \draw[very thick, color=RawSienna] (2,2)--(1,2);

	\node[circle,color=black,fill=black, inner sep=1.5pt] (r) at (0,3) {};
	\node[circle,color=black,fill=black, inner sep=1.5pt] (r) at (2,2) {};
	\node[circle,color=black,fill=black, inner sep=1.5pt] (r) at (0,2) {};
	\node[circle,color=black,fill=black, inner sep=1.5pt] (r) at (0,0) {};
	\node[circle,color=black,fill=black, inner sep=1.5pt] (r) at (1,1) {};	
	\node[circle,color=black,fill=black, inner sep=1.5pt] (r) at (1,2) {};	
	\node[circle,color=black,fill=black, inner sep=1.5pt] (r) at (3,2) {};
	\node[circle,color=black,fill=black, inner sep=1.5pt] (r) at (4,3) {};	
	\node[circle,color=black,fill=black, inner sep=1.5pt] (r) at (5,3) {};

    \node[]  at (-.4,3.3)  {\tiny \textcolor{cyan}{$1$}};
    \node[]  at (-.4,2.3)  {\tiny \textcolor{cyan}{$2$}};
    \node[]  at (-.4,0.3)  {\tiny \textcolor{cyan}{$3$}};
    \node[]  at (.7,2.3)  {\tiny \textcolor{cyan}{$4$}};
    \node[]  at (0.7,1.3)  {\tiny \textcolor{cyan}{$5$}};
    \node[]  at (1.7,2.3)  {\tiny \textcolor{cyan}{$6$}};
    \node[]  at (2.7,2.3)  {\tiny \textcolor{cyan}{$7$}};
    \node[]  at (3.7,3.3)  {\tiny \textcolor{cyan}{$8$}};
    \node[]  at (4.7,3.3)  {\tiny \textcolor{cyan}{$9$}};
\end{scope}
\end{tikzpicture}
\caption{A maximal clique of routes (left) representing a simplex in the length-framed triangulation of $\calF_{\car(\nu)}$ for $\nu=NENE^2NE^2$. 
The bijection $\theta$ of Lemma~\ref{lem:latticeptBijection1} sends each route to a lattice point in the corresponding $\nu$-tree (right).
}
\label{fig.length_clique_and_nu_tree}
\end{figure}

\section{The planar-framed triangulation} 
\label{sec.planarframed}
The goal of this section is to show that the flow polytope $\calF_{\car(\nu)}$ has a regular unimodular triangulation whose dual graph structure is given by the Hasse diagram of a principal order ideal $I(\nu)$ in Young's lattice.
A consequence of this is we can construct a family of posets $Q_\nu$ so that the dual graph structure of the canonical triangulation of the order polytope $\calO(Q_\nu)$ is also $I(\nu)$.

\subsection{Principal order ideals in Young's lattice}
Recall that {\em Young's lattice} $Y$ is the poset on integer partitions $\lambda$ with covering relations $\lambda \lessdot \lambda'$ if $\lambda$ is obtained from $\lambda'$ by removing one corner box of $\lambda'$. 
Note that a lattice path $\nu$ in the rectangular grid defined by $(0,0)$ to $(b,a)$ defines a partition $\lambda(\nu)=(\lambda_1,\ldots, \lambda_a)$ by letting $\lambda_k= b - \sum_{i=a-k+1}^a \nu_i$ for $k=1,\ldots, a$.
The Young diagram for $\lambda(\nu)$ may be visualized as the region within the rectangle from $(0,0)$ to $(b,a)$ which lies NW of $\nu$.
For example in Figure~\ref{fig.numcargraph}, $\nu=NE^2NENNE^3NE$ defines the partition $\lambda(\nu)=(6,3,3,2)$.
An {\em order ideal} of a poset $P$ is a subset $I\subseteq P$ with the property that if $x\in I$ and $y\leq x$, then $y\in I$. An ideal is said to be {\em principal} if it has a single maximal element $x\in P$, and such an ideal will be denoted by $I(x)$. 

If $\mu$ is a $\nu$-Dyck path, then it lies weakly above the path $\nu$ and so $\mu$ can be identified with a partition $\lambda(\mu)$ that is contained in $\lambda(\nu)$. 
Thus there is a one-to-one correspondence between the set of $\nu$-Dyck paths with the set of elements in the order ideal $I(\nu):= I(\lambda(\nu))$ in $Y$.
Under this correspondence, in terms of $\nu$-Dyck paths, a path $\pi$ covers a path $\mu$ if and only if $\pi$ can be obtained from $\mu$ by replacing a consecutive $NE$ pair by a $EN$ pair. 
See the right side of Figure~\ref{fig.35} for an example of $I(\nu)$ with $\nu=NENE^2NE^2$.

\subsection{The triangulation}

\begin{definition}\label{def.planarframing}
Let $G$ be a planar graph that affords a planar embedding in the plane such that if vertex $i$ is at the coordinates $(x_i,y_i)$, then $x_i<x_j$ for all $i<j$. 
This leads to natural orderings $(\prec_{\inedge(i)}, \prec_{\outedge(i)})$ at every inner vertex $i$ of $G$ as follows:
with respect to the planar embedding of $G$, the incoming edges at the vertex $i$ are ordered in increasing order from the top to the bottom, and the same for the outgoing edges from the vertex $i$.
This is the {\em planar framing} for $G$. 
\end{definition}

It is clear that the graphs $\car(\nu)$ have a planar embedding with the properties of Definition~\ref{def.planarframing} if it is embedded so that the path $1,\ldots, n+1$ lies on the $x$-axis.
Figure~\ref{fig.two_framings_ex} gives an example of the planar framing of $\car(\nu)$ with $\nu=NE^2NENNE^3NE$.

\begin{lemma}\label{lem.bijectionlatticepoints}
Let $\nu$ be a lattice path from $(0,0)$ to $(b,a)$.
There exists a (second) bijection $\psi: \calR_\nu \rightarrow \calL_\nu$ between the set $\calR_\nu$ of routes in $\car(\nu)$ and the set $\calL_\nu$ of lattice points  in the rectangle defined by $(0,0)$ and $(b,a)$ that lie weakly above $\nu$.
\end{lemma}
\begin{proof} We fix an embedding of $\car(\nu)$ in the plane so that the path $1,2,\ldots, n+1$ lies on the $x$-axis.
Define a map $\psi: \calR_\nu \rightarrow \calL_\nu$ by $\psi(R) = (j,\ell)$, where $j$ is the number of bounded faces of $\car(\nu)$ that lie below $R$ and above the $x$-axis, and $\ell$ is the number of bounded faces that lie below $R$ and the $x$-axis.
See Figure~\ref{fig.planar_clique} for an example.

To see that $\psi$ is well-defined, first note that any planar embedding of $\car(\nu)$ has $m-n=a+b$ bounded faces, by Euler's formula.
In particular, the fixed planar embedding of $\car(\nu)$ has $a$ bounded faces below the $x$-axis and $b$ bounded faces above the $x$-axis, so $0\leq j\leq b$ and $0\leq \ell\leq a$.
To show that the lattice point $\psi(R)=(j,\ell)$ lies weakly above $\nu$, we must show that $0\leq j \leq \nu_1+\cdots+\nu_\ell$ for a fixed $0\leq \ell \leq a$. 
If $\psi(R)=(j,\ell)$, this means that the last edge of the route $R$ is $(\ell+2,n+1)$.  
By counting the in-degrees of the vertices $3, 4, \ldots, \ell+2$ in $\car(\nu)$, we see that there are at most $\nu_1+\cdots+\nu_{\ell}$ edges in $\car(\nu)$ of the form $(1,j)$ embedded above the $x$-axis for $3\leq j\leq \ell+2$.
Consequently, there are at most $\nu_1+\cdots+\nu_{\ell}$ bounded regions which can lie below the route $R$ and above the $x$-axis.
Therefore, $\psi$ is well-defined.

To see that $\psi$ is invertible, let $(j,\ell) \in \calL_\nu$ so that $ 0 \leq j \leq \nu_1+\cdots+\nu_\ell$.
Then there exists a unique $1<k \leq \ell+2$ such that there are $j$ bounded faces of the embedded $\car(\nu)$ which lie between the edge $(1,k)$ and the $x$-axis.
Let $R$ be the route whose first edge is $(1,k)$ and last edge is $(\ell+2,n+1)$ (recall from Section~\ref{sec.lengthframed} that every route in $\car(\nu)$ is completely characterized by these two edges).  
Then $\psi(R)=(j,\ell)$ as required, and therefore, $\psi$ is a bijection.
\end{proof}

The above bijection $\psi$ leads to a characterization of the routes which appear in every simplex of the planar-framed triangulation of $\calF_{\car(\nu)}$.
As we will see in Theorem~\ref{thm.roottriangulation}, the bijection $\psi$ extends to a bijection $\Psi$ in which a maximal clique of routes in the planar framing of $\car(\nu)$ correspond to the collection of lattice points in a $\nu$-Dyck path.
A $\nu$-Dyck path always contains the lattice points of any initial $N$ steps of $\nu$ and any terminal $E$ steps of $\nu$. Hence, under this bijection, these points correspond to the routes which are coherent with all other routes, and thus appear in every simplex of the planar-framed triangulation.

For example in Figure~\ref{fig.planar_clique}, the routes which appear in every simplex of the planar-framed triangulation of $\calF_{\car(\nu)}$ are labeled $1,2,7,8$ and $9$. 

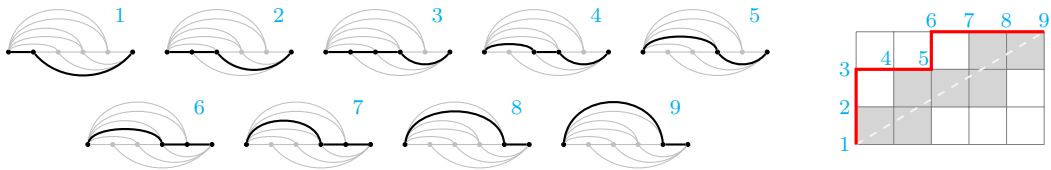
\begin{figure}[ht!]
\centering
\begin{tikzpicture}
\begin{scope}[xshift=0, yshift=0, scale=.33]
    \node[]  at (5.5,1.5)  {\tiny \textcolor{cyan}{$1$}};
	\cvx[color=black](a1) at (1,0) {}; 
	\cvx[color=black](a2) at (2,0) {}; 
	\cvx[color=black!25](a3) at (3,0) {}; 
	\cvx[color=black!25](a4) at (4,0) {}; 
	\cvx[color=black!25](a5) at (5,0) {};
	\cvx[color=black](a6) at (6,0) {};
	\draw[color=black!25] (a2)--(a6);
	\draw[color=black!25] (a1) .. controls (1.25,.4) and (2.75,.4) .. (a3);	
    \draw[color=black!25] (a1) .. controls (1.25,0.8) and (3.75,0.8) .. (a4);
    \draw[color=black!25] (a1) .. controls (1.25,1.25) and (3.75,1.25) .. (a4);    
	\draw[color=black!25] (a1) .. controls (1.25, 1.75) and (4.75,1.75) .. (a5);
	\draw[color=black!25] (a1) .. controls (1.25, 2.25) and (4.75,2.25) .. (a5);
	\draw[color=black!25] (a4) to[out=-50,in=230] (a6);
	\draw[color=black!25] (a3) to[out=-50,in=230] (a6);
    \draw[color=black, thick] (a1)--(a2);
	\draw[color=black, thick] (a2) to[out=-50,in=230] (a6);
\end{scope}
\begin{scope}[xshift=60, yshift=0, scale=.33]
    \node[]  at (5.5,1.5)  {\tiny \textcolor{cyan}{$2$}};
    \cvx[color=black](a1) at (1,0) {}; 
	\cvx[color=black](a2) at (2,0) {}; 
	\cvx[color=black](a3) at (3,0) {}; 
	\cvx[color=black!25](a4) at (4,0) {}; 
	\cvx[color=black!25](a5) at (5,0) {};
	\cvx[color=black](a6) at (6,0) {};
	\draw[color=black!25] (a3)--(a6);
	\draw[color=black!25] (a1) .. controls (1.25,.4) and (2.75,.4) .. (a3);	
    \draw[color=black!25] (a1) .. controls (1.25,0.8) and (3.75,0.8) .. (a4);
    \draw[color=black!25] (a1) .. controls (1.25,1.25) and (3.75,1.25) .. (a4);    
	\draw[color=black!25] (a1) .. controls (1.25, 1.75) and (4.75,1.75) .. (a5);
	\draw[color=black!25] (a1) .. controls (1.25, 2.25) and (4.75,2.25) .. (a5);
	\draw[color=black!25] (a4) to[out=-50,in=230] (a6);
	\draw[color=black!25] (a2) to[out=-50,in=230] (a6);
    \draw[color=black, thick] (a1)--(a3);
	\draw[color=black, thick] (a3) to[out=-50,in=230] (a6);
\end{scope}
\begin{scope}[xshift=120, yshift=0, scale=.33]
    \node[]  at (5.5,1.5)  {\tiny \textcolor{cyan}{$3$}};
    \cvx[color=black](a1) at (1,0) {}; 
	\cvx[color=black](a2) at (2,0) {}; 
	\cvx[color=black](a3) at (3,0) {}; 
	\cvx[color=black](a4) at (4,0) {}; 
	\cvx[color=black!25](a5) at (5,0) {};
	\cvx[color=black](a6) at (6,0) {};
	\draw[color=black!25] (a4)--(a6);
	\draw[color=black!25] (a1) .. controls (1.25,.4) and (2.75,.4) .. (a3);	
    \draw[color=black!25] (a1) .. controls (1.25,0.8) and (3.75,0.8) .. (a4);
    \draw[color=black!25] (a1) .. controls (1.25,1.25) and (3.75,1.25) .. (a4);    
	\draw[color=black!25] (a1) .. controls (1.25, 1.75) and (4.75,1.75) .. (a5);
	\draw[color=black!25] (a1) .. controls (1.25, 2.25) and (4.75,2.25) .. (a5);
	\draw[color=black!25] (a3) to[out=-50,in=230] (a6);
	\draw[color=black!25] (a2) to[out=-50,in=230] (a6);
    \draw[color=black, thick] (a1)--(a4);
	\draw[color=black, thick] (a4) to[out=-50,in=230] (a6);
\end{scope}
\begin{scope}[xshift=180, yshift=0, scale=.33]
    \node[]  at (5.5,1.5)  {\tiny \textcolor{cyan}{$4$}};
    \cvx[color=black](a1) at (1,0) {}; 
	\cvx[color=black!25](a2) at (2,0) {}; 
	\cvx[color=black](a3) at (3,0) {}; 
	\cvx[color=black](a4) at (4,0) {}; 
	\cvx[color=black!25](a5) at (5,0) {};
	\cvx[color=black](a6) at (6,0) {};
	\draw[color=black!25] (a1)--(a3);
	\draw[color=black!25] (a4)--(a6);
		
    \draw[color=black!25] (a1) .. controls (1.25,0.8) and (3.75,0.8) .. (a4);
    \draw[color=black!25] (a1) .. controls (1.25,1.25) and (3.75,1.25) .. (a4);    
	\draw[color=black!25] (a1) .. controls (1.25, 1.75) and (4.75,1.75) .. (a5);
	\draw[color=black!25] (a1) .. controls (1.25, 2.25) and (4.75,2.25) .. (a5);
	\draw[color=black!25] (a3) to[out=-50,in=230] (a6);
	\draw[color=black!25] (a2) to[out=-50,in=230] (a6);
    \draw[color=black, thick] (a1) .. controls (1.25,.4) and (2.75,.4) .. (a3);
    \draw[color=black, thick] (a3)--(a4);
	\draw[color=black, thick] (a4) to[out=-50,in=230] (a6);
\end{scope}
\begin{scope}[xshift=240, yshift=0, scale=.33]
    \node[]  at (5.5,1.5)  {\tiny \textcolor{cyan}{$5$}};
    \cvx[color=black](a1) at (1,0) {}; 
	\cvx[color=black!25](a2) at (2,0) {}; 
	\cvx[color=black!25](a3) at (3,0) {}; 
	\cvx[color=black](a4) at (4,0) {}; 
	\cvx[color=black!25](a5) at (5,0) {};
	\cvx[color=black](a6) at (6,0) {};
	\draw[color=black!25] (a1)--(a4);
	\draw[color=black!25] (a4)--(a6);
	\draw[color=black!25] (a1) .. controls (1.25,.4) and (2.75,.4) .. (a3);	
    \draw[color=black!25] (a1) .. controls (1.25,1.25) and (3.75,1.25) .. (a4);    
	\draw[color=black!25] (a1) .. controls (1.25, 1.75) and (4.75,1.75) .. (a5);
	\draw[color=black!25] (a1) .. controls (1.25, 2.25) and (4.75,2.25) .. (a5);
	\draw[color=black!25] (a3) to[out=-50,in=230] (a6);
	\draw[color=black!25] (a2) to[out=-50,in=230] (a6);
    \draw[color=black, thick] (a1) .. controls (1.25,0.8) and (3.75,0.8) .. (a4);
    \draw[color=black, thick] (a4) to[out=-50,in=230] (a6);
\end{scope}
\begin{scope}[xshift=30, yshift=-35, scale=.33]
    \node[]  at (5.5,1.5)  {\tiny \textcolor{cyan}{$6$}};
    \cvx[color=black](a1) at (1,0) {}; 
	\cvx[color=black!25](a2) at (2,0) {}; 
	\cvx[color=black!25](a3) at (3,0) {}; 
	\cvx[color=black](a4) at (4,0) {}; 
	\cvx[color=black](a5) at (5,0) {};
	\cvx[color=black](a6) at (6,0) {};
	\draw[color=black!25] (a1)--(a4);
	\draw[color=black!25] (a1) .. controls (1.25,.4) and (2.75,.4) .. (a3);	
    \draw[color=black!25] (a1) .. controls (1.25,1.25) and (3.75,1.25) .. (a4);    
	\draw[color=black!25] (a1) .. controls (1.25, 1.75) and (4.75,1.75) .. (a5);
	\draw[color=black!25] (a1) .. controls (1.25, 2.25) and (4.75,2.25) .. (a5);
	\draw[color=black!25] (a4) to[out=-50,in=230] (a6);
	\draw[color=black!25] (a3) to[out=-50,in=230] (a6);
	\draw[color=black!25] (a2) to[out=-50,in=230] (a6);
    \draw[color=black, thick] (a4)--(a6);
    \draw[color=black, thick] (a1) .. controls (1.25,0.8) and (3.75,0.8) .. (a4);
\end{scope}
\begin{scope}[xshift=90, yshift=-35, scale=.33]
    \node[]  at (5.5,1.5)  {\tiny \textcolor{cyan}{$7$}};
    \cvx[color=black](a1) at (1,0) {}; 
	\cvx[color=black!25](a2) at (2,0) {}; 
	\cvx[color=black!25](a3) at (3,0) {}; 
	\cvx[color=black](a4) at (4,0) {}; 
	\cvx[color=black](a5) at (5,0) {};
	\cvx[color=black](a6) at (6,0) {};
	\draw[color=black!25] (a1)--(a4);
	\draw[color=black!25] (a1) .. controls (1.25,.4) and (2.75,.4) .. (a3);	
    \draw[color=black!25] (a1) .. controls (1.25,0.8) and (3.75,0.8) .. (a4);
	\draw[color=black!25] (a1) .. controls (1.25, 1.75) and (4.75,1.75) .. (a5);
	\draw[color=black!25] (a1) .. controls (1.25, 2.25) and (4.75,2.25) .. (a5);
	\draw[color=black!25] (a4) to[out=-50,in=230] (a6);
	\draw[color=black!25] (a3) to[out=-50,in=230] (a6);
	\draw[color=black!25] (a2) to[out=-50,in=230] (a6);
    \draw[color=black, thick] (a1) .. controls (1.25,1.25) and (3.75,1.25) .. (a4);
    \draw[color=black, thick] (a4)--(a6);
\end{scope}
\begin{scope}[xshift=150, yshift=-35, scale=.33]
    \node[]  at (5.5,1.5)  {\tiny \textcolor{cyan}{$8$}};
    \cvx[color=black](a1) at (1,0) {}; 
	\cvx[color=black!25](a2) at (2,0) {}; 
	\cvx[color=black!25](a3) at (3,0) {}; 
	\cvx[color=black!25](a4) at (4,0) {}; 
	\cvx[color=black](a5) at (5,0) {};
	\cvx[color=black](a6) at (6,0) {};
	\draw[color=black!25] (a1)--(a5);
	\draw[color=black!25] (a1) .. controls (1.25,.4) and (2.75,.4) .. (a3);	
    \draw[color=black!25] (a1) .. controls (1.25,0.8) and (3.75,0.8) .. (a4);
    \draw[color=black!25] (a1) .. controls (1.25,1.25) and (3.75,1.25) .. (a4);    
	\draw[color=black!25] (a1) .. controls (1.25, 2.25) and (4.75,2.25) .. (a5);
	\draw[color=black!25] (a4) to[out=-50,in=230] (a6);
	\draw[color=black!25] (a3) to[out=-50,in=230] (a6);
	\draw[color=black!25] (a2) to[out=-50,in=230] (a6);
	\draw[color=black, thick] (a1) .. controls (1.25, 1.75) and (4.75,1.75) .. (a5);
    \draw[color=black, thick] (a5)--(a6);
\end{scope}
\begin{scope}[xshift=210, yshift=-35, scale=.33]
    \node[]  at (5.5,1.5)  {\tiny \textcolor{cyan}{$9$}};
    \cvx[color=black](a1) at (1,0) {}; 
	\cvx[color=black!25](a2) at (2,0) {}; 
	\cvx[color=black!25](a3) at (3,0) {}; 
	\cvx[color=black!25](a4) at (4,0) {}; 
	\cvx[color=black](a5) at (5,0) {};
	\cvx[color=black](a6) at (6,0) {};
	\draw[color=black!25] (a1)--(a5);
	\draw[color=black!25] (a1) .. controls (1.25,.4) and (2.75,.4) .. (a3);	
    \draw[color=black!25] (a1) .. controls (1.25,0.8) and (3.75,0.8) .. (a4);
    \draw[color=black!25] (a1) .. controls (1.25,1.25) and (3.75,1.25) .. (a4);    
	\draw[color=black!25] (a1) .. controls (1.25, 1.75) and (4.75,1.75) .. (a5);
	
	\draw[color=black!25] (a4) to[out=-50,in=230] (a6);
	\draw[color=black!25] (a3) to[out=-50,in=230] (a6);
	\draw[color=black!25] (a2) to[out=-50,in=230] (a6);
    \draw[color=black, thick] (a1) .. controls (1.25, 2.25) and (4.75,2.25) .. (a5);
    \draw[color=black, thick] (a5)--(a6);
\end{scope}

\begin{scope}[xshift=330, yshift=-35, scale=0.5]
    \draw[fill, color=gray!33] (0,0) rectangle (1,1); 
    \draw[fill, color=gray!33] (1,1) rectangle (3,2);
    \draw[fill, color=gray!33] (3,2) rectangle (5,3);
    \draw[fill, color=gray!33] (1,0) rectangle (2,1);
    \draw[fill, color=gray!33] (3,1) rectangle (4,2);  
    \draw[very thin, color=gray!100] (0,0) grid (5,3);  
   	\draw[dashed, thick, color=gray!10] (0,0)--(5,3); 
		
    \draw[very thick, red] (0,0) -- (0,2) -- (2,2) -- (2,3) -- (5,3);    
    \node[]  at (-.3,0)  {\tiny \textcolor{cyan}{$1$}};
    \node[]  at (-.3,1)  {\tiny \textcolor{cyan}{$2$}};
    \node[]  at (-.3,2)  {\tiny \textcolor{cyan}{$3$}};
    \node[]  at (.8,2.3)  {\tiny \textcolor{cyan}{$4$}};
    \node[]  at (1.8,2.3)  {\tiny \textcolor{cyan}{$5$}};
    \node[]  at (2,3.3)  {\tiny \textcolor{cyan}{$6$}};
    \node[]  at (3,3.3)  {\tiny \textcolor{cyan}{$7$}};
    \node[]  at (4,3.3)  {\tiny \textcolor{cyan}{$8$}};
    \node[]  at (5,3.3)  {\tiny \textcolor{cyan}{$9$}};
\end{scope}
\end{tikzpicture}
\caption{A maximal clique of routes (left) representing a simplex in the planar-framed triangulation of $\calF_{\car(\nu)}$ for $\nu=NENE^2NE^2$.  The extension $\Psi$ of the bijection $\psi$ of Lemma~\ref{lem.bijectionlatticepoints} sends this clique to the $\nu$-Dyck path on the right.
}
\label{fig.planar_clique}
\end{figure}

Given two lattice points $(x_1,y_1)$ and $(x_2,y_2)$ with $x_1<x_2$ are said to be {\em incompatible} if $y_1 > y_2$. Otherwise, any other pair of lattice points are said to be {\em compatible}. Maximal sets of compatible lattice points lying above $\nu$ determine a unique $\nu$-Dyck path.
\begin{lemma}\label{lem.planarcoherence}
Let $\prec$ denote the planar framing, and let $\psi$ be the bijection in Lemma~\ref{lem.bijectionlatticepoints}.
Two routes $R_1$ and $R_2$ in the framed graph $(\car(\nu), \prec)$ are coherent if and only if $\psi(R_1)$ and $\psi(R_2)$ are compatible.
\end{lemma}
\begin{proof}
A result of M\'esz\'aros, Morales and Striker~\cite[Lemma 6.5]{MMS19} states that two routes in a planar framing of a graph $G$ are coherent if and only if they are non-crossing in $G$.  
Let $R_1$ and $R_2$ be two routes in $\car(\nu)$ and $\psi(R_1)=(j_1,\ell_1)$, $\psi(R_2)=(j_2,\ell_2)$.

Suppose $R_1$ and $R_2$ are coherent with $\ell_1 < \ell_2$. 
Then the fact that $R_1$ and $R_2$ are non-crossing implies that $j_1 \leq j_2$, hence $(j_1,\ell_1)$ and $(j_2,\ell_2)$ are compatible.
On the other hand suppose $R_1$ and $R_2$ are not coherent.  
Without loss of generality, we may assume that $\ell_1<\ell_2$ (for otherwise, if $\ell_1=\ell_2$ then the routes are coherent).
Let $k$ be the smallest vertex at which the routes cross.  
Then $j_2\leq j_1 \leq k$, and hence $(j_1,\ell_1)$ and $(j_2,\ell_2)$ are not compatible.
\end{proof}

\rootthm*
\begin{proof}
By Lemma~\ref{lem.planarcoherence}, the bijection $\psi$ in Lemma~\ref{lem.bijectionlatticepoints} extends to a bijection $\Psi$ from maximal cliques of routes in the planar-framed $\car(\nu)$ to maximal sets of compatible lattice points lying above $\nu$, which are $\nu$-Dyck paths.
Two simplices in a DKK triangulation of a flow polytope are adjacent if and only if they differ by a single vertex. 
Under the bijection $\Psi$, two simplices are adjacent if and only if their corresponding $\nu$-Dyck paths $\pi_1$ and $\pi_2$ differ by a single lattice point. 
Let $(x_1,y_1) \in \pi_1$ and $(x_2,y_2)\in \pi_2$ be the lattice points which are not contained in both paths. Assume without loss of generality that $x_1<x_2$. Since these lattice points are not compatible, we must have $y_1 > y_2$. Thus $(x_1,y_1)$ is in the top left corner of the single square determined by $(x_1,y_1)$ and $(x_2,y_2)$, while $(x_2,y_2)$ is in the bottom left. In other words, $\pi_1$ and $\pi_2$ differ by a transposition of a consecutive $NE$ pair, which is precisely the description of the covering relation in the principal order ideal $I(\nu)$.
\end{proof}

\begin{example}
Let $\nu= NENE^2NE^2$.  
The bijection $\Psi$ between cliques of routes of $\car(\nu)$ and $\nu$-Dyck paths is shown in Figure~\ref{fig.planar_clique}.
The dual graph of the planar-framed triangulation of $\calF_{\car(\nu)}$ is shown in Figure~\ref{fig.35} on the right. 
\end{example}

\begin{figure}[ht!]
\centering
\begin{tikzpicture}[scale=.75]
\tikzstyle{vertex}=[circle, fill=ForestGreen, inner sep=0pt, minimum size=5pt]
\begin{scope}[scale=.3, xshift=0, yshift=38]
	\draw[thick, color=ForestGreen]    (0,0) -- (10,8) -- (6,13) -- (-4,7) -- (-4,3) -- (0,0);		\vertex[color=ForestGreen] at (0,0)  {};
	\vertex[color=ForestGreen]  at (5,4) {};
	\vertex[color=ForestGreen]  at (10,8) {};
	\vertex[color=ForestGreen]  at (6,13) {};
	\vertex[color=ForestGreen]  at (.5,9.7) {};
	\vertex[color=ForestGreen]  at (-4,7) {};
	\vertex[color=ForestGreen]  at (-4,3) {};
	\draw[thick, color=ForestGreen] (5,4) -- (.5,9.7);
\end{scope}	
\begin{scope}[xshift=10, yshift=-7, scale=0.22]
	\draw[fill, color=gray!33] (0,0) rectangle (2,1);
	\draw[fill, color=gray!33] (1,1) rectangle (4,2);
	\draw[fill, color=gray!33] (3,2) rectangle (5,3);
	\draw[very thin, color=gray!100] (0,0) grid (5,3);
	\draw[thick, color=red] (0,0)--(0,1)--(1,1)--(1,2)--(3,2)--(3,2)--(3,3)--(5,3);
\end{scope}
\begin{scope}[xshift=-77, yshift=20, scale=0.22]
	\draw[fill, color=gray!33] (0,0) rectangle (2,1);
	\draw[fill, color=gray!33] (1,1) rectangle (4,2);
	\draw[fill, color=gray!33] (3,2) rectangle (5,3);
	\draw[very thin, color=gray!100] (0,0) grid (5,3);
	\draw[thick, color=red] (0,0)--(0,2)--(3,2)--(3,3)--(5,3);
\end{scope}
\begin{scope}[xshift=-77, yshift=60, scale=0.22]
	\draw[fill, color=gray!33] (0,0) rectangle (2,1);
	\draw[fill, color=gray!33] (1,1) rectangle (4,2);
	\draw[fill, color=gray!33] (3,2) rectangle (5,3);
	\draw[very thin, color=gray!100] (0,0) grid (5,3);
	\draw[thick, color=red] (0,0)--(0,2)--(2,2)--(2,3)--(5,3);
\end{scope}
\begin{scope}[xshift=-35, yshift=93, scale=0.22]
	\draw[fill, color=gray!33] (0,0) rectangle (2,1);
	\draw[fill, color=gray!33] (1,1) rectangle (4,2);
	\draw[fill, color=gray!33] (3,2) rectangle (5,3);
	\draw[very thin, color=gray!100] (0,0) grid (5,3);
	\draw[thick, color=red] (0,0)--(0,2)--(1,2)--(1,3)--(5,3);
\end{scope}
\begin{scope}[xshift=53, yshift=28, scale=0.22]
	\draw[fill, color=gray!33] (0,0) rectangle (2,1);
	\draw[fill, color=gray!33] (1,1) rectangle (4,2);
	\draw[fill, color=gray!33] (3,2) rectangle (5,3);
	\draw[very thin, color=gray!100] (0,0) grid (5,3);
	\draw[thick, color=red] (0,0)--(0,1)--(1,1)--(1,2)--(2,2)--(2,3)--(5,3);   
\end{scope}
\begin{scope}[xshift=93, yshift=63, scale=0.22]
	\draw[fill, color=gray!33] (0,0) rectangle (2,1);
	\draw[fill, color=gray!33] (1,1) rectangle (4,2);
	\draw[fill, color=gray!33] (3,2) rectangle (5,3);
	\draw[very thin, color=gray!100] (0,0) grid (5,3);
	\draw[thick, color=red] (0,0)--(0,1)--(1,1)--(1,3)--(5,3);    
\end{scope}
\begin{scope}[xshift=12, yshift=123, scale=0.22]
	\draw[fill, color=gray!33] (0,0) rectangle (2,1);
	\draw[fill, color=gray!33] (1,1) rectangle (4,2);
	\draw[fill, color=gray!33] (3,2) rectangle (5,3);
	\draw[very thin, color=gray!100] (0,0) grid (5,3);
	\draw[thick, color=red] (0,0)--(0,3)--(5,3);
\end{scope}

\begin{scope}[scale=1.2, xshift=230, yshift=0]	
	\vertex[color=ForestGreen] at (0,-.2)  {};
	\vertex[color=ForestGreen]  at (0,1) {};
	\vertex[color=ForestGreen]  at (1,2) {};
	\vertex[color=ForestGreen]  at (2,3) {};
	\vertex[color=ForestGreen]  at (-1,2) {};
	\vertex[color=ForestGreen]  at (0,3) {};
	\vertex[color=ForestGreen]  at (1,4) {};
	\draw[thick, color=ForestGreen] (0,-.2) -- (0,1);
	\draw[thick, color=ForestGreen] (0,1) -- (2,3) -- (1,4);
	\draw[thick, color=ForestGreen] (0,1) -- (-1,2) -- (1,4);
	\draw[thick, color=ForestGreen] (1,2) -- (0,3);
\end{scope}	
\begin{scope}[scale=0.22, xshift=1190, yshift=570]
	\draw[fill, color=gray!33] (0,0) rectangle (2,1);
	\draw[fill, color=gray!33] (1,1) rectangle (4,2);
	\draw[fill, color=gray!33] (3,2) rectangle (5,3);
	\draw[very thin, color=gray!100] (0,0) grid (5,3);
	\draw[thick, color=red] (0,0)--(0,1)--(1,1)--(1,2)--(3,2)--(3,2)--(3,3)--(5,3);
\end{scope}
\begin{scope}[scale=0.22, xshift=1050, yshift=430]
	\draw[fill, color=gray!33] (0,0) rectangle (2,1);
	\draw[fill, color=gray!33] (1,1) rectangle (4,2);
	\draw[fill, color=gray!33] (3,2) rectangle (5,3);
	\draw[very thin, color=gray!100] (0,0) grid (5,3);
	\draw[thick, color=red] (0,0)--(0,1)--(1,1)--(1,2)--(2,2)--(2,3)--(5,3);
\end{scope}
\begin{scope}[scale=0.22, xshift=1600, yshift=420]
	\draw[fill, color=gray!33] (0,0) rectangle (2,1);
	\draw[fill, color=gray!33] (1,1) rectangle (4,2);
	\draw[fill, color=gray!33] (3,2) rectangle (5,3);
	\draw[very thin, color=gray!100] (0,0) grid (5,3);
	\draw[thick, color=red] (0,0)--(0,2)--(3,2)--(3,2)--(3,3)--(5,3);
\end{scope}
\begin{scope}[scale=0.22, xshift=910, yshift=270]
	\draw[fill, color=gray!33] (0,0) rectangle (2,1);
	\draw[fill, color=gray!33] (1,1) rectangle (4,2);
	\draw[fill, color=gray!33] (3,2) rectangle (5,3);
	\draw[very thin, color=gray!100] (0,0) grid (5,3);
	\draw[thick, color=red] (0,0)--(0,1)--(1,1)--(1,3)--(5,3);
\end{scope}
\begin{scope}[scale=0.22, xshift=1460, yshift=250]
	\draw[fill, color=gray!33] (0,0) rectangle (2,1);
	\draw[fill, color=gray!33] (1,1) rectangle (4,2);
	\draw[fill, color=gray!33] (3,2) rectangle (5,3);
	\draw[very thin, color=gray!100] (0,0) grid (5,3);
	\draw[thick, color=red] (0,0)--(0,2)--(2,2)--(2,3)--(5,3);
\end{scope}
\begin{scope}[scale=0.22, xshift=1295, yshift=90]
	\draw[fill, color=gray!33] (0,0) rectangle (2,1);
	\draw[fill, color=gray!33] (1,1) rectangle (4,2);
	\draw[fill, color=gray!33] (3,2) rectangle (5,3);
	\draw[very thin, color=gray!100] (0,0) grid (5,3);
	\draw[thick, color=red] (0,0)--(0,2)--(1,2)--(1,3)--(5,3);
\end{scope}
\begin{scope}[scale=0.22, xshift=1070, yshift=-40]
	\draw[fill, color=gray!33] (0,0) rectangle (2,1);
	\draw[fill, color=gray!33] (1,1) rectangle (4,2);
	\draw[fill, color=gray!33] (3,2) rectangle (5,3);
	\draw[very thin, color=gray!100] (0,0) grid (5,3);
	\draw[thick, color=red] (0,0)--(0,3)--(5,3);
\end{scope}

\end{tikzpicture}
\caption{The $\nu$-Tamari lattice (left) and the Hasse diagram of the order ideal $I(\nu)\subseteq Y$ (right) for $\nu=NENE^2NE^2$.
These are the dual graphs of the length-framed and planar-framed triangulations of $\calF_{\car(\nu)}$.  }
\label{fig.35}
\end{figure}
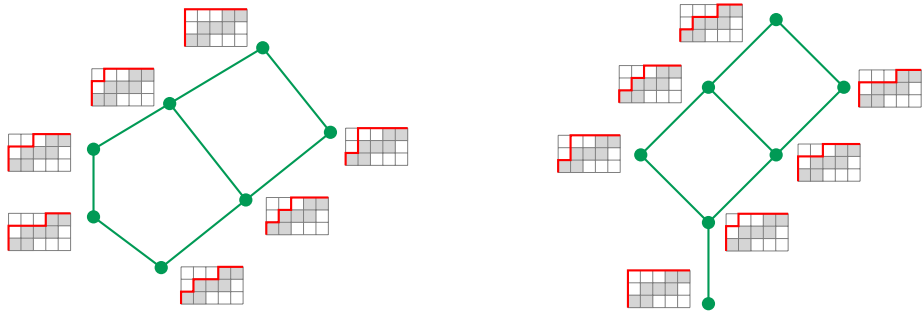

\subsection{Comparing the length-framed and planar-framed triangulations}
A special case when the dual structure of the length-framed and planar-framed triangulations of $\calF_{\car(\nu)}$ are the same is given by the following proposition.

\begin{proposition}
\label{prop.whenTheTriangulationsAreEquivalent}
When $\nu = E^aN^b$, so that the set of $\nu$-Dyck paths is the set of all lattice paths from $(0,0)$ to $(b,a)$, the length-framed triangulation and the planar-framed triangulation of $\calF_{\car(\nu)}$ have the same dual structure.
\end{proposition}
\begin{proof}
We use the $\nu$-Dyck path description (see Section~\ref{subsec:nu-Tamari}) of the $\nu$-Tamari lattice $\mathrm{Tam}(\nu)$ in this proof.
Let $\mu$ be a $\nu$-Dyck path.
For any valley point $p$ of $\mu$, the next lattice point $q$ in $\mu$ with $\horiz_\nu(p)=\horiz_\nu(q)$ is the next lattice point after $p$.
This is because the horizontal distance of any of the lattice points in a run of consecutive $N$ steps is the same when $\nu=E^aN^b$.
Performing a rotation on $\mu$ at the valley point $p$ to obtain the $\nu$-Dyck path $\mu'$ is then the same as exchanging the $EN$ pair centered at $p$ with an $NE$ pair of steps in $\mu$.
Thus $\mu <_\nu \mu'$ is a covering relation in the lattice $\mathrm{Tam}(\nu)$ if and only if it is a covering relation in the dual order ideal $I(\nu)^*$.
Therefore, $\mathrm{Tam}(\nu)= I(\nu)^*$.
Lastly, since the partition $\lambda(\nu)=(b^a)$ is rectangular, then $I(\nu)$ is self-dual.
Therefore, $\mathrm{Tam}(\nu)$ and $I(\nu)$ are isomorphic.
\end{proof}

\subsection{A connection with order polytopes} \label{sec.orderpolytopes}
In this subsection, $G$ is a planar graph on the vertex set $[n+1]$ with a planar embedding satisfying the properties outlined in Definition~\ref{def.planarframing}.  
We further assume that the in-degree and out-degree of each vertex $i$ for $i=2,\ldots, n$ is at least one.
A result of M\'esz\'aros, Morales and Striker~\cite[Theorem 3.11]{MMS19} states that for such a graph $G$, the flow polytope $\calF_G$ is integrally equivalent to the order polytope $\calO(P_G)$, where $P_G$ is a poset that is induced by the bounded faces of the planar embedding of $G$.

In this section, we explain how our results for flow polytopes on the caracol graphs $\car(\nu)$ lead to analogous results for a certain class of order polytopes $\calO(Q_\nu)$.
We give a brief background of known results relating order polytopes and flow polytopes following the exposition of~\cite{MMS19}, and explain their implications when applied to $\car(\nu)$.

Let $(P,\leq_P)$ be a finite poset with elements $\{p_1,\ldots, p_d\}$.
The {\em order polytope} of $P$ is the set of points
$$\calO(P) = \left\{ (x_{p_1},\ldots, x_{p_d})\in [0,1]^d \mid x_{p_i} \leq x_{p_j} \hbox{ if } p_i\leq_P p_j \right\}.$$
Given a \emph{linear extension} $\sigma:P\rightarrow [d]$ of the poset $P$, i.e. an order preserving bijection with $[d]$ endowed with its natural order, define the simplex
$$\Delta_\sigma=\left\{ (x_{p_1},\ldots, x_{p_d}) \in [0,1]^d \mid x_{\sigma^{-1}(1)}\leq  \cdots \leq x_{\sigma^{-1}(d)} \right\}.$$
The {\em canonical triangulation} of $\calO(P)$, first defined by Stanley~\cite{Stanley86}, is the set of simplices 
$$\left\{ \Delta_\sigma \mid \sigma \hbox{ is a linear extension of } P\right\}.$$
Thus the normalized volume of $\calO(P)$ is the number of linear extensions of $P$.

For a planar graph $G$ with a fixed embedding in the plane, the {\em truncated dual graph} $G^*$ of $G$ is the dual graph whose vertices correspond to the bounded faces of $G$.  
Viewing $G^*$ as embedded on the plane also, then the orientation on the edges of $G$ induces an orientation on the edges of $G^*$.
The graph $G^*$ then induces the Hasse diagram of a poset that is denoted by $P_G$.
See Figure~\ref{fig.orderpolytope}.

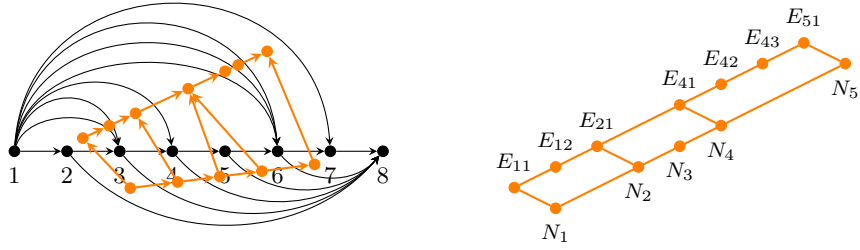
\begin{figure}[ht!]
\centering
\begin{tikzpicture}
\begin{scope}[xshift=0, scale=.7]
	\vertex[fill,label=below:\scriptsize{$1$}](a1) at (1,0) {};
	\vertex[fill,label=below:\scriptsize{$2$}](a2) at (2,0) {};
	\vertex[fill,label=below:\scriptsize{$3$}](a3) at (3,0) {};
	\vertex[fill,label=below:\scriptsize{$4$}](a4) at (4,0) {};
	\vertex[fill,label=below:\scriptsize{$5$}](a5) at (5,0) {};
	\vertex[fill,label=below:\scriptsize{$6$}](a6) at (6,0) {};
	\vertex[fill,label=below:\scriptsize{$7$}](a7) at (7,0) {};
	\vertex[fill,label=below:\scriptsize{$8$}](a8) at (8,0) {};

	\draw[-stealth] (a1)--(a2);
	\draw[-stealth] (a2)--(a3);
	\draw[-stealth] (a3)--(a4);
	\draw[-stealth] (a4)--(a5);
	\draw[-stealth] (a5)--(a6);
	\draw[-stealth] (a6)--(a7);
	\draw[-stealth] (a7)--(a8);
	\draw[-stealth] (a1) .. controls (1.25, 1.3) and (2.75, 1.3) .. (a3);	
	\draw[-stealth] (a1) .. controls (1.25, .8) and (2.75, .8) .. (a3);
	\draw[-stealth] (a1) .. controls (1.25, 1.7) and (3.75, 1.7) .. (a4);
	\draw[-stealth] (a1) .. controls (1.25, 2.2) and (5.75, 2.2) .. (a6);	
	\draw[-stealth] (a1) .. controls (1.25, 2.7) and (5.75, 2.7) .. (a6);
	\draw[-stealth] (a1) .. controls (1.25, 3.2) and (5.75, 3.2) .. (a6);
	\draw[-stealth] (a1) .. controls (1.25, 3.7) and (6.75, 3.7) .. (a7);
	\draw[-stealth] (a6) to[out=-50,in=230] (a8);
	\draw[-stealth] (a5) to[out=-50,in=230] (a8);
	\draw[-stealth] (a4) to[out=-50,in=230] (a8);
	\draw[-stealth] (a3) to[out=-50,in=230] (a8);
	\draw[-stealth] (a2) to[out=-50,in=230] (a8);
	
	\pvx[color=orange, minimum size =4pt] (e1) at (2.3,.25){};
	\pvx[color=orange, minimum size =4pt] (e2) at (2.8,.485){};
	\pvx[color=orange, minimum size =4pt] (e3) at (3.3,.721){};
	\pvx[color=orange, minimum size =4pt] (e4) at (4.3,1.19){};	
	\pvx[color=orange, minimum size =4pt] (e5) at (5,1.52){};
	\pvx[color=orange, minimum size =4pt] (e6) at (5.25,1.64){};
	\pvx[color=orange, minimum size =4pt] (e7) at (5.8,1.9){};
	\pvx[color=orange, minimum size =4pt] (n1) at (3.2,-.7){};
	\pvx[color=orange, minimum size =4pt] (n2) at (4.1,-.584){};
	\pvx[color=orange, minimum size =4pt] (n3) at (4.9,-.481){};
	\pvx[color=orange, minimum size =4pt] (n4) at (5.7,-.378){};
	\pvx[color=orange, minimum size =4pt] (n5) at (6.7,-.25){};
	\draw[-stealth, thick, color=orange] (e1)--(e2) {};
	\draw[-stealth, thick, color=orange] (e2)--(e3) {};
	\draw[-stealth, thick, color=orange] (e3)--(e4) {};
	\draw[-stealth, thick, color=orange] (e4)--(e5) {};
	\draw[-stealth, thick, color=orange] (e5)--(e6) {};
	\draw[-stealth, thick, color=orange] (e6)--(e7) {};
	\draw[-stealth, thick, color=orange] (n1)--(n2) {};
	\draw[-stealth, thick, color=orange] (n2)--(n3) {};
	\draw[-stealth, thick, color=orange] (n3)--(n4) {};
	\draw[-stealth, thick, color=orange] (n4)--(n5) {};
	\draw[-stealth, thick, color=orange] (n1)--(e1) {};
	\draw[-stealth, thick, color=orange] (n2)--(e3) {};
	\draw[-stealth, thick, color=orange] (n3)--(e4) {};
	\draw[-stealth, thick, color=orange] (n4)--(e4) {};	
	\draw[-stealth, thick, color=orange] (n5)--(e7) {};	
\end{scope}
\begin{scope}[scale=0.55, xshift=380, yshift=-25]
	\vertex[fill, label=above:{\tiny$E_{11}$}, color=orange](e11) at (0,0) {};
	\vertex[fill, label=above:{\tiny$E_{12}$},color=orange](e12) at (1,.5) {};
	\vertex[fill, label=above:{\tiny$E_{21}$},color=orange](e21) at (2,1) {};
	\vertex[fill, label=above:{\tiny$E_{41}$},color=orange](e41) at (4,2) {};
	\vertex[fill, label=above:{\tiny$E_{42}$},color=orange](e42) at (5,2.5) {};
	\vertex[fill, label=above:{\tiny$E_{43}$},color=orange](e43) at (6,3) {};
	\vertex[fill, label=above:{\tiny$E_{51}$},color=orange](e51) at (7,3.5) {};	
	\vertex[fill, label=below:{\tiny$N_1$}, color=orange](n1) at (1,-.5) {};	
	\vertex[fill, label=below:{\tiny$N_2$},color=orange](n2) at (3,0.5) {};
	\vertex[fill, label=below:{\tiny$N_3$},color=orange](n3) at (4,1) {};
	\vertex[fill, label=below:{\tiny$N_4$},color=orange](n4) at (5,1.5) {};
	\vertex[fill, label=below:{\tiny$N_5$},color=orange](n5) at (8,3) {};	
	\draw[thick, color=orange] (e11)--(e51);
	\draw[thick, color=orange] (n1)--(n5);
	\draw[thick, color=orange] (n1)--(e11);
	\draw[thick, color=orange] (n2)--(e21);
	\draw[thick, color=orange] (n4)--(e41);
	\draw[thick, color=orange] (n5)--(e51);
\end{scope}
\end{tikzpicture}
\caption{Let $\nu=NE^2NENNE^3NE$.  
The truncated dual graph of $\car(\nu)$ is the graph whose vertices are the bounded faces of the embedded $\car(\nu)$ (left).  The Hasse diagram of the poset $Q_\nu=P_{\car(\nu)}$ is induced by the truncated dual (right).
} 
\label{fig.orderpolytope}
\end{figure}

In the case $\nu= NE^{\nu_1}\cdots NE^{\nu_a}$ is a lattice path from $(0,0)$ to $(b,a)$, then the poset $P_{\car(\nu)}$ has $a+b$ elements (corresponding to the bounded faces of the embedded $\car(\nu)$), and is constructed in the following way.
There are $a$ elements labeled $N_1,\ldots, N_a$ corresponding to the bounded faces that are uniquely determined by each pair of edges of the form $(j+2,n+1)$ and $(j+3,n+1)$ for $j=1,\dots, a$.  
For each $j=1,\ldots,a$, if there are $k$ edges of the form $(1,j+2)$ in $\car(\nu)$, then there are $k$ elements labeled $E_{j,1},\ldots, E_{j,k}$.  
Since there are $b$ edges of the form $(1,j+2)$ for $j=1,\dots, a$, then there are $b$ elements labeled with an $E$.
The relations in $P_{\car(\nu)}$ are $N_i < N_j$ for $1\leq i<j\leq a$, $E_{j,k} < E_{\ell,n}$ if $(j,k)$ appears before $(\ell,n)$ in lexicographic order, and $N_i < E_{j,k}$ if $i\leq j$. 

Note that if we think of $N_i$ as $N_{i,0}$, then listing the subscripts of these $N$s and $E$s lexicographically recovers the $\nu$-Dyck path.
With this observation, it means that we can define a class of posets $Q_\nu$ (equal to $P_{\car(\nu)}$) indexed by lattice paths $\nu$ without any reference to flow polytopes.
See Figure~\ref{fig.orderpolytope}.

It was first observed by Postnikov (also see M\'esz\'aros, Morales and Striker~\cite[Theorem 1.3]{MMS19}) that the canonical triangulation of $\calO(P_G)$ is the same as the planar-framed DKK triangulation of $\calF_G$ up to an integral equivalence.
Combined with Theorems~\ref{thm.associahedralTriangulation}
and~\ref{thm.roottriangulation}, we have the following two corollaries.

\begin{corollary}
The canonical triangulation of the order polytope $\calO(Q_\nu)$ has dual graph which is the Hasse diagram of the principal order ideal $I(\nu)$ in Young's lattice.
\end{corollary}

\begin{corollary}
The order polytope $\calO(Q_\nu)$ has a regular unimodular triangulation whose dual graph is the $\nu$-Tamari lattice $\mathrm{Tam}(\nu)$. 
\end{corollary}

For a poset $P$, let $J(P)$ denote the lattice of order ideals of $P$ ordered by inclusion. 
From~\cite[Section 5]{Stanley86}, the maximal chains of $J(P)$ are in bijection with the simplices in the canonical triangulation of $\calO(P)$.
This gives another perspective on the direct relationship between simplices in the planar-framed triangulation of $\calF_{\car(\nu)}$, maximal cliques in the flow polytope $\calF_{\car(\nu)}$, $\nu$-Dyck paths, maximal chains in $J(Q_\nu)$, and simplices in the canonical triangulation of $\calO(Q_\nu)$.
In this case, the Hasse diagram of $J(Q_\nu)$ can be obtained by taking the lattice on the points $\calL_\nu$ which lie above $\nu$, and rotating it counterclockwise by $45$ degrees.

Having obtained results for order polytopes via methods for flow polytopes, we now end this section with a result for flow polytopes via methods for order polytopes.
\begin{corollary}
Let $\nu=NE^{\nu_1}\cdots NE^{\nu_a}$ be a lattice path from $(0,0)$ to $(b,a)$. 
Let $\mathrm{peak}(\nu)$ denote the number of consecutive $NE$ pairs in $\nu$.
The number of facets of $\calF_{\car(\nu)}$ is $a+b+\mathrm{peak}(\nu)$.
\end{corollary}
\begin{proof}
A result of Stanley~\cite{Stanley86} states that the facets of an order polytope $\calO(P)$ correspond to the covering relations of the poset $\widehat{P}=P \cup \{\hat0,\hat1\}$.
The poset $\widehat{P}_{\car(\nu)}$ consists of a lower chain $\hat{0} \lessdot N_1 \lessdot \cdots \lessdot N_a$ with $a$ covering relations, an upper chain $E_{11} \lessdot \cdots \lessdot \hat{1}$ with $b$ covering relations, and additional covering relations of the form $N_i \lessdot E_{i,k}$ if and only if $N_iE_{i,k}$ is a consecutive $NE$ pair (in other words, a peak) in $\nu$.  
The result then follows since $\calO(P_{\car(\nu)})$ is integrally equivalent to $\calF_{\car(\nu)}$.
\end{proof}

\section{The \texorpdfstring{$h^*$-}-vector of the \texorpdfstring{$\nu$-}-caracol flow polytope}\label{sec.hstar}

The $h^*$-vector of a lattice polytope coincides with the $h$-vector of any of its unimodular triangulations \cite[Theorem 10.3]{BS07}, so we will compute the $h$-vector of the planar-framed triangulation of $\calF_{\car(\nu)}$. This extends a result of M\'esz\'aros~\cite{M16} for the classical case when $\nu=(1^n)$.

We begin by recalling some relevant definitions from~\cite{Ziegler07}. 
Given a simplicial complex, a {\em shelling} is an ordering $F_1,...,F_s$ of its facets such that for every $i<j$ there is some $k<j$ such that the intersection $F_i \cap F_j \subseteq F_k \cap F_j$, and $F_k \cap F_j$ is a facet of $F_j$. 
A simplicial complex is said to be {\em shellable} if it admits a shelling. 
The $h$-vectors of shellable simplicial complexes have non-negative entries which can be computed combinatorially from the shelling order as follows. 
For a fixed shelling order $F_1,...,F_s$ define the restriction $R_j$ of the facet $F_j$ as the set $R_j := \{v\in F_j : v \text{ is a vertex in }F_j \text{ and } F_j \setminus v \subseteq F_i \text{ for some } 1 \leq i < j\}$. 
Then the $i$-entry of the $h$-vector is given by $h_i = |\{j:|R_j| = i, 1 \leq j \leq s\}|$.

\begin{lemma}
Let $\mathcal{C}$ be the planar-framed triangulation of $\mathcal{F}_{\car(\nu)}$ interpreted as a simplicial complex. 
Any linear extension of $I(\nu)$ gives a shelling order of $\mathcal{C}$.
\end{lemma}

\begin{proof}
By Theorem~\ref{thm.roottriangulation} we can give the dual graph of $\calC$ the structure of $I(\nu)$, identifying each facet in $\calC$ with the associated $\nu$-Dyck path in $I(\nu)$. 
For a linear extension $L$ of $I(\nu)$, we can order the facets $F_1,...,F_s$ of $\mathcal{C}$ according to $L$. 
Let $\pi_i$ and $\pi_j$ be two $\nu$-Dyck paths in $L$, with $i<j$. 
Let $\pi_{s_1}$ be the minimal $\nu$-Dyck path that covers both $\pi_i$ and $\pi_j$, i.e. $\pi_{s_1} = \pi_i \vee \pi_j$ in $I(\nu)$. 
Now $\pi_{s_1}$ contains the lattice points in $\pi_i \cap \pi_j$, and so $F_i\cap F_j \subseteq F_{s_1}$. 
It is clear that there exists a sequence of $\nu$-Dyck paths $\pi_{s_1}, \pi_{s_2},...,\pi_j$ such that each path contains the lattice points $\pi_i\cap \pi_j$, and each path is formed from the previous path by replacing a consecutive $NE$ pair with $EN$. 
Given such a sequence of paths, let $\pi_k$ be the second to last path in the sequence. 
Replacing a consecutive $NE$ pair with $EN$ in $\pi_k$ yield $\pi_j$. 
Now $k<j$, and $F_i \cap F_j$ is contained in every facet $F_{s_\ell}$ for $1\leq \ell \leq k$. 
In particular, $F_i \cap F_j \subseteq F_k \cap F_j$. 
Furthermore, $\pi_k$ and $\pi_j$ differ by a single lattice point, so $F_k\cap F_j$ is a facet of $F_j$.
\end{proof}

Let $\nu$ be a lattice path from $(0,0)$ to $(b,a)$.
For $i=0,\ldots,a$, the {\em $\nu$-Narayana number} $\Nar_\nu(i)$ is the number of $\nu$-Dyck paths with $i$ valleys (recall that a valley is a consecutive $EN$ pair). 
The {\em $\nu$-Narayana polynomial} is 
$N_\nu(x) = \sum_{i\geq 0} \mathrm{Nar}_\nu(i)x^i$.
For more on these definitions, see \cite{BY} or \cite{CPS19}, for example.

\hstarthm*
\begin{proof}
It suffices to find the $h$-vector of the planar-framed triangulation of $\calF_{\car(\nu)}$. 
Any linear extension of the order ideal $I(\nu)$ gives a shelling order of the planar-framed triangulation of $\calF_{\car(\nu)}$.
The $i$-th entry of the $h$-vector is 
\begin{align*}
    h_i &= |\{j:|R_j| = i, 1 \leq j \leq s\}| \\    
    &= |\{\text{paths in $I(\nu)$ that cover exactly } i \text{ paths} \}| \\
    & = |\{\text{paths with exactly } i \text{ valleys}\}|\\
    &= \mathrm{Nar}_\nu(i).\qedhere
\end{align*}
\end{proof}
A different proof of Theorem~\ref{thm.hstar} can be obtained by computing the $h$-vector of the length-framed triangulation of $\calF_{\car(\nu)}$, which by Corollary \ref{cor.geometric_realization} is combinatorially equivalent to the $(I,\overline{J})$-Tamari complex with the pair $(I,\overline{J})$ associated to $\nu$, which we also call the \emph{$(I,\overline{J})$-Tamari complex}. In \cite[Lemma 4.5]{CPS19} a shelling order on facets of this complex was used to show that the $h$-vector of the $(I,\overline{J})$-Tamari complex is given by the $\nu$-Narayana numbers. 
Since any lattice unimodular triangulation can be used to calculate the $h^*$-vector of $\calF_{\car(\nu)}$, Theorem \ref{thm.hstar} provides a new proof that the $h$-vector of the $(I,\overline{J})$-Tamari complex is given by the $\nu$-Narayana numbers. 

\section*{Acknowledgments}
The second and fourth authors are extremely grateful to AIM and the SQuaRE group ``Computing volumes and lattice points of flow polytopes'' as some of the ideas of this work came from discussions within the group. In particular, we want to thank Alejandro Morales for the many enlightening discussions and explanations on triangulations of flow polytopes. Martha Yip is partially supported by Simons Collaboration Grant 429920.

\printbibliography

\end{document}